\documentclass[reqno,11pt,letterpaper]{amsart}

\usepackage{amsmath,amssymb,amsthm,mathrsfs,bbm}
\usepackage{accents}
\usepackage{arydshln}
\usepackage{upgreek}
\usepackage{slashed}
\usepackage{xifthen}
\usepackage{graphicx}
\usepackage{longtable}
\usepackage[inline]{enumitem}
\usepackage[normalem]{ulem}

\usepackage[T1]{fontenc}
\usepackage[utf8]{inputenc}
\usepackage{mathtools} 
\usepackage{amsthm}
\usepackage{amssymb} 
\usepackage{amsfonts, dsfont} 
\usepackage{graphicx} 
\usepackage{color} 
\usepackage{subcaption}
\usepackage{fancybox} 
\usepackage{float}  
\usepackage{verbatim}
\usepackage{stmaryrd}
\usepackage{empheq}
\usepackage{booktabs}
\usepackage{mathrsfs}
\usepackage{tikz}
\usetikzlibrary{cd}
\usepackage{appendix}
\usepackage{ytableau}

\usepackage{comment} 
\excludecomment{TOOLONG} 

\usepackage{xcolor}

\definecolor{winered}{rgb}{0.7,0,0}
\definecolor{lessblue}{rgb}{0,0,0.7}

\usepackage{tikz-cd}

\tikzcdset{row sep/normal=2cm}
\tikzcdset{column sep/normal=2cm}
\tikzset{ampersand replacement=\&}

\usepackage[pdftex,colorlinks=true,linkcolor=winered,linktocpage=true,citecolor=lessblue,breaklinks=true,bookmarksopen=true]{hyperref}

\newcounter{mnotecount}[section]
\renewcommand{\themnotecount}{\thesection.\arabic{mnotecount}}
\newcommand{\mnote}[1]{\protect{\stepcounter{mnotecount}}${\raisebox{0.5\baselineskip}[0pt]{\makebox[0pt][c]{\color{magenta}{\tiny\em$\bullet$\themnotecount}}}}$\marginpar{\raggedright\tiny\em$\!\bullet$\themnotecount:
#1}\ignorespaces}

\makeatletter
\newcommand{\myitem}[3]{\item[#2]\def\@currentlabel{#3}\label{#1}}
\makeatother

\usepackage{titletoc}

\makeatletter

\def\@tocline#1#2#3#4#5#6#7{
\begingroup
  \par
    \parindent\z@ \leftskip#3 \relax \advance\leftskip\@tempdima\relax
                  \rightskip\@pnumwidth plus 4em \parfillskip-\@pnumwidth
    \ifcase #1 
       \vskip 0.6em \hskip 0em 
       \or
       \or \hskip 0em 
       \or \hskip 1em 
    \fi%
    #6
    \nobreak\relax{\leavevmode\leaders\hbox{\,.}\hfill}
    \hbox to\@pnumwidth {\@tocpagenum{#7}}
  \par
\endgroup
}

 \def\l@section{\@tocline{0}{0pt}{0pc}{}{}}

\renewcommand{\tocsection}[3]{%
  \indentlabel{\@ifnotempty{#2}{ 
    \ignorespaces\bfseries{#2. #3}}}
  \indentlabel{\@ifempty{#2}{\ignorespaces\bfseries{#3}}{}} 
    \vspace{1.5pt}}

\renewcommand{\tocsubsection}[3]{%
  \indentlabel{\@ifnotempty{#2}{
    \ignorespaces#2. #3}}
  \indentlabel{\@ifempty{#2}{\ignorespaces #3}{}}
    \vspace{1.5pt}}

\renewcommand{\tocsubsubsection}[3]{%
  \indentlabel{\@ifnotempty{#2}{
    \ignorespaces#2. #3}}
  \indentlabel{\@ifempty{#2}{\ignorespaces #3}{}}
    \vspace{1.5pt}}

\makeatother

\makeatletter
\def\@nomenstarted{0}
\newlength{\@nomenoldtabcolsep}

\newcommand{\nomenstart}
  {%
    \def\@nomenstarted{1}%
    \setlength{\@nomenoldtabcolsep}{\tabcolsep}%
    \setlength{\tabcolsep}{3.5pt}%
    \begin{longtable}{p{0.11\textwidth} p{0.86\textwidth}}
  }

\newcommand{\nomenitem}[2]{%
    \ifcase\@nomenstarted%
      \or 
      \or \\ 
    \fi%
    #1\,{\leavevmode\leaders\hbox{\,.}\hfill} & #2%
    \def\@nomenstarted{2}%
  }%
\newcommand{\nomenend}
  {\\%
      \end{longtable}%
      \setlength{\tabcolsep}{\@nomenoldtabcolsep}%
      \def\@nomenstarted{0}%
  }
\makeatother

\def\Hop{\mathsf{H}}
\def\Cop{\mathsf{C}}
\def\Dop{\mathsf{D}}

\def\Kop{\mathsf{K}}

\def\id{\mathrm{id}}

\def\id{\mathrm{id}}

\newcommand{\sCurl}{\mathscr{C}}
\newcommand{\sCurlDagger}{\mathscr{C}^\dagger}
\newcommand{\sTwist}{\mathscr{T}}
\def\InvSymb{\mathbb{I}}
\def\ImA{\Im\mathbf{A}}
\def\ReA{\Re\mathbf{A}}
\def\AA{\mathbf{A}}
\def\htf{h^{\mathrm{tf}}}

\makeatletter
\newcommand{\vast}{\bBigg@{4}}
\newcommand{\Vast}{\bBigg@{5}}
\makeatother

\allowdisplaybreaks

\numberwithin{equation}{section}
\numberwithin{figure}{section}
\newtheorem{thm}{Theorem}[section]

\newtheorem{prop}[thm]{Proposition}
\newtheorem{lemma}[thm]{Lemma}

\newtheorem*{thm*}{Theorem}
\newtheorem*{prop*}{Proposition}
\newtheorem*{cor*}{Corollary}
\newtheorem*{conj*}{Conjecture}

\theoremstyle{definition}
\newtheorem{definition}[thm]{Definition}

\theoremstyle{remark}
\newtheorem{rmk}[thm]{Remark}

\newcommand{\mc}{\mathcal}
\newcommand{\cA}{\mc A}
\newcommand{\cB}{\mc B}
\newcommand{\cC}{\mc C}

\newcommand{\cE}{\mc E}
\newcommand{\cF}{\mc F}

\newcommand{\cH}{\mc H}

\newcommand{\cM}{\mc M}
\newcommand{\cN}{\mc N}
\newcommand{\cO}{\mc O}

\newcommand{\cR}{\mc R}
\newcommand{\cS}{\mc S}

\newcommand{\cV}{\mc V}

\newcommand{\cX}{\mc X}

\newcommand{\ms}{\mathscr}

\newcommand{\scri}{\ms I}

\newcommand{\C}{\mathbb{C}}
\newcommand{\N}{\mathbb{N}}
\newcommand{\R}{\mathbb{R}}
\newcommand{\Z}{\mathbb{Z}}

\newcommand{\Sph}{\mathbb{S}}

\newcommand{\sfG}{\mathsf{G}}

\newcommand{\bfa}{\mathbf{a}}

\newcommand{\fm}{\mathfrak{m}}

\newcommand{\ft}{\mathfrak{t}}

\newcommand{\sld}{\slashed{d}{}}
\newcommand{\slg}{\slashed{g}{}}

\newcommand{\sldelta}{\slashed{\delta}{}}
\newcommand{\slDelta}{\slashed{\Delta}{}}
\newcommand{\slnabla}{\slashed{\nabla}{}}

\newcommand{\slstar}{\slashed{\star}}
\newcommand{\sltr}{\operatorname{\slashed\tr}}

\newcommand{\scal}{\mathsf{S}}
\newcommand{\vect}{\mathsf{V}}

\newcommand{\scalspace}{\mathbf{S}}
\newcommand{\vectspace}{\mathbf{V}}

\newcommand{\Ker}{\operatorname{Ker}}

\newcommand{\mathspan}{\operatorname{span}}
\newcommand{\supp}{\operatorname{supp}}

\newcommand{\tr}{\operatorname{tr}}

\newcommand{\Ups}{\Upsilon}

\newcommand{\eps}{\epsilon}

\newcommand{\hra}{\hookrightarrow}
\newcommand{\la}{\langle}

\newcommand{\ol}{\overline}
\newcommand{\pa}{\partial}
\newcommand{\ra}{\rangle}

\newcommand{\wh}{\widehat}
\newcommand{\wt}{\widetilde}

\newcommand{\ubar}[1]{\underaccent{\bar}{#1}}

\newcommand{\bop}{{\mathrm{b}}}
\newcommand{\scop}{{\mathrm{sc}}}

\newcommand{\scl}{{\mathrm{sc}}}

\newcommand{\Diff}{\mathrm{Diff}}

\newcommand{\Vb}{\cV_\bop}
\newcommand{\Diffb}{\Diff_\bop}

\newcommand{\Vsc}{\cV_\scl}
\newcommand{\Diffsc}{\Diff_\scl}

\newcommand{\Omegasc}{{}^{\scop}\Omega}

\newcommand{\Tb}{{}^{\bop}T}
\newcommand{\Tsc}{{}^{\scl}T}

\newcommand{\half}{{\tfrac{1}{2}}}

\newcommand{\sigmasc}{{}^\scop\sigma}

\newcommand{\CI}{\cC^\infty}
\newcommand{\CIdot}{\dot\cC^\infty}

\newcommand{\Hb}{H_{\bop}}

\newcommand{\Hbext}{\bar H_{\bop}}

\newcommand{\Hbsupp}{\dot H_{\bop}}

\newcommand{\Hext}{\bar H}

\newcommand{\Hsupp}{\dot H}

\newcommand{\Hsc}{H_{\scop}}

\newcommand{\Ric}{\mathop{\rm Ric}}
\newcommand{\Weyl}{\mathop{\rm Weyl}}

\newcommand{\bhm}{\fm}
\newcommand{\bha}{\afrak}
\newcommand{\bhvecta}{\bfa}
\newcommand{\bhn}{\mathfrak{n}}
\newcommand{\bhc}{\mathfrak{c}}
\newcommand{\bhp}{p}

\newcommand{\SL}{\mathop{\rm SL}}

\newcommand{\SO}{\mathop{\rm SO}}

\newcommand{\Co}{\mathbb{C}}

\newcommand{\openbigpmatrix}[1]
  {%
    \def\@bigpmatrixsize{#1}%
    \addtolength{\arraycolsep}{-#1}%
    \begin{pmatrix}%
  }
\newcommand{\closebigpmatrix}
  {%
    \end{pmatrix}%
    \addtolength{\arraycolsep}{\@bigpmatrixsize}%
  }

\newlength{\enummargin}

\DeclareGraphicsExtensions{.mps,.pdf,.png}
\newcommand{\inclfig}[1]{\includegraphics{#1}}

\makeatletter
\newcommand*{\fwbw}[1]{\expandafter\@fwbw\csname c@#1\endcsname}
\newcommand*{\@fwbw}[1]{\ifcase #1 \or {\rm fw}\or {\rm bw}\fi}
\AddEnumerateCounter{\fwbw}{\@fwbw}
\makeatother

\usepackage{tipa}
\DeclareMathOperator{\tho}{\text{\rm\textthorn}}
\DeclareMathOperator{\edt}{\eth}

\newcommand{\Lie}{\mathscr{L}}

\newcommand{\II}{\mathbb I}

\newcommand{\mfrak}{\mathfrak{m}}
\newcommand{\afrak}{\mathfrak{a}}
\newcommand{\nfrak}{\mathfrak{n}}
\newcommand{\cfrak}{\mathfrak{c}}

\def\InvSymb{\mathbb{I}}

\begin{document}

\title[Mode analysis for linearized Einstein]{Mode analysis for the linearized Einstein equations on the Kerr metric : the large $\bha$ case.}

\date{\today}


\author{Lars Andersson}
\address{Albert Einstein Institute, Am M\"uhlenberg 1, D-14476 Potsdam, Germany}
\email{laan@aei.mpg.de}

\author{Dietrich H\"afner}
\address{Universit\'e Grenoble Alpes, Institut Fourier, 100 rue des maths, 38402 Gi\`eres, France}
\email{dietrich.hafner@univ-grenoble-alpes.fr}

\author{Bernard F. Whiting}
\address{Department of Physics, University of Florida, 2001 Museum Road, Gainesville, FL 32611-8440, USA}
\email{bernard@phys.ufl.edu}

\date{\today}

\begin{abstract}
We give a complete analysis of mode solutions for the linearized Einstein equations and the $1-$form wave operator on the Kerr metric in the large $\bha$ case. By mode solutions we mean solutions of the form $e^{-it_*\sigma}\tilde{h}(r,\theta,\varphi)$ where $t_*$ is a suitable time variable. The corresponding Fourier transformed $1-$form wave operator and linearized Einstein operator are shown to be Fredholm between suitable function spaces and $\tilde{h}$ has to lie in the domain of these operators. These spaces are constructed following the general framework of Vasy \cite{VasyMicroKerrdS}, \cite{Va1}-\cite{Va2} along the lines of \cite{HHV}. No mode solutions exist for ${\Im}\, \sigma\ge 0,\, \sigma\neq 0$. For $\sigma=0$ mode solutions are Coulomb solutions for the $1-$form wave operator and linearized Kerr solutions plus pure gauge terms in the case of the linearized Einstein equations. If we fix a De Turck/wave map gauge, then the zero mode solutions for the linearized Einstein equations lie in a fixed $7-$dimensional space. The proof relies on the absence of modes for the Teukolsky equation shown by the third author in \cite{Wh} and a complete classification of the gauge invariants of linearized gravity on the Kerr spacetime, see \cite{2018PhRvL.121e1104A, 2019arXiv191008756A}. \end{abstract}

\maketitle

\setlength{\parskip}{0.00in}
\tableofcontents
\setlength{\parskip}{0.05in}
\section{Introduction}
There has been important progress in our understanding of stability properties of black hole solutions of the Einstein equations in recent years. Non-linear stability is  known for very slowly rotating Kerr-de Sitter, see Hintz-Vasy \cite{HV}. Recently, non-linear stability has been proved for very slowly rotating Kerr solutions in a series of papers by Giorgi, Klainerman, and  Szeftel \cite{KS1, KS, GKS}. See also  Dafermos, Holzegel, Rodnianski and Taylor \cite{DHRT} for related work. 

All these non linear stability results are based on an a priori understanding of linear stability. Linear stability of the Schwarzschild metric was shown by Dafermos, Holzegel and Rodnianski \cite{DHR2}. We also refer to work of Giorgi \cite{Gi}, Hung, Keller and Wang \cite{HKW} in this context. Linear stability for the Kerr metric was shown by Andersson, B\"ackdahl, Blue and Ma \cite{ABBM}, and by H\"afner, Hintz and Vasy \cite{HHV} for small angular momentum of the black hole. The linear stability of the Kerr metric is closely linked to decay properties of solutions of the Teukolsky equation.  Such results have been obtained by Dafermos, Holzegel and Rodnianski  \cite{DHR1},   Finster and Smoller \cite{FS}, Ma \cite{Ma}, Ma and Zhang \cite{MZ}, Shlapentokh-Rothman and Teixeira da Costa \cite{SRTC}.    See also \cite{IoKl} for related work.

Almost all of the above mentioned results are restricted to small angular momentum. It is expected that most black holes are rapidly rotating \cite{Thorne:1974ve,Ananna:2020ben}. It is thus worth having a closer look where the restriction to small angular momentum comes from. We will use in this paper the setting of the linear stability paper by H\"afner, Hintz and Vasy \cite{HHV}. The result in \cite{HHV} is a result for small angular momentum but, as we will argue in the following, the main missing points to get the result for the full subextreme range of the angular momentum of the black hole are on the level of the analysis of mode solutions. 

To understand this in a bit more detail let us first fix a wave map/de Turck gauge\footnote{See Section \ref{Sec6.2} for a precise definition.} and let $L_g$ be the gauge fixed linearized Einstein operator with this gauge. We consider a suitable time variable $t_*$ which is constant along ingoing resp. outgoing principal null geodesics close to the horizon resp. close to null infinity. To explain how the method works we consider the forcing problem 
\[
  L_g h = f,\qquad t_*\geq 0\ \ \text{on}\ \supp h,\ \supp f,
\]
where $f$ has compact support in $t_*$ and suitable decay (roughly, $r^{-2-\epsilon}$) as $r\to\infty$. The approach is then to take the Fourier transform in $t_*$, giving the representation
\begin{equation}
\label{EqIIFT}
  h(t_*) = \frac{1}{2\pi}\int_{\Im\sigma=C} e^{-i\sigma t_*} \wh{L_g}(\sigma)^{-1}\hat f(\sigma)\,d\sigma,
\end{equation}
initially for $C$ large enough (which gives exponential bounds for $h$). We refer to \cite{Vai} for an introduction to this approach and to \cite{BoHa} for an application of it in a simpler situation. The advantage of taking the time variable $t_*$ rather than the Boyer Lindquist time $t$ is that precise mapping properties of $\wh{L_g}(\sigma)$ are easier to read off, and, more importantly, the analysis near $\sigma=0$ is simplified.

The strategy is then to shift the contour of integration in~\eqref{EqIIFT} to $C=0$, which requires a detailed analysis of $\wh{L_g}(\sigma)$. If $\wh{L_g}(\sigma)^{-1}$ has some suitable regularity properties up to the real axis, then \eqref{EqIIFT} gives decay properties of $h(t_*)$. The better the regularity is, the better is the decay. In our concrete situation however $\wh{L_g}(0)$ has a non trivial kernel and the resolvent  $\wh{L_g}(\sigma)^{-1}$ has to be decomposed into a singular and a regular part. Elements of the kernel of $\wh{L_g}(\sigma)$  are called mode solutions \footnote{Note that the notion of mode solution depends on the exact domain of the operator.} and understanding them is crucial for a precise description of the singular part of  $\wh{L_g}(\sigma)^{-1}$. Concretely the ingredients one needs are the following :   
\begin{enumerate}
\item A robust\footnote{This means that the framework remains unchanged under sufficiently small perturbations of the metric.} Fredholm framework for the operators $\wh{L_g}(\sigma)$. In particular the authors of \cite{HHV} construct suitable function spaces such that the operator $\wh{L_g}(\sigma)$ acts as a Fredholm operator of index $0$ between them. We will show in this paper that this construction can be carried out for all subextreme values of the angular momentum of the black hole. 
\item High energy estimates, i.e. estimates for the resolvent $\wh{L_g}^{-1}(\sigma)$ for large $\vert {\Re} \sigma\vert$ and bounded ${\Im} \sigma$. These estimates only use the structure of the trapping which is $r-$normally hyperbolic. Dyatlov has shown that the trapping in the Kerr metric has the same structure for all 
subextreme values of the angular momentum per unit mass $\bha$ of the black hole (see \cite{Dy}).   
\item Uniform Fredholm estimates down to $\sigma=0$. These estimates only use the asymptotic structure of the metric at infinity and we show in this paper that they hold for all subextreme values of $\bha$. We refer to \cite{Va1}, \cite{Va2} for the method used. 
\item The regularity of the resolvent at $\sigma=0$. Similarly to the previous point the important ingredient here is the asymptotic structure of the metric and we expect that the estimates hold for all subextreme values of $\bha$.  However the study of the precise regularity of the resolvent is postponed to future work.  
\item Mode stability for $L_g$, i.e. invertibility of $\wh{L_g}(\sigma)$ for $\sigma\neq 0,\, \Im\, \sigma\ge 0$ and precise understanding of the $0$ modes (elements of the kernel of $\wh{L_g}(0)$), which, only for small $\bha$, has been obtained in \cite{HHV}.  \end{enumerate} 
Of the steps listed above, it is only mode stability in the sense of point $(5)$ that does not follow for all subextreme values of $\bha$ from the general robust Fredholm framework.\footnote{Note also that a Fredholm setting for the wave equation on the De Sitter Kerr metric was recently established by O. Lindblad Petersen and A. Vasy in the full range of $\bha$, see \cite{LPV}.} The main result of the present paper fills this gap. 

We now give an informal statement of our main theorem. Recall that the Kerr black hole is a 2$-$parameter family of metrics $g_b$, with parameter $b=(\bhm, \bha)$ corresponding to mass and angular momentum per unit mass, respectively. The  linearized Einstein operator $L_g$ in harmonic gauge with respect to the Kerr background is just the Lichnerowicz d'Alembertian acting on symmetric 2$-$tensors, which upon taking a Fourier transform  induces an operator $\wh{L_g}(\sigma)$. A mode solution $\dot g_{ab}$ with frequency $\sigma$ is a solution to $\wh{L_g}(\sigma) \dot g_{ab} = 0$ with boundary conditions corresponding to the absence of radiation entering the black hole exterior. We can now give an informal version of our main theorem. 
\begin{thm}[Mode stability for linearized gravity]\label{ThmMS:intro} Consider the gauge fixed linearized Einstein equation on a subextreme Kerr spacetime. Let $\Im \sigma \geq 0$, and let $\dot g_{ab}$ be a mode solution of the gauge fixed linearized Einstein equation with frequency $\sigma$. Then $\dot g_{ab}$ is, modulo gauge, a perturbation of the Kerr metric with respect to the Kerr parameters. 
\end{thm}
See Theorem \ref{ThmMS} below for the precise statement. The result mentioned in point $(5)$ above is a consequence of Theorem \ref{ThmMS:intro} together with ideas developed in application of the robust Fredholm setup for slowly rotating Kerr spacetimes, which are easily adapted to the present case. See section \ref{Sec7} below.

In contrast to the Fredholm analysis, the proof of mode stability requires the precise form of the equations and in most cases uses separation of variables and a delicate analysis of the separated equations. Mode stability can fail in prominent examples such as the Klein-Gordon equation on the Kerr spacetime \cite{SR} or the charged Klein-Gordon field on the De Sitter-Reissner-Nordstr\"om metric \cite{BH}. 

The general strategy for the mode analysis is to first analyse the modes for the linearized Einstein equations without gauge fixing. However to establish linear stability, it will be important to  understand the mode solutions for the gauge fixed Einstein equations. In the wave map/de Turck  gauge, the $1-$form wave operator plays the role of a gauge propagation operator; this explains why the mode analysis for this operator is important in this context. It should however be pointed out that already in the small $\bha$ case, quadratically growing generalized zero modes appear in the usual wave map/de Turck gauge. At least in the small $\bha$ case, these modes can be eliminated by constraint damping via a perturbation argument \cite{HHV}.  

Summarizing the complete analysis of mode solutions divides into the following points:
\begin{enumerate}
\item An analysis of mode solutions of the linearized Einstein operator. 
\item  An analysis of mode solutions of the $1-$form wave operator. 
\item Implementation of constraint damping.  
\item Non-degenerate control of generalized zero energy states. 
\end{enumerate}
We address points $(1)$ and $(2)$ in this paper, $(3),\, (4)$ are postponed to future work.   

Mode analysis for perturbations of the Kerr metric has a long history and as already mentioned can be linked to the mode analysis of the Teukolsky equation.  The central breakthrough in this context was the paper by Whiting in 1989 showing absence of modes with positive imaginary part for the Teukolsky equation for all $\bha$ subextreme \cite{Wh}. We also refer to \cite{CTC} for a new proof of this result. Later, Andersson, Ma, Paganini and Whiting \cite{AMPW} showed the absence of modes also on the real axis for non zero spectral parameter. In the present paper we establish the absence of suitably defined zero frequency modes. 
We also refer to \cite{CTC}, \cite{Hi2} for partial mode stability results in the De Sitter Kerr case.

The $1-$form wave equation in Kerr is not separable, but divergence free solutions of this equation give rise to solutions of the Maxwell equations. When considering solutions of the Maxwell equations or the linearized Einstein equations, one can compute the so called Teukolsky scalars which are solutions of the Teukolsky equation. Whereas this equation doesn't have any mode solutions, this is true neither for the linearized Einstein equation nor for the $1-$form wave equation.  Indeed the linearized Einstein equation will have zero modes consisting of a linearized Kerr metric plus a pure gauge solution, and the $1-$form wave equation will have solutions which correspond to Coulomb solutions for the Maxwell field. However the vanishing of the Teukolsky scalars will be the central information to show that the only mode solutions encountered are the expected ones. Whereas one obtains the mode solutions for the $1-$form wave equation more or less by direct integration, the situation is more complicated for the Einstein case.   

In the Einstein case, the gauge invariants of linearized gravity on the Kerr spacetime will play a crucial role in our proof. They have been completely classified in  \cite{2018PhRvL.121e1104A, 2019arXiv191008756A}. It is a remarkable fact that these gauge invariants have, for vanishing extreme Teukolsky scalars, exactly the form that they have for the Plebanski-Demianski family of line elements, parametrized by $\mfrak,\, \afrak,\, \cfrak$ and $\nfrak$, see \cite{Aksteiner:GI}. The set of gauge invariants is complete, see \cite{2019arXiv191008756A}. It follows that locally the perturbation to consider is up to gauge a linearized Plebanski-Demianski line element. We make this argument global by using the setting of \cite{2019arXiv191008756A} and the Poincar\'e lemma. We then have to show that the nut parameter perturbation $\dot \nfrak$ and the acceleration parameter perturbation $\dot \cfrak$ are zero to argue that the solution is a linearized Kerr solution plus a pure gauge term.  For a solution that decays like $\cO(r^{-1})$ at infinity, it follows directly from decay considerations for the gauge invariants that $\dot \nfrak$ and $\dot \cfrak$ are zero.  However, this fall-off is more than what we want to impose, a priori, in our functional setting. The remedy is a normal operator argument that gives a certain polyhomogeneous expansion of our mode solution. This argument however can't be applied directly to the linearized Einstein equation but only for the gauged fixed one. We therefore first correct the solution of the linearized Einstein equation without gauge fixing to a solution of the gauge fixed linearized Einstein equation by adding a linearized Kerr metric plus gauge term. 
 
The paper is organized as follows:
\begin{enumerate}
\item In Section \ref{SB} we introduce b$-$ and scattering structures as well as the corresponding Sobolev spaces. 
\item Section \ref{SecKerr} is devoted to the description of the Kerr metric. 
\item In Section \ref{Sec4} we summarize the existing results on the Teukolsky  equation and show the absence of zero modes. We also describe polyhomogeneous expansions of the
mode solutions and give some geometric background on the Teukolsky equation. 
\item Section \ref{Sec1form} gives the results for the $1-$form wave operator. We explain the link to the Maxwell equation and show that the only possible modes are Coulomb type solutions. 
\item Section \ref{SecEinst} is devoted to the analysis of the linearized Einstein equations. For spectral parameter different from zero no mode solutions exist. Zero modes are found to be linearized Kerr solutions plus pure gauge terms. 
\item Eventually Section \ref{Sec7} is devoted to the analysis of the gauge fixed linearized Einstein operator using the usual wave map/De Turck gauge. Again for spectral parameter different from zero no mode solutions exist. Zero modes are linearized Kerr plus gauge term. The gauge term lies in a fixed $7-$dimensional space characterized by a three dimensional space of asymptotic translations, a Coulomb type solution and asymptotic rotations. Our results are analogous to those obtained in \cite{HHV} in the case of small angular momentum of the black hole.  
 
\end{enumerate}

{\bf Acknowledgments.}
 This project was started when discussing with Andr\'as Vasy about the possibility of generalizing the linear stability result of \cite{HHV} to the large $\bha$ case, we are very grateful for his input. We are also very grateful to Steffen Aksteiner and Thomas B\"ackdahl for allowing us to use their argument which gives Proposition \ref{prop6.5}. Many thanks also to Steffen Aksteiner, Peter Hintz and Pascal Millet for very fruitful discussions on this paper. This project was started during the program "General Relativity, Geometry and Analysis: beyond the first 100 years after Einstein"  at the Mittag Leffler Institute in Stockholm in 2019 and we are grateful to this Institute for hospitality.   

\section{$b$ and scattering structures}
\label{SB}

We first discuss geometric structures on manifolds with boundaries or corners, and corresponding function spaces. Thus, let $X$ be a compact $n-$dimensional manifold with boundary $\pa X\neq\emptyset$, and let $\rho\in\CI(X)$ denote a boundary defining function: $\pa X=\rho^{-1}(0)$, $d\rho\neq 0$ on $\pa X$. We then define the Lie algebras of \emph{$b-$vector fields} and \emph{scattering vector fields} by
\begin{equation}
\label{EqBVf}
  \Vb(X) = \{ V\in\cV(X) \colon V\ \text{is tangent to}\ \pa X \}, \quad
  \Vsc(X) = \rho\Vb(X).
\end{equation}
In local \emph{adapted coordinates} $x\geq 0$, $y\in\R^{n-1}$ on $X$, with $x=0$ locally defining the boundary of $X$ (thus $\rho=a(x,y)x$ for some smooth $a>0$), elements of $\Vb(X)$ are of the form $a(x,y)x\pa_x+\sum_{i=1}^{n-1}b^i(x,y)\pa_{y^i}$, with $a,b^i\in\CI(X)$, while elements of $\Vsc(X)$ are of the form $a(x,y)x^2\pa_x+\sum_{i=1}^{n-1}b^i(x,y) x\pa_{y^i}$. Thus, there are natural vector bundles
\[
  \Tb X\to X, \quad \Tsc X \to X,
\]
with local frames given by $\{x\pa_x,\pa_{y^i}\}$ and $\{x^2\pa_x,x\pa_{y^i}\}$ resp., such that $\Vb(X)=\CI(X;\Tb X)$ and $\Vsc(X)=\CI(X;\Tsc X)$; thus, for example, $x\pa_x$ is a \emph{smooth, non-vanishing} section of $\Tb X$ down to $\pa X$. Over the interior $X^\circ$, these bundles are naturally isomorphic to $T X^\circ$, but the maps $\Tb X\to T X$ and $\Tsc X\to T X$ fail to be injective over $\pa X$. We denote by $\Diffb^m(X)$, resp.\ $\Diffsc^m(X)$ the space of $m-$th order b$-$, resp.\ scattering differential operators, consisting of linear combinations of up to $m-$fold products of elements of $\Vb(X)$, resp.\ $\Vsc(X)$.

The dual bundles $\Tb^*X\to X$ (b$-$cotangent bundle), resp.\ $\Tsc^*X\to X$ (scattering cotangent bundle) have local frames
\[
  \frac{d x}{x},\ d y^i,\quad \text{resp.}\quad \frac{d x}{x^2},\ \frac{d y^i}{x},
\]
which are \emph{smooth} down to $\pa X$ as sections of these bundles (despite their being singular as standard covectors, i.e.\ elements of $T^*X$). A \emph{scattering metric} is then a section $g\in\CI(X;S^2\,\Tsc^*X)$ which is a non-degenerate quadratic form on each scattering tangent space $\Tsc_p X$, $p\in X$; b$-$metrics are defined analogously.

These structures arise naturally on compactifications of non-compact manifolds, the simplest example being the \emph{radial compactification of $\R^n$}, defined by
\begin{equation}
\label{EqBRadCp}
  \ol{\R^n} := \bigl(\R^n \sqcup ([0,1)_\rho\times\Sph^{n-1})\bigr) / \sim
\end{equation}
where the relation $\sim$ identifies a point in $\R^n\setminus\{0\}$, expressed in polar coordinates as $r\omega$, $r>0$, $\omega\in\Sph^{n-1}$, with the point $(\rho,\omega)$ where we can choose
\[
  \rho=r^{-1};
\]
this has a natural smooth structure, with smoothness near $\pa\ol{\R^n}=\rho^{-1}(0)$ meaning smoothness in $(\rho,\omega)$. In polar coordinates in $r>1$, the space of b$-$vector fields is then locally spanned over $\CI(\ol{\R^n})$ by $\rho\pa_\rho=-r\pa_r$ and $\cV(\Sph^{n-1})$; scattering vector fields are spanned by $\rho^2\pa_\rho=-\pa_r$ and $\rho\cV(\Sph^{n-1})$. Using standard coordinates $x^1,\ldots,x^n$ on $\R^n$, scattering vector fields on $\ol{\R^n}$ are precisely those of the form
\[
  \sum_{i=1}^n a^i\pa_{x^i}, \quad a^i\in\CI(\ol{\R^n});
\]
this entails the statement that $\{\pa_{x^1},\ldots,\pa_{x^n}\}$, which is a frame of $T^*\R^n$, extends by continuity to a smooth frame of $\Tsc^*\ol{\R^n}$ down to $\pa\ol{\R^n}$. Thus, the space of scattering vector fields on $\ol{\R^n}$ is generated over $\CI(\ol{\R^n})$ by constant coefficient (translation-invariant) vector fields on $\R^n$. On the other hand, $\Vb(\ol{\R^n})$ is spanned over $\CI(X)$ by vector fields on $\R^n$ with coefficients which are \emph{linear} functions, i.e.\ by $\pa_{x^1},\ldots,\pa_{x^n}$, and $x^i\pa_{x^j}$, $1\leq i,j\leq n$.

On the dual side, $\Tsc^*\ol{\R^n}$ is spanned by $d x^i$, $1\leq i\leq n$, \emph{down to} $\pa\ol{\R^n}$. Therefore, a scattering metric $g\in\CI(\ol{\R^n},S^2\,\Tsc^*\ol{\R^n})$ is a non-degenerate linear combination of $d x^i\otimes_s d x^j=\half(d x^i\otimes d x^j+d x^j\otimes d x^i)$ with $\CI(\ol{\R^n})$ coefficients. In particular, the Euclidean metric
\[
  (d x^1)^2+\cdots+(d x^n)^2 \in \CI(\ol{\R^n};S^2\,\Tsc^*\ol{\R^n})
\]
is a Riemannian scattering metric.

By extension from $T^*X^\circ$, one can define Hamilton vector fields $H_p$ of smooth functions $p\in\CI(\Tsc^*X)$. In fact $H_p\in\Vsc(\Tsc^*X)$ is a scattering vector field on $\Tsc^*X$, which is a manifold with boundary $\Tsc^*_{\pa X}X$. (Likewise, if $p\in\CI(\Tb^*X)$, then $H_p\in\Vb(\Tb^*X)$.) For us, the main example will be the Hamilton vector field $H_G$ where $G(z,\zeta):=|\zeta|_{g_z^{-1}}^2$ is the dual metric function of a scattering metric $g\in\CI(X;S^2\,\Tsc^*X)$.

We next introduce Sobolev spaces corresponding to b$-$ and scattering structures. As an integration measure on $X$, let us use a \emph{scattering density}, i.e.\ a positive section of $\Omegasc^1 X=|\Lambda^n\,\Tsc^*X|$, which in local adapted coordinates takes the form $a(x,y)|\tfrac{d x}{x^2}\frac{d y}{x^{n-1}}|$ with $0<a\in\CI(X)$. (On $\ol{\R^n}$, one can take $|d x^1\cdots d x^n|$.) This provides us with a space $L^2(X)$; the norm depends on the choice of density, but all choices lead to equivalent norms. Working with a b$-$density on the other hand would give a different space, namely a weighted version of $L^2(X)$; we therefore stress that even for b$-$Sobolev spaces, we work with a scattering density. Thus, for $m\in\N_0$, we define
\[
  H_\bullet^m(X) := \{ u\in L^2(X) \colon V_1\cdots V_j u\in L^2(X)\ \forall\,V_1,\ldots,V_j\in\cV_\bullet(X),\ 0\leq j\leq m \},\quad
  \bullet=\bop,\scop,
\]
called \emph{b$-$} or \emph{scattering Sobolev space}. Using a finite spanning set in $\cV_\bullet(X)$, one can give this the structure of a Hilbert space; $H_\bullet^m(X)$ for general $m\in\R$ is then defined by duality and interpolation. If $q\in\R$, we denote weighted Sobolev spaces by
\[
  H_\bullet^{m,q}(X) = \rho^q H_\bullet^m(X) = \{ \rho^q u \colon u\in H_\bullet^m(X) \}.
\]
For example, $\Hsc^{m,q}(\ol{\R^n})\cong\la x\ra^{-q}H^m(\R^n)$ is the standard weighted Sobolev space on $\R^n$. The space of weighted ($L^2-$)conormal functions, $\Hb^{\infty,q}$, on $X$ is defined as
\[
  \Hb^{\infty,q}(X) = \bigcap_{m\in\R}\Hb^{m,q}(X).
\]
Dually, we define
\[
  \Hb^{-\infty,q}(X) = \bigcup_{m\in\R}\Hb^{m,q}(X).
\]
Note that $\Hb^{m,q}(X)\subset\cC^{-\infty}(X):=\CIdot(X)^*$ (where $\CIdot(X)\subset\CI(X)$ is the subspace of functions vanishing to infinite order at $\pa X$) consists of tempered distributions. (In particular, they are extendible distributions at $\pa X$ in the sense of \cite[Appendix B]{Ho}.) We furthermore introduce the notation
\begin{equation}
\label{EqBHsPM}
  H_\bop^{m,q+} := \bigcup_{\eps>0} H_\bop^{m,q+\eps},\quad
  H_\bop^{m,q-} := \bigcap_{\eps>0} H_\bop^{m,q-\eps}
\end{equation}
for $m\in\R\cup\{\pm\infty\}$. A space closely related to $\Hb^{\infty,q}(X)$ is
\[
  \cA^q(X) := \{ u\in\rho^q L^\infty(X) \colon \Diffb(X)u\subset\rho^q L^\infty(X) \},
\]
consisting of \emph{weighted $L^\infty-$conormal functions}. For $X=\ol{\R^3}$, we have the inclusions
\[
  \Hb^{\infty,q}(\ol{\R^3}) \subset \cA^{q+3/2}(\ol{\R^3}), \quad
  \cA^q(\ol{\R^3}) \subset \Hb^{\infty,q-3/2-}(\ol{\R^3}),
\]
by Sobolev embedding. (The shift $\tfrac32$ in the weight is due to our defining b$-$Sobolev spaces with respect to scattering densities; indeed, for $m>\tfrac32$,
\begin{equation}
\label{EqBHbSobEmb}
  \Hb^{m,q}(\ol{\R^3};|d x^1\,d x^2\,d x^3|)=\Hb^{m,q+3/2}(\ol{\R^3};\la r\ra^{-3}|d x^1\,d x^2\,d x^3|) \hra \la r\ra^{-q-3/2}L^\infty(\ol{\R^3}),
\end{equation}
with the second density here being a b$-$density on $\ol{\R^3}$.) We define $\cA^{q+}$ and $\cA^{q-}$ analogously to~\eqref{EqBHsPM}. These notions extend readily to sections of rank $k$ vector bundles $E\to X$: for instance, in a local trivialization of $E$, an element of $H_\bullet^{m,q}(X,E)$ is simply a $k-$tuple of elements of $H_\bullet^{m,q}(X)$.

Suppose now $X'$ is a compact manifold with boundary, and let $X\subset X'$ be a submanifold with boundary. Suppose that its boundary decomposes into two non-empty sets
\begin{equation}
\label{EqBInt}
  \pa X = \pa_-X \sqcup \pa_+X,\qquad \pa_-X=\pa X\setminus\pa X',\quad \pa_+X=\pa X';
\end{equation}
we consider $\pa_+X$ to be a boundary `at infinity', while $\pa_-X$ is an interior, `artificial' boundary. Concretely, this means that we define (by a slight abuse of notation)
\[
  \Vb(X) := \{ V|_X \colon V\in\Vb(X') \}, \quad
  \Vsc(X) := \{ V|_X \colon V\in\Vsc(X') \};
\]
these vector fields are b or scattering at infinity, but are unrestricted at $\pa_-X$. A typical example is $X'=\ol{\R^n}$ and  $X=\ol{\{r\geq r_0>0\}}\subset X'$, in which case $\pa_-X=\{r=r_0\}$, while $\pa_+X=\pa X'$ is the boundary (at infinity) of $\ol{\R^n}$. See Figure~\ref{FigBInt}.

\begin{figure}[!ht]
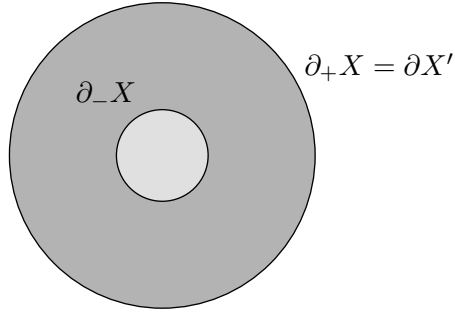

  \centering
\inclfig{FigBInt}
  \caption{A typical example of the setting~\eqref{EqBInt}: $X$ (dark gray) is a submanifold of $X'$ (the union of the dark and light gray regions) with two boundary components $\pa_+X=\pa X'$ and $\pa_-X\subset(X')^\circ$. We then consider function spaces such as $\Hbext^{m,q}(X)$, which measure b$-$regularity of degree $m$ at $\pa_+X$ (with decay rate $q$), and standard regularity (regularity with respect to incomplete vector fields) at $\pa_-X$.}
\label{FigBInt}
\end{figure}

There are now two natural classes of Sobolev spaces: those consisting of \emph{extendible distributions},
\begin{equation}
\label{EqBExt}
  \Hext_\bullet^{m,q}(X) := \{ u|_{X^\circ} \colon u\in H_\bullet^{m,q}(X') \},\quad \bullet=\bop,\scop,
\end{equation}
and those consisting of \emph{supported distributions},
\begin{equation}
\label{EqBSupp}
  \Hsupp_\bullet^{m,q}(X) := \{ u \colon u\in H_\bullet^{m,q}(X'),\ \supp u\subset X \}.
\end{equation}
Away from $\pa_-X$, these are the same as the standard spaces $H_\bullet^{m,q}(X)$; thus, the subspaces of $\Hext_\bullet^{m,q}(X)$ or $\Hsupp_\bullet^{m,q}(X)$ consisting of those elements which are polyhomogeneous (in particular automatically conormal) at $\pa_+X$ are well-defined.

\section{Geometric background}
In this section we introduce the Kerr family of metrics and some of the geometrical structures we will use in later sections. 
\label{SecKerr}
\subsection{The Kerr family}
We shall consider the Kerr family of solutions to the vacuum Einstein equation, parametrized by $b=(\bhm,\bha)\in\R^+\times\R$, denoting mass and angular momentum per unit mass, respectively. We restrict to the subextreme case $\vert \bha\vert< \bhm$. Let 
\begin{equation*}
r_\pm:=\bhm \pm\sqrt{\bhm^2-\bha^2}. 
\end{equation*}
For given $b=(\bhm,\bha)$, we put $b_0=(\bhm,0)$. 
Let $(\theta, \phi) \in [0,\pi] \times [0,2\pi]$ be coordinates on $\Sph^2$. The Kerr metric in Boyer-Lindquist coordinates $(t,r,\theta,\varphi) \in \mathbb{R} \times (r_+ , \infty) \times \Sph^2$ is    
\begin{subequations}
\label{EqKaMetric}
\begin{align}
  g_b^{\rm BL} &= \frac{\Delta_b}{\varrho_b^2}(d t-\bha\sin^2\theta\,d\varphi)^2 - \varrho_b^2\Bigl(\frac{d r^2}{\Delta_b}+d\theta^2\Bigr) - \frac{\sin^2\theta}{\varrho_b^2}\bigl(\bha\,d t-(r^2+\bha^2)d\varphi\bigr)^2, \\
  \intertext{with inverse} 
G_b^{\rm BL} &= \frac{1}{\Delta_b\varrho_b^2}\bigl((r^2+\bha^2)\pa_t+\bha\pa_\varphi\bigr)^2 - \frac{\Delta_b}{\varrho_b^2}\pa_r^2 - \frac{1}{\varrho_b^2}\pa_\theta^2 - \frac{1}{\varrho_b^2\sin^2\theta}(\pa_\varphi+\bha\sin^2\theta\,\pa_t)^2, \\
  \quad \Delta_{b} &= r^2-2\bhm r+\bha^2, \ \ 
  \varrho_{b}^2 = r^2+\bha^2\cos^2\theta. 
  \end{align}
\end{subequations}
Then $g_b^{\rm BL}$ is the metric of an isolated, rotating, stationary black hole with event horizon $\{ r=r_+ \}$, and domain of outer communications is $\{r > r_+\}$. Further, $\{r=r_-\}$ is the inner horizon.  
We have  $\sqrt{|\det g_b|} = \varrho_b^2 \sin\theta$. The Kerr metric is a solution of the Einstein vacuum equation:
\begin{equation}
\label{EqKaEinstein}
  \Ric\bigl(g_{b}^{\rm BL}\bigr)=0.
\end{equation}
The form~\eqref{EqKaMetric} of the metric breaks down at $r_{\pm}$ which are the roots of $\Delta_b$. Let 
\begin{equation*}
r_*=\int \frac{r^2+\bha^2}{\Delta_b} dr\quad \mbox{with}\quad r_*(3\bhm)=0. 
\end{equation*}
Consider coordinates $(t_*,r,\varphi_*,\theta)$, where $t_*(t,r)=t+F(r)$ and $\varphi_*(t,r)=\varphi+T(r)$ are smooth functions such that 
\begin{eqnarray*}
F(r)&=&\left\{\begin{array}{c} r_*\quad\mbox{for}\, r\le 3\bhm,\\
-r_*\quad\mbox{for}\, r\ge 4\bhm, \end{array}\right.
\end{eqnarray*}
\begin{eqnarray*}
T(r)&=&\left\{\begin{array}{c} \int \frac{a}{\Delta_b}\quad\mbox{for}\, r\le 3\bhm,\, T(3\bhm)=0,\\
0\quad\mbox{for}\, r\ge 4\bhm. \end{array}\right.
  \end{eqnarray*}
For $r\le 3\bhm$ the metric then takes the following form 
\begin{equation*}
g_{b,*}=\frac{\Delta_b}{\varrho_b^2}(dt_*-\bha\sin^2\theta d\varphi_*)^2-2(dt_*-\bha\sin^2\theta d\varphi_*)dr-\varrho_b^2d\theta^2-\frac{\sin^2\theta}{\varrho_b^2}(\bha dt_*-(r^2+\bha^2)d\varphi_*)^2,
\end{equation*}  
which is clearly smooth up to $r_+$.

Given
\begin{equation}
\label{eqr0}
r_- < r_0 < r_+ ,
\end{equation}
let
$$
X^0_b = [r_0,\infty)\times \Sph^2 
$$
and consider 
\begin{align}
\label{MX}
M^\circ_b = \R \times X^0_b, 
\end{align}
with coordinates $t_*, r, \theta, \varphi_*$. 
We compactify $X^0_b$ as follows: recalling the definition of $\ol{\R^3}$ from \eqref{EqBRadCp}, we set
\begin{equation*}
X:=\ol{X^0_b}\subset\ol{\R^3},\, \rho:=\frac{1}{r}.
\end{equation*}
 Thus, $X=\ol{\{r\geq r_0\}}$, and we let $\pa_-X=r^{-1}(r_0)$, $\pa_+X=\pa\ol{\R^3}\subset X$. Within $X$, the topological boundary of $\cX=(r_+,\infty)\times\Sph$ has two components,
\begin{equation}
\label{EqK0StaticBdy}
  \pa\cX = \pa_-\cX \sqcup \pa_+\cX,\qquad
  \pa_-\cX := r^{-1}(r_+), \quad
  \pa_+\cX := \rho^{-1}(0) = \pa_+ X.
\end{equation}
Note that $\pa_-X$ is distinct from $\pa_-\cX$, and is indeed a hypersurface lying \emph{beyond} the event horizon. Note that $\R \times\pa X$ has two components,
\begin{equation}
\label{EqK0SurfFin}
  \Sigma_{\rm fin} := \R \times\pa_- X
\end{equation}
(which is a spacelike hypersurface inside of the black hole) and $\R_{t_*}\times\pa_+ X=\R_{t_*}\times\pa_+\cX$ (which is future null infinity, typically denoted $\scri^+$); moreover, the future event horizon, $\cH^+$, is $\R_{t_*}\times\pa_-\cX$.

\begin{figure}[!ht]
  \centering
    \inclfig{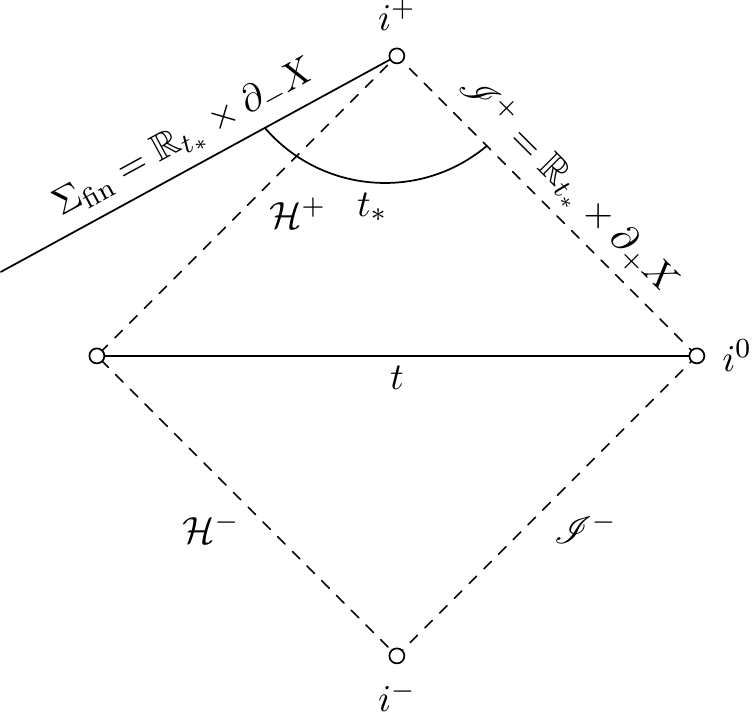}
  \caption{Illustration of time functions on $M^\circ_b$ in the Penrose diagram of the Kerr metric (including future/past null infinity $\scri^\pm$, the future/past event horizon $\cH^\pm$, spacelike infinity $i^0$, and future/past timelike infinity $i^\pm$). Shown are level sets of  the functions $t$ and $t_*$. We also indicate the boundaries $\R_{t_*}\times\pa_+ X$ (future null infinity again) and $\R_{t_*}\times\pa_- X$ (a spacelike hypersurface beyond the future event horizon).}  
\label{FigK0Time}
\end{figure}
Concerning the inverse metric we obtain for $r\le 3\bhm$ 
\begin{align*}
\label{INH}
G&=-\frac{\bha^2\sin^2\theta}{\varrho_b^2}\partial_{t_*}^2-2\frac{\bha}{\varrho_b^2}\partial_{\varphi_*}\partial_{t_*}-2\frac{r^2+\bha^2}{\varrho_b^2}\partial_r\partial_{t_*}-2\frac{\bha}{\varrho_b^2}\partial_{\varphi_*}\partial_r-\frac{1}{\varrho_b^2\sin^2\theta}\partial_{\varphi_*}^2\\
&-\frac{\Delta_b}{\varrho_b^2}\partial_r^2-\frac{1}{\varrho_b^2}\partial_{\theta}^2. 
\end{align*}     
For $r\ge 4\bhm$ we obtain 
\begin{align*}
G&=-\frac{\bha^2\sin^2\theta}{\varrho_b^2}\partial_{t_*}^2+4\frac{a\bhm r}{\varrho_b^2\Delta_b}\partial_{\varphi_*}\partial_{t_*}-\frac{\Delta_b-\bha^2\sin^2\theta}{\varrho_b^2\Delta_b\sin^2\theta}\partial_{\varphi_*}^2-2\rho^2\frac{r^2+\bha^2}{\varrho_b^2}\partial_{\rho}\partial_{t_*}\\
&-\frac{\rho^4}{\varrho_b^2}\Delta_b\partial_{\rho}^2-\frac{1}{\varrho_b^2}\partial_{\theta}^2. 
\end{align*}
We will also have to consider linearized Kerr metrics which are given by 
\begin{equation}\label{gdoteqn}
\dot{g}_b(\dot{b})=\left.\frac{d}{ds}\right\vert_{s=0}g_{b+s\dot{b}},\, \dot{b}\in \R^2.
\end{equation}
They fulfill the linearized Einstein equations 
\begin{align*}
D_{g_b}Ric(\dot{g}_b(\dot{b}))=0.
\end{align*}

\begin{rmk} \label{rem:orbit}
The explicit forms of the Kerr metric introduced above represent a rotating, isolated black hole at rest at the center of coordinates, with axis of rotation aligned with the axial Killing field $\partial_\varphi$. Acting on the Kerr black hole with global Poincar\'e transformations, corresponding to a change of reference frame, leads to a family of translated, boosted and rotated black holes. By analogy with a relativistic, massive, spinning particle, the linear and angular momenta of the Kerr black hole take values in a coadjoint orbit of the Poincar\'e group, determined by the parameters $\bhm, \bha$. Linearized perturbations of the Kerr metric are solutions of the linearized vacuum Einstein equation on the Kerr background, 
\begin{align} 
D_{g_b}\Ric \dot g = 0 . 
\end{align} 
In view of the above remarks, we see that $\ker D_{g_b}\Ric$ contains a subspace corresponding to the infinitesmal action of the Poincare group on the Kerr black hole, i.e. to the tangent space of the coadjoint orbit. In particular,  $\ker D_{g_b}\Ric$ contains perturbations that change the axis of rotation of the black hole. However, as we shall see in sections \ref{SecEinst}, \ref{Sec7} below, these are pure gauge, except for perturbations $\dot\bhm, \dot \bha$ of the two Kerr parameters.  
\end{rmk}

\begin{rmk}
All pairings will be defined with respect to the natural inner product on $L^2(X^0_b)$, i.e. 
\begin{equation*}
\la f,g\ra:=\int_{X^0_b}fg\varrho_b^2\sin\theta drd\theta d\varphi_*
\end{equation*}
for functions and 
\begin{equation*}
\la\omega,\nu\ra:=\int_{X^0_b}\omega^a\nu_a \varrho_b^2\sin\theta drd\theta d\varphi_*
\end{equation*}
for $1-$forms. 
\end{rmk}
\subsection{Stationarity, vector bundles, and geometric operators}
In the notation~\eqref{MX}, denote the projection to the spatial manifold by
\[
  \pi_X \colon M^\circ_b \to X^\circ_b.
\]

Suppose $E_1\to X^\circ_b$ is a vector bundle; then differentiation along $\pa_{t_*}$ is a well-defined operation on sections of the pullback bundle $\pi_X^*E_1$. The tangent bundle of $M^\circ_b$ is an important example of such a pullback bundle, as
\[
  T M^\circ_b \cong \pi_X^*(T_{t_*^{-1}(0)}M^\circ_b),
\]
likewise for the cotangent bundle and other tensor bundles.

Let $E_2\to X^\circ_b$ be another vector bundle, and suppose $\wh L(0)\in\Diff(X^\circ_b;E_1,E_2)$ is a differential operator; fixing $\ft=t_*+F$, $F\in\CI(X^\circ_b)$, we can then define its \emph{stationary extension} by assigning to $u\in\CI(M^\circ;\pi_X^*E_1)$ the section $(L u)(\ft,-):=\wh L(0)(u(\ft,-))$ of $\pi_X^*E_2$; this extension \emph{does} depend on the choice of $\ft$. The action of $L$ on stationary functions on the other hand is independent of the choice of $\ft$ since
\begin{equation}
\label{EqKStStat}
  L\pi_X^*=\pi_X^*\wh L(0).
\end{equation}
Via stationary extension, one can consider $\Diffb(X;E)$ (for $E\to X$ a smooth vector bundle down to $\pa_+ X$) to be a subalgebra of $\Diff(M^\circ_b;\pi_X^*E)$; likewise $\Diffsc(X;E)\hra\Diff(M^\circ_b;\pi_X^*E)$.

Conversely, if $L\in\Diff(M^\circ_b;\pi_X^*E_1,\pi_X^*E_2)$ is \emph{stationary}, i.e.\ commutes with $\pa_{t_*}$, there exists a unique (independent of the choice of $\ft$) operator $\wh L(0)\in\Diff(X^\circ_b;E_1,E_2)$ such that the relation~\eqref{EqKStStat} holds. More generally, we can consider the formal conjugation of $L$ by the Fourier transform in $\ft$,
\[
  \wh L(\sigma) := e^{i\sigma\ft}L e^{-i\sigma\ft} \in \Diff(X^\circ_b;E_1,E_2),
\]
where we identify the stationary operator $e^{i\sigma\ft}L e^{-i\sigma\ft}$ with an operator on $X^\circ_b$. Switching from $\ft$ to another time function, $\ft+F'$, $F'\in\CI(X^\circ)$, amounts to conjugating $\wh L(\sigma)$ by $e^{i\sigma F'}$.

In order to describe the uniform behavior of geometric operators at $\pa_+X$ concisely, we need to define a suitable extension of $T^*M^\circ_b$ to `infinity'. To accomplish this, note that the product decomposition~\eqref{MX} induces a splitting 
\[
  T^*M^\circ_b \cong 
  \pi_T^*(T^*\R_{t_*}) \oplus \pi_X^*(T^*X^\circ_b),
\]
where $\pi_T\colon M^\circ_b\to\R_{t_*}$ is the projection. We therefore define the \emph{extended scattering cotangent bundle} of $X$ by
\begin{equation}
\label{EqKStExtTsc}
  \wt\Tsc{}^*X := \R d t_* \oplus \Tsc^*X.
\end{equation}
At this point, $d t_*$ is merely a name for the basis of a trivial real rank $1$ line bundle over $X$; considering the pullback bundle $\pi_X^*\wt\Tsc{}^*X\to M^\circ_b$, we identify it with the differential of $t_*\in\CI(M^\circ_b)$, giving an isomorphism
\begin{equation}
\label{EqKStExtTscIso}
  (\pi_X^*\wt\Tsc{}^*X)|_{M^\circ_b} \cong T^*M^\circ_b.
\end{equation}
Smooth sections of $\wt{\Tsc^*}X\to X$ are linear combinations, with $\CI(X)$ coefficients, of $d t_*$ and the 1-forms $d x^i$, where $(x^1,x^2,x^3)$ are standard coordinates on $X^\circ_b\subset\R^3$. 
For a stationary metric $g$ on $M^\circ_b$, there exists a unique $g'\in\CI(X^\circ_b;S^2\,\wt\Tsc{}^*X)$ such that $\pi_X^*g'=g$, namely $g'$ is the restriction (as a section of $S^2\,\wt\Tsc{}^*X$) of $g$ to any transversal of $\pi_X$, such as level sets of $t_*$. Identifying $g$ with $g'$ and applying this to the Kerr family, we then have 
\begin{equation}
\label{EqKStMetrics}
  g_b \in \CI(X;S^2\,\wt\Tsc{}^*X);
\end{equation}
it is non-degenerate down to $\pa_+ X$. Let $\slg$ be the standard metric on the sphere. Then we have
\begin{equation}
\label{EqKStMetrics2}
\begin{split}
  &g_{b_0} - \ubar g \in \rho\CI, \quad \ubar g:=d t^2-d r^2-r^2\slg, \\
  &g_{(\bhm,\bha)}-g_{(\bhm,0)} \in \rho^2\CI,
\end{split}
\end{equation}
i.e.\ a Kerr metric equals the Minkowski metric $\ubar g$ to leading order, and is a $\cO(\rho^2)$ perturbation of the Schwarzschild metric of the same mass. 
We proceed to discuss basic geometric operators on Kerr spacetimes. We write
\begin{equation}
\label{EqKStDel}
  (\delta_g^*\omega)_{\mu\nu}=\half(\omega_{\mu;\nu}+\omega_{\nu;\mu}), \quad
  (\delta_g h)_\mu=-h_{\mu\nu;}{}^\nu, \quad
  \sfG_g = 1-\half g\,\mathrm{tr}_g,
\end{equation}
and furthermore denote by
\begin{equation}
\label{EqKStWave}
  \Box_{g,0},\ \Box_{g,1},\ \Box_{g,2},
\end{equation}
the wave operator $-\!\tr_g\nabla^2$ on scalars, 1-forms, and symmetric 2-tensors, respectively. These are sections of bundles with fiber dimension $1,4$ and $10$ resp. When the bundle is clear from the context, we shall simply write $\Box_g$.

When $g$ is the Kerr metric, then $g|_{\Tsc X\times\Tsc X}=-\hat{h}+\rho\CI$ is to leading order equal to the Euclidean metric $\hat{h}=(d x^1)^2+(d x^2)^2+(d x^3)^2$ on $\R^3$ (equipped with standard coordinates $(x^1,x^2,x^3)$ on $\R^3\setminus B(0,3\bhm)\cong X^\circ\setminus\{r<3\bhm\}$). Thus, the leading order terms at $\rho=0$ are simply those of the corresponding operators on Minkowski space $\R^4=\R_t\times\R^3_x$ with metric
\begin{equation}
\label{EqKStMink}
  \ubar g = d t^2-d x^2.
\end{equation}
But the latter take a very simple form in the standard coordinate trivialization of $\wt\Tsc{}^*X$ by $d t$, $d x^i$, $i=1,2,3$. We have (see \cite[Lemma 3.4]{HHV}):

\begin{lemma}
\label{LemmaKStNormal}
  Let $N_0=1$, $N_1=4$, $N_2=10$. For $g=g_{(\bhm,\bha)}$, we have
  \[
    \wh{\Box_{g,j}}(0) - \wh{\Box_{\ubar g,j}}(0) \in \rho^3\Diffb^2,
  \]
  where $\Box_{\ubar g}$ is the scalar wave operator on Minkowski space, given by $\Box_{\ubar g,j}=\Box_{\ubar g}\otimes 1_{N_j\times N_j}$ in the standard coordinate basis. Likewise,
  \[
    \delta_g^* - \delta_{\ubar g}^* \in \rho^2\Diffb^1, \quad
    \delta_{g_{(\bhm,\bha)}}^*-\delta_{g_{(\bhm,0)}}^* \in \rho^3\Diffb^1.
  \]
\end{lemma}
In the language of~\cite{MelroseAPS}, the \emph{normal operators} of $\wh{\Box_{g,j}}(0)$ and $\wh{\Box_{\ubar g,j}}(0)$ at $\pa_+ X$ are the same. 
\begin{TOOLONG}
\section{Main result}
\label{SecMT}
The main result of this paper is the 
 \begin{thm}
 Let $0<\bha<\bhm$, $\sigma\in\C$, $\Im\sigma\geq 0$, and suppose $\dot g=e^{-i\sigma t_*}\tilde{h},\, \tilde{h}\in \Hbext^{\infty,q}(X;S^{2}\wt{\Tsc^*X}),$ $q\in (-3/2;-1/2)$ is an outgoing mode solution of the linearized Einstein equation
  \begin{equation*}
     D_g\Ric(\dot g)=0.
  \end{equation*}
  Then there exist parameters $\dot\bhm\in\R$, $\dot\bha\in\R^3$, and an outgoing 1-form $\omega\in \Hbext^{\infty,q-1}(X;\wt{\Tsc}^*X)$, such that
  \begin{equation*}
    \dot g-\dot g_{(\bhm,\bha)}(\dot\bhm,\dot\bha)=\delta_g^*\omega.
  \end{equation*}
\end{thm}
We also obtain an analogous result for the linearized Einstein equation with fixed De Turck/wave map gauge. In that case more can be said about the pure gauge solutions which lie in a finite dimensional space characterized by asymptotic translations, asymptotic rotations and Coulomb type solutions of the $1-$form wave operator, see Theorem \ref{PropL0} for details. 
\end{TOOLONG} 

\section{The Teukolsky master equation}
\label{Sec4}
Perturbations of Kerr are significantly more complicated to analyze than perturbations of Schwarzschild. Real progress became possible when Teukolsky \cite{Te} wrote down his master equation describing gauge invariant perturbations of various spin. We begin by setting up the framework required for this formulation and proceed to discuss the absence of mode solutions.   

\subsection{The Teukolsky operator} \label{sec:teuk} 
We briefly review the construction of the Teukolsky operator, we refer e.g. to \cite{Ch}, \cite{GHP},  \cite{Te}  for details of the construction and to \cite{M} for a summary of the geometric background of the Teukolsky equation. 
Since $\cM$ is a non-compact,  oriented and time oriented Lorentzian 4-manifold, it admits a spin structure, i.e. an $\SL(2,\Co)$ principal bundle double covering the orthonormal frame bundle $\SO(\cM)$. 
Let $\cS, \bar \cS$ be the spinor bundle and its conjugate with fibers $\Co^2$ and $\bar\Co^2$, respectively. The isomorphism $\cS \otimes \bar \cS \to  T_{\Co} \cM$ yields a correspondence between tensors and spinors. Sections of $\cS$ are denoted with capital latin indices, eg. $\kappa^A$, while sections of $\bar\cS$ are denoted with primed indices, eg. $\varpi^{A'}$. 
The action of $\SL(2,\Co)$ on $\Co^2$ leaves a symplectic structure on $\cS$ invariant. The spin-metric $\epsilon_{AB} = \epsilon_{[AB]}$ is given by normalizing this symplectic form so that $g_{ab} = \epsilon_{AB} \epsilon_{A'B'}$. The spin-metric $\epsilon_{AB}$ and its inverse $\epsilon^{AB}$, where $\epsilon_{AB} \epsilon^{CB} = \delta_A{}^C$, are used to raise and lower spinor indices, eg. $\kappa^A \epsilon_{AB} = \kappa_B$. A spin-frame is a dyad $o^A, \iota^A$ normalized such that 
\begin{align} \label{eq:oiota}
o_A \iota^A = 1. 
\end{align} 
This corresponds to a complex null tetrad 
\begin{align} 
\ell^a = o^A o^{A'}, \quad n^a=\iota^A \iota^{A'}, \quad m^a =  o^A \iota^{A'}, \quad \bar m^a = \iota^A o^{A'},
\end{align}   
with 
\begin{align} 
g_{ab} = 2 (\ell_{(a} n_{b)} - m_{(a} \bar m_{b)}).
\end{align} 
The Newman-Penrose (NP) and Geroch-Held-Penrose (GHP) formalisms  represent spinor and tensor fields in terms of scalar dyad and tetrad components. Note that  \eqref{eq:oiota} is left invariant under rescalings 
\begin{align} \label{eq:rescale}
o^A \to \lambda o^A, \quad \iota^A \to \lambda^{-1} \iota^A 
\end{align} 
for a non-vanishing, complex scalar field $\lambda$ on $\cM$. A scalar $\phi$ constructed by projecting on a dyad or tetrad then transforms as 
\begin{align} \label{eq:proper}
\phi \to |\lambda|^{2w} \left ( \frac{\lambda}{\bar \lambda} \right ) ^s \phi
\end{align} 
where $w, s$ are the boost- and spin-weights of $\phi$, respectively. Scalars transforming as in \eqref{eq:proper} are termed properly weighted. Properly weighted scalars can be viewed as sections of complex line bundles $\mathcal{B}(s,w)$ with structure group $\Co_* = \Co \setminus \{0\}$. A spin dyad (or a complex null tetrad), provides a trivialization of $\mathcal{B}(s,w)$ and in this trivialization, sections of $\mathcal{B}(s,w)$ correspond to complex functions on $\cM$. 
\begin{TOOLONG} 

extends to tensorsGiven a spin dyad $o^A, \iota^A$ with conjugate dyad $o^{A'}, \iota^{A'}$ and a complex null tetrad $l^a, n^a, m^a, \bar m^a$ with $g_{ab} = 2(\ell_{(a} n_{b)} - m_{(a} \bar m_{b)})$, the Infeld-van der Waerden symbol or soldering form $\sigma_{AA'}^a$ provides a map $\cS \otimes \bar \cS \to T_{\Co} \cM$ defined by
Dualizing gives $\sigma^{AA'}_a$. 
The Infeld-van der Waerden symbol provides a correspondence between spinors and tensors. 
The action of the spin group $\SL(2,\Co)$ on $\cS$ leaves invariant a symplectic form $\epsilon_{AB}$ with conjugate $\epsilon_{A'B'}$. Normalizing $\epsilon_{AB}$ so that $g_{ab} = \sigma_a^{AA'} \sigma_b^{BB'} \epsilon_{AB} \epsilon_{A'B'}$  gives the spin metric. Choosing a normalized spin dyad with $\epsilon_{AB} o^A \iota^B = 1$ \mnote{LA: check sign}  

NNNN

Let $\cS$ be the spin bundle over $\cM$, $\epsilon$ the associated symplectic form, $j:\cS\otimes \bar{\cS}\rightarrow T_{\C}{\cM}$ the canonical isomorphism, $\cA$ the bundle of spin frames and $\cN$ the bundle of null frames.  
There is a natural double covering map $d:\cA\rightarrow \cN$, $(o, \iota) \mapsto (o\otimes \overline{o}, \iota\otimes \overline{\iota}, o\otimes \overline{\iota}).$ $\cA$ has two connected components $\cA_0,\, \cA_1$ and so does $\cN\,\, (\cN_0=d(\cA_0),\, \cN_1=d(\cA_1))$. $\cA_0$ is a $\C^*$ principal bundle. Let us consider the representation 
\[\rho_{s,w}:
\begin{cases}
\C^* \rightarrow GL(\C) \\
z \mapsto (a\mapsto z^{-w-s}\overline{z}^{-w+s}a).
\end{cases}
\]
We denote by $\cB(s,w)$ 
the vector bundle associated to $\cA_0$ and this representation. There is a natural identification between sections of $\mathcal{B}(s,w)$ and the set of complex valued functions $f$ defined on $\mathcal{A}_0$ such that for all $z\in \C^*$ 
\begin{equation}f(a \cdot z) = z^{w+s}\overline{z}^{w-s}f(a). \label{spinWeight}
\end{equation} 
To simplify the computations it is sometimes interesting to find a smaller principal bundle with a representation such that the associated vector bundle is isomorphic to $\mathcal{B}(s,w)$.
We can consider $\mathcal{A}_{0,r}:= \mathcal{A}_0/\R_+^*$ (we quotient by the usual action of $\R_+^*\subset \C^*$). Similarly we define $\mathcal{N}_{0,r} := \mathcal{N}_0/\R_+^*$. Then the map $d$ induces a double cover between $\mathcal{A}_{0,r}$ and $\mathcal{N}_{0,r}$ (we still call this induced map $d$). Moreover $\mathcal{A}_{0,r}$ and $\mathcal{N}_{0,r}$ have both a structure of $U(1)$ principal bundle over $\mathcal{M}$. 

The Levi Civita connection on the manifold lifts to a canonical connection $\Theta$ on $\cB(s,w)$. We now call $\boldsymbol{o}$ (resp. $\boldsymbol{\iota}$) the first (resp. second) projection from $\mathcal{A}_0$ to $\mathcal{S}$ and we define $\boldsymbol{l} := \boldsymbol{o}\otimes\overline{\boldsymbol{o}}$, $\boldsymbol{n} := \boldsymbol{\iota}\otimes\overline{\boldsymbol{\iota}}$ and $\boldsymbol{m} = \boldsymbol{o}\otimes\overline{\boldsymbol{\iota}}$. Note that thanks to the map $j$, $\boldsymbol{l}$, $\boldsymbol{m}$, $\boldsymbol{n}$ can be seen as $T_{\C}\mathcal{M}$ valued maps. \mnote{This sentence needs to be reformulated}This decomposition is often used after a first decomposition of the cospinor or cotensor into symmetric spinors (see \cite[Section 3.3]{PeRi1} for more details about this type of decomposition).

\end{TOOLONG} 
It is an important fact that the irreducible representations of $\SL(2,\Co)$ are given by symmetric spinors. Using the spinor-tensor correspondence already mentioned above, one may expand each tensor field in terms of symmetric spinors and $\epsilon$ factors. For example, the electromagnetic field strength tensor is a real 2-form $\mathrm{F}_{ab} = \mathrm{F}_{[ab]}$ and we have (see (3.4.20) in \cite{PeRi1} for details)
\begin{align} 
\mathrm{F}_{ab} = \phi_{AB} \epsilon_{A'B'} + \epsilon_{AB} \bar \phi_{A'B'}
\end{align} 
for a symmetric spinor field $\phi_{AB}$, the electromagnetic spinor. Given a dyad $o^A, \iota^A$ with corresponding tetrad $\ell^a, n^a, m^a, \bar m^a$, the Newman-Penrose scalars corresponding to $\phi_{AB}$ and $F_{ab}$ are 
\begin{subequations}
\label{Maxwell_scalars}
\begin{align}
\phi_1:=\phi_{AB} o^A o^B &= \mathrm{F}_{ab} l^a m^b, \\
\phi_0:=\phi_{AB} o^A \iota^B &= \frac{1}{2}\left( \mathrm{F}_{ab} \ell^a n^b + \mathrm{F}_{ab} \bar m^a m^b \right),\\
\phi_{-1}:=\phi_{AB} \iota^A \iota^B &= \mathrm{F}_{ab} \bar m^a n^b.
\end{align}
\end{subequations}
The notation used here does not conform to the convention used by Newman and Penrose and GHP, in that indices denote the spin(and boost)-weights of the Maxwell scalars $\phi_i$. This notational difference also applies to the $\Psi_i$ below. 

Similarly, we have the following decomposition for the Weyl tensor $W_{abcd}$ (see \cite[Eq. (4.6.41)]{PeRi1}):
\begin{align}
W_{abcd}  = \Psi_{ABCD} \epsilon_{A'B'} \epsilon_{C'D'} + \epsilon_{AB} \epsilon_{CD} \bar \Psi_{A'B'C'D'} \label{Weyl_spinor}
\end{align}
where the Weyl spinor, $\Psi_{ABCD}$, is a section of $(\mathcal{S}')^{\odot 4}$. We can compute the spin-weighted Weyl scalar components of $\Psi$ from components of $W$ ({\it c.f.} (4.11.6) and (4.11.9) in \cite{PeRi1}):
\begin{subequations} \label{Weyl_scalars}
\begin{align}
\Psi_2 &= \Psi_{ABCD} o^A o^B o^C o^D = W_{abcd} \ell^a m^b \ell^c m^d, \\
\Psi_1 &= \Psi_{ABCD} o^A o^B o^C \iota^D = W_{abcd} \ell^a n^b \ell^c m^d,  \\ 
\Psi_0 &= \Psi_{ABCD} o^a o^B \iota^C \iota^D = W_{abcd} \ell^a m^b \bar m^c n^d,  \\ 
\Psi_{-1} &= \Psi_{ABCD} o^A \iota^B \iota^C \iota^D = W_{abcd} \ell^a n^b \bar m^c n^d,  \\ 
\Psi_{-2} &= \Psi_{ABCD} \iota^A \iota^B\iota^C\iota^D = W_{abcd} n^a \bar m^b n^c \bar m^d.  
\end{align} 
\end{subequations} 
We remark that the scalars $\phi_i$ and $\Psi_i$ are all properly weighted, and each has boost-weight equal to its spin-weight. Since Kerr is Petrov type D, there are dyads or, equivalently, null tetrads, such that all Newman-Penrose Weyl scalars are zero, except $\Psi_0$. Any such dyad or tetrad is called principal. From now on, unless otherwise stated, we shall be working in a principal dyad $o^A, \iota^A$ on the Kerr background, and the corresponding principal null tetrad $\ell^a, n^a, m^a, \bar m^a$. The condition that the dyad is principal fixes the dyad  up to a rescaling of the form \eqref{eq:rescale}. If $\phi_{AB}$ is the Maxwell field on the charged Kerr-Newman spacetime, which is also Petrov type D, then in a principal tetrad, $\phi_0$ is the only non-vanishing Maxwell scalar. 

The Levi-Civita connection $\nabla_a$ defined with respect to $g_{ab}$ lifts to a connection $\Theta_a$ on $\mathcal{B}(s,w)$, the GHP connection.  See \cite[Eq. (2.14a)]{GHP}, \cite{Har}. 
We define the GHP operators $\text{\th}, \text{\th}', \text{\dh}, \text{\dh}'$, cf. \cite[Eq. (2.14)]{GHP}:  
\[
\begin{array}{cc}
\text{\th}:\begin{cases}
\Gamma(\mathcal{B}(s,w))\rightarrow \Gamma(\mathcal{B}(s,w+1))\\
u \mapsto \ell^a \Theta_a u
\end{cases}
&
\text{\th}':\begin{cases}
\Gamma(\mathcal{B}(s,w))\rightarrow \Gamma(\mathcal{B}(s,w-1))\\
u \mapsto n^a \Theta_a u
\end{cases}\\

\text{\dh}:\begin{cases}
\Gamma(\mathcal{B}(s,w))\rightarrow \Gamma(\mathcal{B}(s+1,w))\\
u \mapsto m^a \Theta_a u
\end{cases}
&
\text{\dh}':\begin{cases}
\Gamma(\mathcal{B}(s,w))\rightarrow \Gamma(\mathcal{B}(s-1,w))\\
u \mapsto \bar m^a \Theta_a u
\end{cases}
\end{array}
\]

and, in the notation of Remark 4.7 in \cite{M}, we put:
\[
G_s = (\text{\th}-2s\varrho-\bar{\varrho})(\text{\th}'-\varrho') - (\text{\dh} -\bar\tau' - 2s \tau)(\text{\dh}' -\tau'),
\]
where $\varrho, \varrho', \tau, \tau'$ are among the properly (spin- and boost-)weighted GHP spin coefficients, and are given by 
\begin{align}\label{NPspin_cos}
\varrho={}&m^a\bar{m}^b\nabla_b l_a,\quad \varrho'=\bar{m}^am^b\nabla_b n_a,\quad \tau=m^an^b\nabla_bl_a,\quad{\rm and}\quad \tau'=\bar{m}^al^b\nabla_bn_a,
\end{align}
cf. \cite[Eq. (2.3)]{GHP}.
\begin{TOOLONG} 
For later use in Section \ref{gauge_inv}, we also note that the spin coefficients defined in \eqref{NPspin_cos} occur as:
\begin{align}\label{NPspin_cos}
\varrho={}&-\boldsymbol{b(\overline{m})},\quad \varrho'=-\boldsymbol{c(m)},\quad \tau=-\boldsymbol{b(n)},\quad{\rm and}\quad \tau'=-\boldsymbol{c(l)}.
\end{align}
\end{TOOLONG}  
The Teukolsky operator is then given by
\[\check{T}_{s} := 2G_s-4(s-1)\left(s-\frac{1}{2}\right){\Psi}_0.\]

\subsection{Spherical symmetry}
\subsubsection{Spherical harmonic decompositions}
\label{SsYS}
In this section we collect some facts about some geometric operators acting on functions and sections of certain bundles on the standard sphere.

Let $\slg$ be the standard metric on $\Sph^2$, and denote geometric operators on $\Sph^2$ using a slash, thus $\sltr=\tr_\slg$, $\sldelta=\delta_\slg$, etc. We denote by $Y_{l m}$, $l\in\N_0$, $m\in\Z$, $|m|\leq l$, the usual spherical harmonics on $\Sph^2$ satisfying $\slDelta Y_{l m}=l(l+1)Y_{l m}$. Define the space 
\begin{equation}
\label{EqYS0}
  \scalspace_l := \mathspan\{Y_{l m} \colon |m|\leq l\}
\end{equation}
of degree $l$ spherical harmonics. Thus, $L^2(\Sph^2)=\bigoplus_{j\in\N_0}\scalspace_j$ is an orthogonal decomposition.

Consider next 1-forms on $\Sph^2$. Denote the Hodge Laplacian by $\slDelta_H=(\sld+\sldelta)^2$; the tensor Laplacian $\slDelta_{\slg,1}=-\sltr\slnabla^2$ (also denoted $\slDelta$ for brevity) satisfies $\slDelta=\slDelta_H-\Ric(\slg)=\slDelta_H-1$. Therefore, a spectral decomposition of $\slDelta$ on $L^2(\Sph^2;T^*\Sph^2)$ is provided by the scalar/vector decomposition
\begin{equation}
\label{EqYS1}
  \sld\scalspace_l,\ 
  \vectspace_l := \slstar\sld\scalspace_l \subset \ker\bigl(\slDelta-(l(l+1)-1)\bigr)\qquad (l\geq 1);
\end{equation}
note that $\sldelta\vectspace_l=0$, and that the two spaces in~\eqref{EqYS1} are trivial for $l=0$.

For symmetric 2-tensors finally, we have an analogous orthogonal decomposition into scalar and vector type symmetric 2-tensors: the scalar part consists of a pure trace and a trace-free part, the latter defined using the trace-free symmetric gradient $\sldelta_0^*:=\sldelta^*+\half\slg\sldelta$:
\begin{subequations}
\begin{equation}
\label{EqYS2scal}
  \scalspace_l\slg\ \ (l\geq 0),\qquad
  \sldelta_0^*\sld\scalspace_l\ \ (l\geq 2).
\end{equation}
(Note here that for $\scal\in\scalspace_0\oplus\scalspace_1$, we have $\sldelta_0^*\sld\scal=0$, hence the restriction to $l\geq 2$.) The vector part consists only of trace-free tensors with $l\geq 2$ (since the 1-forms in $\vectspace_1$ are Killing),
\begin{equation}
\label{EqYS2vect}
  \sldelta^*\vectspace_l\ \ (l\geq 2).
\end{equation}
\end{subequations}

The geometric operators on $\Sph^2$ which we will encounter here preserve scalar and vector type spherical harmonics; indeed, this holds in the strong sense that a scalar type function/1-form/symmetric 2-tensor built out of a particular $\scal\in\scalspace_l$ is mapped into another scalar type tensor \emph{with the same $\scal$}, likewise for vector type tensors; this is clear for $\sld$ on functions, $\sldelta$ on 1-forms ($\sldelta(\sld\scal)=l(l+1)\scal$). Furthermore, for $\scal\in\scalspace_l$ and $\vect\in\vectspace_l$,
\begin{gather*}
  \sldelta^*(\sld\scal) = -\tfrac{l(l+1)}{2}\scal\slg + \sldelta_0^*\sld\scal, \quad
  \sldelta(\scal\slg)=-\sld\scal,\quad
  \sldelta(\sldelta_0^*\sld\scal)=\tfrac{l(l+1)-2}{2}\sld\scal, \\
  \sldelta\sldelta^*\vect = \tfrac{l(l+1)-2}{2}\vect, \quad
  \slDelta(\sldelta_0^*\sld\scal) = (l(l+1)-4)\sldelta_0^*\sld\scal, \quad
  \slDelta(\sldelta^*\vect) = (l(l+1)-4)\sldelta^*\vect.
\end{gather*}
\subsubsection{A spin weighted spherical Laplacian} 
Let $H:\Sph^3\rightarrow \Sph^2$ be the Hopf bundle. This is a $U(1)$ principal bundle. We consider the following representation of $U(1)$
\begin{align*}
\rho_s:\left\{\begin{array}{rcl} U(1)&\rightarrow&Gl(\C),\\
z&\mapsto &(a\mapsto z^{-2s} a).\end{array}\right. 
\end{align*}
Let $\slashed{\mathcal B}(s)$ be the complex line bundle associated to $H$ and this representation. We have an identification between sections of $\slashed{\mathcal B}(s)$ and complex functions on $\Sph^3$ such that $f(a\cdot z)=z^{2s}f(a)$, where $\cdot$ is the right action of $U(1)$ on $\Sph^3$. The Levi-Civita connection on $\Sph^2$ lifts to a connection $\slashed{\Theta}$ on $\slashed{\mathcal B}(s)$. Let $\slashed{m}^a,\slashed{\bar{m}}^a$ be a complex dyad on $\Sph^2: \slashed{m}^a\slashed{\bar{m}}_a=1$.  This then defines spherical  edth operators $\mathring{\text{\dh}},\, \mathring{\text{\dh}}'$,
  \begin{align} 
\mathring{\text{\dh}}: \left\{\begin{array}{rcl}
\Gamma(\slashed{\mathcal{B}}(s))&\rightarrow& \Gamma(\slashed{\mathcal{B}}(s+1))\\ u&\mapsto&\slashed{m}^a\slashed{\Theta}_au,\end{array}\right. \\ 
\mathring{\text{\dh}}': \left\{\begin{array}{rcl}
\Gamma(\slashed{\mathcal{B}}(s))&\rightarrow& \Gamma(\slashed{\mathcal{B}}(s-1))\\ u&\mapsto&\slashed{\bar{m}}^a\slashed{\Theta}_au.\end{array}\right. 
\end{align} 
Putting 
\begin{align*}
\slashed{m}^a=\frac{1}{\sqrt{2}}(\partial_{\theta}+\frac{i}{\sin\theta}\partial_{\varphi}),\quad \slashed{\bar{m}}^a=\frac{1}{\sqrt{2}}(\partial_{\theta}-\frac{i}{\sin\theta}\partial_{\varphi}),
\end{align*}
$\mathring{\text{\dh}}$ and $\mathring{\text{\dh}}'$ can be computed in a standard trivialization (see \cite{GMNRS}\footnote{The factor $\frac{1}{\sqrt{2}}$ comes from a useful normalization.}):
\begin{align*}
\mathring{\text{\dh}}=\frac{1}{\sqrt{2}}(\partial_{\theta}+\frac{i}{\sin\theta}\partial_{\varphi}-s\cot\theta),\\
\mathring{\text{\dh}}'=\frac{1}{\sqrt{2}}(\partial_{\theta}-\frac{i}{\sin\theta}\partial_{\varphi}+s\cot\theta). 
\end{align*}
Note: these are identical to the operators introduced in (1.16c-d) of \cite{ABBM} for the conformally rescaled metric on the asymptotic sphere at null infinity.
The operator   
\begin{align} 
- \slDelta^{[s]} := -2\mathring{\text{\dh}}' \mathring{\text{\dh}} : \Gamma(\slashed{\mathcal{B}}(s)) \to \Gamma(\slashed{\mathcal{B}}(s))
\end{align} 
is a second order, spin weighted, linear, spherical operator. In the local trivialization above it takes the form 
\begin{align} 
- \slDelta^{[s]}  = \frac{1}{\sin^2\theta}D_{\phi}^2+\frac{1}{\sin\theta}D_{\theta}\sin \theta D_{\theta}+2s\frac{\cos\theta}{\sin^2\theta}D_{\phi}+s^2\cot^2\theta-s.
\end{align} 
Here we have used $D_{\bullet}=\frac{1}{i}\partial_{\bullet}$. $- \slDelta^{[s]}$ has to be considered as an operator acting on $L^2(\Sph^2;\slashed{\mathcal B}(s))$. $(- \slDelta^{[s]}, H^2(\Sph^2; \slashed{\mathcal B}(s))$ is selfadjoint. It has  a discrete spectrum with eigenvalues ($l\ge \vert s\vert$), 
\begin{align} 
\lambda_{l}^{[s]}=l(l+1)-s(s+1).
\end{align} 
The eigenfunctions are of the form $Y_{l m}^{[s]}=S_{l}^{[s]}(\cos\theta)e^{im \phi_*}\, (m,\, \vert s\vert\le l)$. We then put 
\begin{align*}
\scalspace^{[s]}_{l} := \mathspan\{Y_{l m}^{[s]} \colon |m|,\, \vert s\vert\leq l\}.
\end{align*}
This gives an orthogonal decomposition 
\begin{equation}
\label{decomp1}
L^2(\Sph^2;\mathring{\mathcal{B}}(s))=\bigoplus_{l\ge \vert s\vert} \scalspace^{[s]}_{l}.
\end{equation}
\begin{rmk}

As an alternative, one could consider $-2\mathring{\text{\dh}} \mathring{\text{\dh}}'$ (see (1.13) in \cite{ABBM}), or one might instead consider the symmetric $- (\mathring{\text{\dh}} \mathring{\text{\dh}}'+\mathring{\text{\dh}}' \mathring{\text{\dh}})$, as the two dimensional operator associated with the restriction to the sphere of the connection $\slashed{\Theta}$. However $-\slashed{\Delta}^{[s]}$ is the operator which we identify in the Teukolsky equation. All these operators are only shifts one with respect to another; more precisely $-\mathring{\text{\dh}} \mathring{\text{\dh}}'=-\mathring{\text{\dh}}' \mathring{\text{\dh}}+s$. 
\begin{TOOLONG}
\item The link to Section \ref{sec:teuk} is the following. The bundle ${\mathcal B}(s,0)$ is isomorphic to $\R_t\times\R_{r>r_+}\times \slashed{\mathcal B}(s)$. A section $Z$ of $\slashed{\mathcal B}(s)$ can be considered as a section $\tilde{Z}$ of ${\mathcal B}(s,0)$ which is constant in $(t,r)$. We then have $\slashed{\Theta}Z=\Theta\tilde{Z}$. To compute the operators $\text{\dh},\, \text{\dh}'$ in a concrete trivialization we have to choose a tetrad. A natural choice is Kinersley's tetrad, see \eqref{principalVectors}-\eqref{vectorM}. In this tetrad we have $m^a=\frac{1}{r+ia\cos\theta}\slashed{m}^a$ and thus 
\begin{align*}
(\mathring{\text{\dh}}Z)(\omega)=((1+ia\cos\theta){\text{\dh}}\tilde{Z})(1,1,\omega),\quad \forall \omega\in \Sph^2. 
\end{align*}
\end{enumerate}
\end{TOOLONG}
\end{rmk}

\subsection{The Teukolsky equation associated to the Maxwell field and linearized gravity} 
\label{SecTeu}
Let $\boldsymbol{\Phi}_0$ be a solution of the scalar wave equation: $\Box_{g,0}\boldsymbol{\Phi}_0:=-\nabla^\mu\nabla_\mu\boldsymbol{\Phi}_0=0$.
Let $F$ be an antisymmetric two tensor fulfilling the Maxwell equations on the Kerr metric
\begin{equation}
\label{Maxwell}
dF=0,\quad \delta_{g} F=0
\end{equation}
and let $\boldsymbol{\Phi}_{\pm 1} \equiv \phi_{\pm 1}$, where $\phi_{\pm 1}$ were introduced in \eqref{Maxwell_scalars}.  
Finally, let $h$ be a symmetric two tensor (such as the $\dot{g}$ of equation \eqref{gdoteqn}), fulfilling the linearized Einstein equations around a Kerr solution :
\begin{equation}
\label{LinEinst}
D_{g}Ric(h)=0.
\end{equation}
Let $\htf_{ab}, \slashed{h}$ denote the tracefree and trace parts of $h_{ab}$.
Let $h_{ABA'B'}$ be the spinor corresponding to $h_{ab}$. We have $\htf_{ABA'B'} = h_{(AB)(A'B')}$, $\slashed{h} =  h_A{}^A{}_{B'}{}^{B'}$. 
We shall make use of the spinor variational operator $\vartheta$ introduced in \cite{BVK}. The linearized Weyl spinor is  \cite[Eq. (45)]{2019JMP....60h2501A}
\begin{align} 
\vartheta \Psi_{ABCD} = \half \nabla_{(A}^{A'} \nabla_{B}^{B'} \htf_{CD)A'B'} + \tfrac{1}{4} \slashed{h} \Psi_{ABCD}.
\end{align}  
The operator $\vartheta$ eliminates the dependence on the variation of the tetrad which otherwise is present when one varies the Newman-Penrose scalars. Let $\dot W_{abcd} = (D_g \Weyl(h))_{abcd}$ be the linearized Weyl tensor. 
The Teukolsky equation involves only the linearization of the Weyl  scalars with extreme spin weights 
$\Psi_{\pm 2}$, as given in \eqref{Weyl_scalars}. 
We define:

\begin{subequations} \label{varPsis}
\begin{align}
\boldsymbol{\Phi}_{2}:=~& \vartheta\Psi_2 = \vartheta \Psi_{ABCD} o^A o^B o^C o^D = \dot W_{abcd} \ell^a m^b \ell^c m^d, \\
& \vartheta\Psi_1 = \vartheta\Psi_{ABCD} o^A o^B o^C \iota^D = \dot W_{abcd} \ell^a n^b \ell^c m^d,  \\ 
& \vartheta\Psi_0 = \vartheta\Psi_{ABCD} o^A o^B \iota^C \iota^D = \dot W_{abcd} \ell^a m^b \bar m^c n^d,  \\ 
& \vartheta\Psi_{-1} = \vartheta\Psi_{ABCD} o^A \iota^B \iota^C \iota^D = \dot W_{abcd} \ell^a n^b \bar m^c n^d,  \\ 
\boldsymbol{\Phi}_{-2}:=~&\vartheta \Psi_{-2} = \vartheta \Psi_{ABCD} \iota^A \iota^B \iota^C \iota^D  = \dot W_{abcd} n^a \bar m^b n^c \bar m^d, 
\end{align}
\end{subequations}
for the linearized gravitational field. 
Let 
\begin{align} 
\bhp=r-i\bha\cos \theta 
\end{align} 
 be a 0-(spin and boost)weighted scalar.  We now put
 \begin{align*} 
\boldsymbol{ \Psi}_{[s]}=\left\{\begin{array}{cc} \boldsymbol{\Phi}_{s} & \mbox{if}\, s\ge 0,\\
\bhp^{-2s}\boldsymbol{\Phi}_{s}& \mbox{if}\, s<0.\end{array}\right.
 \end{align*}
 $\boldsymbol{ \Psi}_{[s]}$ then fulfills the Teukolsky equation\footnote{With suitable definitions, the Teukolsky equation is also satisfied by fields of half-integer spin.  These will not be needed here.} 
 \begin{align*}
 \check{T}_s\boldsymbol{ \Psi}_{[s]}=0.
 \end{align*}

We now choose a concrete tetrad, the Kinnersley tetrad \cite{Ki}, given by 
\begin{align}
\label{principalVectors}
\ell &= \frac{r^2+\bha^2}{\Delta_b}\partial_t + \partial_r + \frac{\bha}{\Delta_b}\partial_\phi, \\
\label{principalVectors2}
n &= \frac{r^2 + \bha^2}{2\varrho_b^2} \partial_t - \frac{\Delta_b}{2\varrho_b^2} \partial_r + \frac{\bha}{2\varrho_b^2}\partial_\phi, \\
\label{vectorM}
m &= \frac{i\bha\sin\theta}{\sqrt{2}\bar{\bhp}}\partial_t + \frac{1}{\sqrt{2}\bar{\bhp}}\partial_\theta + \frac{i}{\sqrt{2}\bar{\bhp}\sin\theta}\partial_{\phi} . 
\end{align}
As explained above, the choice of tetrad provides a local trivialization of $\mathcal{B}(s,w)$ and 
$T_s=\varrho_b^2 \check{T}_s$ takes the form 
 \begin{align*}
T_s =& \left(\frac{(r^2+\bha^2)^2}{\Delta_b} - \bha^2 \sin^2\theta\right)\partial_t^2 + \frac{4\bhm \bha r}{\Delta_b}\partial_t\partial_\phi + \left(\frac{\bha^2}{\Delta_b} - \frac{1}{\sin^2\theta}\right)\partial_\phi^2\\
& - \Delta_b^{-s}\partial_r\left(\Delta_b^{s+1}\partial_r\right)
 - \frac{1}{\sin\theta}\partial_\theta \left(\sin\theta \partial_\theta\right) - 2s\left(\frac{\bha(r-\bhm)}{\Delta_b}+\frac{i\cos\theta}{\sin^2\theta}\right)\partial_\phi\\
 & - 2s\left(\frac{\bhm(r^2-\bha^2)}{\Delta_b} - r - i \bha \cos \theta\right)\partial_t + \left(s^2 \cot^2\theta - s\right).
\end{align*}
We can now write the Teukolsky equation as 
\begin{align}
\label{TME}
T_s\boldsymbol{\Psi}_{[s]}=0.
\end{align}

\subsection{The Teukolsky equation in a normalized tetrad}
\label{Sec4.3}
We now normalize the tetrad by performing a boost rotation
\[
\tilde{\ell}^{a}=\Delta_b \ell^{a},\, \tilde{n}^{a}=\frac{1}{\Delta_b}n^{a},\, \tilde{m}^{a}=m^{a}. 
\]
Recall that, for the Maxwell and Weyl scalars under consideration, the spin- and boost-weight coincide. A component of spin weight $s$ in this new tetrad then satisfies $\tilde{\Psi}_{[s]}=\Delta_b^s\Psi_{[s]}$. We will also use the $(t_*,r,\varphi_*,\theta)$ coordinate system.
Let us put $\tilde{T}_s=\Delta_b^{s}T_s\Delta_b^{-s}$. Using also $\rho=\frac{1}{r}$ (not to be confused with an NP spin coefficient) we find for $r\ge 4\bhm$:
\begin{eqnarray*}
\tilde{T}_s=&-\bha^2\sin^2\theta\partial_{t_*}^2+\frac{4\bhm \bha r}{\Delta_b}\partial^2_{t_*\varphi_*}+\left[\frac{\bha^2}{\Delta_b}-\frac{1}{\sin^2\theta}\right]\partial^2_{\varphi_*}-\rho^2\partial_{\rho}\Delta_b^{s+1}\rho^2\partial_{\rho}\Delta_b^{-s}\nonumber\\
&+2r\partial_{t_*}-2(r^2+\bha^2)\rho^2\partial_{\rho}\partial_{t_*}-2s(r-\bhm)\frac{r^2+a^2}{\Delta_b}\partial_{t_*}-\frac{1}{\sin\theta}\partial_{\theta}\sin\theta\partial_{\theta}\nonumber\\
&-2s\left[\frac{\bha(r-\bhm)}{\Delta_b}+\frac{i\cos\theta}{\sin^2\theta}\right]
\partial_{\varphi_*}-2s\left[\frac{\bhm(r^2-\bha^2)}{\Delta_b}-r-i\bha\cos\theta\right]\partial_{t_*}\nonumber\\
&+(s^2\cot^2\theta-s).
\end{eqnarray*}
For $r\le 3\bhm$ we have 
\begin{eqnarray*}
\tilde{T}_s=&-\bha^2\sin^2\theta\partial_{t_*}^2-\frac{1}{\sin^2\theta}\partial^2_{\varphi_*}-\frac{1}{\sin\theta}\partial_{\theta}\sin\theta\partial_{\theta}-\partial_r\Delta_b^{s+1}\partial_{r}\Delta_b^{-s}\nonumber\\
&-2\bha\partial^2_{t_*\varphi_*}-2\bha\partial_r\partial_{\varphi_*}-(r^2+\bha^2)\partial_{t_*}\partial_r-\partial_{r}(r^2+\bha^2)\partial_{t_*}
\nonumber\\
&-2s\frac{i\cos\theta}{\sin^2\theta}
\partial_{\varphi_*}+2s\left[r+i\bha\cos\theta\right]\partial_{t_*}+(s^2\cot^2\theta-s).
\end{eqnarray*} 
We will often consider the rescaled and Fourier transformed operator 
\begin{align*}\hat{T}_s(\sigma)=\rho^2e^{it_*\sigma}\tilde{T}_se^{-it_*\sigma}.
\end{align*} 
\begin{rmk}
Note that in the exterior region $\hat{T}_s(\sigma)=e^{i\sigma F(r)}(\cF_t(\rho^2\tilde{T}_s))(\sigma)e^{-i\sigma F(r)}$, where $\cF_t$ is the Fourier transformed operator with respect to the time variable $t$. In particular both operators coincide for $\sigma=0$. 
\end{rmk}
The Sobolev spaces we will work with are the b$-$Sobolev spaces for sections of $\cB(s,s)$, $\Hbext^{m,q}(X;\cB(s,s))$. We will often write $\Hbext^{m,q}(X),\, \cA^q(X)$ instead of  $\Hbext^{m,q}(X;\cB(s,s)),$ $ \cA^q(X;\cB(s,s))$. 

The bundle $\cB(s,s)$ is the cartesian product of the trivial bundle $Id_{\R_{t_*}\times [r_0,\infty)}$ and $\slashed{\mathcal B}(s)$. 
Recall that $-\slDelta^{[s]}$ acts on sections of $\slashed{\mathcal B}(s)$,
\begin{align*}
-\slDelta^{[s]}: H^2(\Sph^2;\slashed{\mathcal B}(s))\rightarrow L^2(\Sph^2;\slashed{\mathcal B}(s))
\end{align*}
has a discrete spectrum with eigenvalues ($l\ge \vert s\vert$),
\[\lambda_{l}^{[s]}=l(l+1)-s(s+1)\]
and that we have the natural decomposition 
\begin{equation}
\label{decomp1}
L^2(\Sph^2;\slashed{\mathcal B}(s))=\bigoplus_{l\ge \vert s\vert} \scalspace^{[s]}_{l},
\end{equation}
where $\scalspace^{[s]}_{l}$ is the eigenspace corresponding to the eigenvalue $\lambda_{l}^{[s]}$. This entails :
\begin{align*}
\Hbext^{m,q}(X)=\bigoplus_{l\ge \vert s\vert} \Hbext^{m,q}\left(\left[0,\frac{1}{r_0}\right),\frac{d\rho}{\rho^4}\right)\otimes \scalspace^{[s]}_{l}.
\end{align*}
We put 
\begin{align*}
\Hbext^{m,q,l}(X)=\Hbext^{m,q}\left(\left[0,\frac{1}{r_0}\right),\frac{d\rho}{\rho^4}\right)\otimes \scalspace^{[s]}_{l}.
\end{align*}

We will consider the Teukolsky operator $\hat{T}_s(\sigma)$ as an operator 
\begin{align*}
\hat{T}_s(\sigma):\{\Psi \in \Hbext^{m,q(s)};\, \hat{T}_s(\sigma)\Psi \in \Hbext^{m-1,q(s)+2}\}\rightarrow \Hbext^{m-1,q(s)+2},\\
\hat{T}_s(0):\{\Psi \in \Hbext^{m,q(s)};\, \hat{T}_s(0)\Psi \in \Hbext^{m-1,q(s)+2}\}\rightarrow \Hbext^{m-1,q(s)+2}. 
\end{align*}
The explicit form of the Teukolsky operator shows that $\hat{T}_s(0)$ preserves the spaces $\scalspace^{[s]}_{l}$. We can therefore consider the Teukolsky operator also as an operator 
\begin{align*}
\hat{T}_s(0):\{\Psi \in \Hbext^{m,q(s),l};\, \hat{T}_s(0)\Psi \in \Hbext^{m-1,q(s)+2,l}\}\rightarrow \Hbext^{m-1,q(s)+2,l}. 
\end{align*}
Note also that $\hat{\Psi}_{[s]}=\sum_{l}\hat{\Psi}_{[s]}^{l},\, \hat{\Psi}^{l}_{[s]}\in \Hbext^{m,q(s),l}$ is a solution of $\hat{T}_s(0)\hat{\Psi}_{[s]}=0$ if and only if 
$\hat{T}_s(0)\hat{\Psi}^{l}_{[s]}=0$ for all ${l}$.

\subsection{Asymptotic behavior}
To analyze the asymptotic behavior of the solutions of the Teukolsky equation, we will use normal operator arguments. An important role will be played by the Mellin transform and its inverse, see e.g. \cite[Section 3.1]{VasyMicroKerrdS}. 
Let $M_I=[0,\infty)\times \Sph^2$. The Mellin transform is defined by 
\begin{align*}
({\mathcal M}\phi)(\lambda,\omega)&=\int_0^{\infty}\rho^{-i\lambda}\phi(\rho,\omega)\frac{d\rho}{\rho}
\end{align*}
and it gives an isometric isomorphism of $L^2(M_I;\frac{d\rho}{\rho}d\omega)$ with $L^2(\R_{\lambda};L^2(\Sph^2;d\omega))$. Noting that $L^2(M_I;\frac{d\rho}{\rho}d\omega)=\rho^{-3/2}L^2(M_I;\frac{d\rho}{\rho^4}d\omega)$ we obtain an isomorphism of  $\rho^{-3/2}L^2(M_I;\frac{d\rho}{\rho^4}d\omega)$ with $L^2(\R_{\lambda};L^2(\Sph^2;d\omega))$. The inverse Mellin transform is given by 
\begin{align*}
({\mathcal M^{-1}}\psi)(\rho,\omega)&=\int_{\R}\rho^{i\lambda}\psi(\lambda,\omega)d\lambda.
\end{align*}
More generally for $\phi\in \rho^qL^2_b(M_I;\frac{d\rho}{\rho^4}d\omega)=\rho^{q+3/2}L^2_b(M_I;\frac{d\rho}{\rho}d\omega)$, $\cM\phi(\bullet-i(q+3/2))$ is well defined as an element of $L^2(\R_{\lambda};L^2(\Sph^2d\omega))$ and its inverse Mellin transform becomes 
\begin{align*}
({\mathcal M^{-1}}\psi)(\rho,\omega)&=\int_{-\infty-i(q+3/2)}^{\infty-i(q+3/2)}\rho^{i\lambda}\psi(\lambda,\omega)d\lambda.
\end{align*}
If $\phi\in \rho^{q}L^2_b(M_I;\frac{d\rho}{\rho^4}d\omega)$ has compact support in $\rho$, the Mellin transform $\cM\phi$ extends to a holomorphic function in ${\Im } \lambda>-(q+3/2)$ with values in $L^2(\Sph^2;d\omega)$, satisfying 
\begin{align*}
\sup_{-(q+3/2)<\alpha<C}\Vert \cM\phi(\bullet+i\alpha)\Vert_{L^2(\R_{\lambda};L^2(\Sph^2))}<\infty
\end{align*} 
for all $C<\infty$, and $\cM\phi(\bullet+i\alpha)$ extends continuously to $\alpha=-(q+3/2)$ in the topology of $L^2(\R_{\lambda};L^2(\Sph^2)$. Since $\cM$ intertwines differentiation $\rho D_{\rho}$ and multiplication by $\lambda$, we obtain similar statements for weighted $b-$ Sobolev spaces, merely
\begin{align*}
\Hbext^{b,q}(M_I;\frac{d\rho}{\rho^4}d\omega)\ni\phi\mapsto \cM\phi(\bullet-i(q+3/2))\in \bigcap_{j=0}^k \la \lambda\ra^{-j} L^2(\R_{\lambda};H^{k-j}(\Sph^2)). 
\end{align*}

\subsubsection{Polyhomogeneous expansion}

\begin{prop}
\label{prop3.1}
Let $q(s)=q-2s$ if $s\ge 0$ and $q(s)=q$ if $s<0$.  
\begin{enumerate}
\item Let $m>1/2+s,\, q\in(-3/2,-1/2),\, m+q(s)>-\half-2s$. Suppose $\hat{\Psi}_{[s]}\in \Hbext^{m,q(s)}(X),$ $\hat{\Psi}_{[s]}=\sum_{l}\hat{\Psi}^{l}_{[s]},\, \hat{\Psi}^{l}_{[s]}\in \Hbext^{m,q(s),l}(X)$. If 
\begin{equation*}
\hat{T}_s(0)\hat{\Psi}_{[s]}=0,
\end{equation*} 
then there exist $\hat{\Psi}^{0,l}_{[s]}\in \scalspace^{[s]}_{l} $ such that.
\begin{align*}
\hat{\Psi}^{l}_{[s]}-\rho^{l+1-s}\hat{\Psi}_{[s]}^{0,l}\in \cA^{(l+2-s)-}
\end{align*}
In particular we have 
\begin{align*}
\hat{\Psi}_{[s]}-\rho^{w(s)}\hat{\Psi}_{[s]}^{0}\in \cA^{(w(s)+1)-},
\end{align*}
$w(s)=1$ if $s\ge 0,\, w(s)=1-2s$ if $s<0$ and $\hat{\Psi}_{[s]}^0\in \scalspace^{[s]}_{\vert s\vert} $. 
\item Let $\Im\,\sigma\ge 0,\,\sigma\neq 0,\, q(s)<-\half,\, m>1/2+s,\, \, m+q(s)>-\half-2s.$
Suppose that $\hat{\Psi}_{[s]}\in {\rm Ker}\hat{T}_s(\sigma)\cap \Hbext^{m,q(s)}(X)$. Then there exists $\hat{\Psi}^0_{[s]}\in C^{\infty}(\Sph^2;\slashed{\mathcal B}(s))$ such that 
\[\hat{\Psi}_{[s]}-\rho\hat{\Psi}^0_{[s]}\in \cA^{2-}.\] 
\end{enumerate}
\end{prop} 
\begin{proof}
\begin{enumerate}
\item As noted at the end of Section \ref{Sec4.3} each $\hat{\Psi}_{[s]}^{l}$ is a solution of $\hat{T}_{[s]}(0)\hat{\Psi}^{l}_{[s]}=0$. 
Now note that $\hat{\Psi}^{l}_{[s]}\in \Hbext^{\infty,q(s),l}(X)$. This follows by elliptic regularity, propagation of regularity at the radial sets at the horizons and real principal type propagation. We refer to \cite{HHV}, \cite{M2} for details. The normal operator at $\rho=0$ is given by 
\[\cN=\rho^2(-(\rho\pa_{\rho})^2+(1-2s)\rho\pa_{\rho}-\Delta^{[s]}+2s).\]
Restricted to $\scalspace^{[s]}_{l}$, $-\Delta^{[s]}$ acts by multiplication with $\lambda_{[s]}^{l}$. To compute the boundary spectrum we Mellin transform the equation. We find the indicial equation 
\begin{equation}
\label{ind0}
P(\lambda)=\lambda^2+i(1-2s)\lambda+l(l+1)-s(s-1)=0,
\end{equation}
which has the roots $i(s+l)$ and $-i(l+1-s)$. Here $l\ge \vert s\vert$. To understand the asymptotic behavior of $\hat{\Psi}^{l}_{[s]}$ we compute 
\begin{eqnarray}
\label{4.11a}
\rho^{-2}\cN(\chi \hat{\Psi}^{l}_{[s]})&=&\rho^{-2}(\cN-\hat{T}_s(0))(\chi\hat{\Psi}^{l}_{[s]})+\rho^{-2}[\hat{T}_s(0),\chi]\hat{\Psi}^{l}_{[s]}\in {}\Hbext^{\infty,q(s)+1,l}, 
\end{eqnarray}
where $\chi$ is a cut-off, identically $1$ near $\pa_+X$.  We Mellin transform this equation, divide by $(\lambda-i(s+l))(\lambda+(i(l+1-s))$ and integrate along $\Im\, \lambda=-(q(s)+5/2)$, for the inverse Mellin transform\footnote{If $q(s)+5/2=l+1-s$ we replace $q(s)+5/2$ by $q(s)+5/2-\epsilon$ for $\epsilon>0$ small.}. Shifting the contour and using Cauchy's formula gives 
\begin{align}
\label{pexpsi}
\hat{\Psi}^{l}_{[s]}=c^{l}_0\rho^{-(s+l)}+c^{l}_1\rho^{l+1-s}+v,\, c^{l}_0,\, c^{l}_1\in  \scalspace^{[s]}_{l},\, v\in \Hbext^{\infty,q(s)+1,l}.
\end{align}
As $q(s)+3/2>-(l+s)$ we have $c_0^{l}\rho^{-(s+l)}\notin \Hbext^{\infty,q(s),l}$ for $c_0^{l}\neq 0$ and thus $c^{l}_0=0$. This gives 
\begin{align}
\label{4.9}
\hat{\Psi}^{l}_{[s]}-c^{l}_1\rho^{l+1-s}\in \Hbext^{\infty,q(s)+1,l}.
\end{align} 
Now either $q(s)+1>l+1/2-s$ and we are done because $\Hbext^{\infty, l+1/2-s}\subset \cA^{(l+2-s)-}$ or we have to repeat the argument.  Therefore let us suppose that we have 
\begin{align}
\label{4.9n}
\hat{\Psi}^{l}_{[s]}-c^{l}_n\rho^{l+1-s}\in \Hbext^{\infty,q(s)+n,l}
\end{align} 
after $n$ steps. We now apply the same procedure to $\hat{\Psi}^{l}_{[s]}-c^{l}_n\rho^{l+1-s}$. We have 
\begin{align*}
\rho^{-2}\cN(\chi(\hat{\Psi}^{l}_{[s]}-c^{l}_n\rho^{l+1-s}))&=\rho^{-2}(\cN-\hat{T}_s(0))(\chi(\hat{\Psi}^{l}_{[s]}-c^{l}_n\rho^{l+1-s}))\\
&+\rho^{-2}(\hat{T}_s(0)-\cN)\chi c^{l}_n\rho^{l+1-s}\\
&-c^l_n\rho^{-2}[\cN,\chi]\rho^{l+1-s}
+\rho^{-2}[\hat{T}_s(0),\chi]\hat{\Psi}^{l}_{[s]}.
\end{align*}
We have 
\begin{align*}
&\rho^{-2}(\cN-\hat{T}_s(0))(\chi(\hat{\Psi}^{l}_{[s]}-c^{l}_n\rho^{l+1-s}))\in \Hbext^{\infty,\min\{l+1/2-s,q(s)+n+1\}-,l},\\
&\rho^{-2}(\cN-\hat{T}_{[s]}(0))\chi c^{l}_n\rho^{l+1-s}\in \cA^{(l-s+2)},\\
&c^{l}_1\rho^{-2}[\cN,\chi]\rho^{l+1-s},\, \rho^{-2}[\hat{T}_s(0),\chi]\hat{\Psi}^{l}_{[s]}\in \Hbext^{\infty,\infty,l}. 
\end{align*}
Now either $q(s)+n+1>l+1/2-s$ and we are done or we obtain \eqref{4.9n} with $n$ replaced by $n+1$. By induction we obtain the result after a finite number of steps.

Let us now prove the second claim in (1). By the same argument as before we know that $\hat{\Psi}_{[s]}\in \Hbext^{\infty,q(s)}$ and we have the estimate 
\begin{align*}
\Vert \hat{\Psi}_{[s]}\Vert_{\Hbext^{\tilde{m},q(s)}}\le C_{\tilde{m}} \Vert \hat{\Psi}_{[s]}\Vert_{\Hbext^{m,q(s)}},\quad \forall \tilde{m}\ge m. 
\end{align*}
By restriction we obtain the same estimate for $\hat{\Psi}^l_{[s]}$. Now we first repeat the same procedure as before for $\hat{\Psi}^{l}_{[s]},\, \vert s\vert\le l\le \vert s\vert+1$. Let then $l\ge\vert s\vert+2$. 
For those $l$ we repeat the argument $n$ times. As long as 
\begin{align*}
q(s)+3/2+n< l+1-s
\end{align*}
$c_0^l$ and $c_1^l$ in \eqref{pexpsi} are both zero. We obtain sufficient decay by this if 
\begin{align*}
q(s)+3/2+n\ge w(s)+1.
\end{align*}
Both are fulfilled if 
\begin{align*}
1/2+2\vert s\vert-q\le n< l+\vert s\vert-1/2-q,
\end{align*}
which can be arranged because $l\ge \vert s\vert+2$. We will check that we obtain estimates which are uniform in $l$. 
Let
\begin{align*}
\tilde{f}^l(\lambda)=\cM(\rho^{-2}(\cN-\hat{T}_s(0))(\chi\hat{\Psi}^{l}_{[s]})+\rho^{-2}[\hat{T}_s(0),\chi]\hat{\Psi}^{l}_{[s]})(\lambda). 
\end{align*}
We then have 
\begin{align*}
\hat{\Psi}^l_{[s]}=\cM^{-1}(P^{-1}(\lambda)\tilde{f}^l(\lambda)). 
\end{align*}
We now use Parseval's identity, the fact that $P^{-1}(\lambda)$ is uniformly bounded on the contour as well as 
\begin{align*}
\Vert \cM^{-1}(\tilde{f}^l(\lambda))\Vert_{\Hbext^{\tilde{m},q(s)+1,l}}\le C\Vert \hat{\Psi}^l_{[s]}\Vert_{\Hbext^{\tilde{m},q(s),l}}
\end{align*}
with $C$ independent of $l$. This gives 
\begin{align*}
\Vert \hat{\Psi}^l_{[s]}\Vert_{\Hbext^{\tilde{m},q(s)+1,l}}\le C\Vert \hat{\Psi}^l_{[s]}\Vert_{\Hbext^{\tilde{m},q(s),l}}
\end{align*}
with $C$ independent of $l$. We obtain equivalent estimates in each step. Summarizing we find that for all $l>\vert s\vert+1$, $\hat{\Psi}^l_{[s]}\in \Hbext^{\tilde{m},w(s)-1/2,l}$ for all $\tilde{m}\ge m$ and we have an estimate 
\begin{align*}
\Vert \hat{\Psi}_{[s]}^l\Vert_{\Hbext^{\tilde{m},w(s)-1/2,l}}\le C \Vert \hat{\Psi}_{[s]}^l\Vert_{\Hbext^{\tilde{m},q(s),l}}
\end{align*}
with $C$ independent of $l$. Using the continuous embedding $\Hbext^{\infty,q}\subset \cA^{q+3/2}$ and the convergence of the series $\sum_{l}\hat{\Psi}^{l}_{[s]}$ in $\Hbext^{m,q(s)}$ this gives the claim. 

\item Smoothness away from $\partial X$ follows as in part (1), while the radial point estimates in \cite{Va1} show that $\hat{\Psi}_{[s]}$ is conormal at $\partial X$. We refer to \cite{HHV}, \cite{M2} for details. We see that the normal operator at $\rho=0$ is 
\[\cN=-2\sigma\rho(\rho D_{\rho}+i).\]
The boundary spectrum of the normal operator consists of the single point 
\[\{(-i,0)\}.\] 
The asymptotics is then established in the same way as in (1). 
\end{enumerate}
\end{proof}

\subsubsection{Asymptotic behavior in the $(t,r,\omega)$ coordinate system}

Suppose $\hat{\Psi}_{[s]}\in \Hbext^{m,q(s)}(X),$ with $m$ and $q(s)$ fulfilling the conditions of Proposition \ref{prop3.1}, 
\begin{equation}
\label{3.2.1}
\hat{T}_s(\sigma)\hat{\Psi}_{[s]}(\rho,\omega_*)=0
\end{equation} 
and  
\begin{equation*}
\hat{\Psi}_{[s]}(\rho,\omega_*)=e^{ik\varphi_*}\hat{F}_{[s]}(\rho,\theta)
\end{equation*}
for some $k\in \Z$. Note that $\hat{\Psi}_{[s]}, \hat{F}_{[s]}$ also depend on $\sigma,\, k$. Then the function 
\begin{equation*}
\Psi_{[s]}(t,r,\omega)=e^{-i\sigma t_*(t,r)}e^{ik\varphi_*(\varphi,r)}\Delta_b^{-s}\hat{F}_{[s]}(\sigma,k,\rho,\omega)=:e^{-i\sigma t}e^{ik\varphi}F_{[s]}(r,\theta)
\end{equation*}
is a solution of \eqref{TME}.
Note that we have for $r\le 3\bhm$ 
\begin{align*}
F_{[s]}(r,\theta)=e^{ik\int\frac{a}{\Delta_b}dr}e^{-i\sigma r_*}\Delta_b^{-s}\hat{F}_{[s]}(\sigma,k,\rho(r),\theta)
\end{align*}
and for $r\ge 4\bhm$ 
\begin{align*}
F_{[s]}(r,\theta)=e^{i\sigma r_*}\Delta_b^{-s}\hat{F}_{[s]}(\sigma,k,\rho(r),\theta).
\end{align*}
Let $\xi:=\frac{i(ak-2\bhm r_+\sigma)}{(r_+-r_-)}$. In the $\sigma=0$ case we require that $\hat{F}_{[s]}(\rho(r),\theta)=R(r)S_{l}(\theta)$ for some fixed $l$, where $e^{ik\varphi}S_{l}(\theta)$ is an eigenfunction of $-\Delta^{[s]}$(acting on $H^2(\slashed{\mathcal B}(s))$ with eigenvalue $\lambda^{l}_{[s]}$).    
\begin{prop}
\label{prop4.3}
We have 
\begin{align*}
F_{[s]}(r,\theta)&\sim (r-r_+)^{\xi-s},\, r\rightarrow r_+,\\
F_{[s]}(r,\theta)&\sim e^{i\sigma r}r^{2i\bhm\sigma}\frac{1}{r^{1+2s}},\,  r\rightarrow\infty,\, \sigma\neq 0,\, \\
F_{[s]}(r,\theta)&\sim \frac{1}{r^{1+l+s}},\,  r\rightarrow\infty, \, \sigma=0.
\end{align*}
\end{prop}
\begin{proof}
For $\sigma=0$ the asymptotic behavior at $r=\infty$ immediatly follows from Proposition \ref{prop3.1}. Using that $r_*\sim r+2\bhm \ln (r)$ close to $\infty$ we obtain the asymptotic behavior at $\infty$ also in the $\sigma\neq 0$ case. Concerning the behavior at $r=r_+$ we observe that $\hat{F}_{[s]}$ has to be continuous at $r=r_+$. We then use 
\begin{eqnarray*}
r_*&=&\int_{3\bhm}^r \frac{r'^2+\bha^2}{\Delta_b(r')}dr'\\
&\sim&\frac{r_+^2+\bha^2}{(r_+-r_-)}\ln (r-r_+),\, r\rightarrow r_+,\\
\int_{3\bhm}^r\frac{\bha}{\Delta_b(r')}dr'&\sim&\frac{\bha}{r_+-r_-}\ln (r-r_+),\, r\rightarrow r_+
\end{eqnarray*}
and obtain
\[e^{-i\sigma r_*}e^{im\int\frac{\bha}{\Delta_b}dr}\sim(r-r_+)^\xi,\]
which gives the asymptotic behavior at $r=r_+$. 
\end{proof}

\subsection{Absence of modes}
The following theorem follows from a theorem by Whiting \cite{Wh} in the $\Im \sigma>0$ case and for $\Im\sigma=0,\, \sigma\neq 0$ from a theorem of Andersson, Ma, Paganani and Whiting \cite{AMPW}. 
\begin{thm}
\label{Thm3.1}
Let $q(s)=q-2s$ if $s\ge 0$ and $q(s)=q$ if $s<0$. 
\begin{enumerate}
\item Let $\Im\,\sigma\ge 0,\,\sigma\neq 0,\, 
m>1/2+s,\, q(s)<-1/2,\, m+q(s)>-1/2-2s.$ 
Suppose that  
\begin{equation*}
\hat{T}_s(\sigma)\hat{\Psi}_{[s]}(\rho,\omega_*)=0,
\end{equation*}
with $\hat{\Psi}_{[s]}\in \Hbext^{m,q(s)}(X)$. Then $\hat{\Psi}_{[s]}=0$. 
\item Let $m>1/2+s,\, q\in(-3/2,-1/2),\, m+q(s)>-1/2-2s$.  Suppose that  
\begin{equation*}
\hat{T}_s(0)\hat{\Psi}_{[s]}(\rho,\omega_*)=0,
\end{equation*}
with $\hat{\Psi}_{[s]}\in \Hbext^{m,q(s)}(X)$. Then $\hat{\Psi}_{[s]}=0$. 
\end{enumerate}
\end{thm}
\begin{proof}
We can suppose that $\hat{\Psi}_{[s]}$ has a single mode in $\varphi_*$: 
\[\hat{\Psi}_{[s]}(\rho,\theta,\varphi_*)=e^{ik\varphi_*}\hat{F}_{[s]}(\rho,\theta).\]
We then build up $F_{[s]}(r,\theta)$ as in the previous subsection. To treat the case $\sigma\neq 0$ we use the results of \cite{Wh} and \cite{AMPW}. Note that the results in  
\cite{Wh} and \cite{AMPW} are results for the separated equation. To argue that this gives the absence of modes for general functions in our Sobolev space, we can apply Theorem 1.1 in \cite{FS1}. To see this in a bit more detail, we first come back to the $(t,r,\omega)$ coordinate system. Let 
\begin{align*}
F_{[s]}(r,\theta)&=e^{ik(\varphi_*-\varphi)}e^{i\sigma(t-t_*)}\Delta_b^{-s}\hat{F}_{[s]}(\rho(r),\theta),\, \cH_k=e^{ik\varphi} L^2((0,\pi),\sin \theta d\theta),\\
\cA_k&=\sigma^2\bha^2\sin^2\theta+\frac{k^2}{\sin^2\theta}+2sk\frac{\cos\theta}{\sin^2\theta}+s^2\cot^2\theta+2s\sigma \bha\cos\theta-\frac{1}{\sin\theta}\partial_{\theta}\sin\theta\partial_{\theta}.
\end{align*}
We have 
\begin{align*}
L^2(\R\times \Sph^2;dr_*d\omega)=\bigoplus_kL^2(\R,dr_*)\otimes \cH_k.
\end{align*}
Note that 
\begin{align*}
T_s\vert_{L^2(\R)\otimes{\cH}_k}=\cR_k+\cA_k,
\end{align*}
where $\cR_k$ is an operator only in the $r_*$ variable. Let us now fix $c>0,\, s$ and $k$ and let $U\subset \C$ be the strip $\vert \Im\, \sigma\vert<c$. Then by \cite[Theorem 1.1]{FS1} there exists a family of bounded operators $Q_n(\sigma)$ on $\cH_k$ defined for all $n\in \N$ and $\sigma\in U$ such that   
\begin{enumerate}
\item The image of each of the operators $Q_n(\sigma)$ is a finite dimensional invariant subspace of $\cA_k$.\footnote{The dimension of the image is at most $2$ for $n\ge1$, but this is not important for our purposes.} 
\item The $Q_n(\sigma)$ are complete in the sense that for every $\sigma\in U$
\begin{align*}
\sum_{n=0}^{\infty}Q_n={\bf 1}.
\end{align*}  
\end{enumerate}

Let ${\cH}_k^n=Q_n{\cH}_k$. By general linear algebra $\cA_k$ acting on $\cH^n_k$ can be decomposed into Jordan blocks. We can suppose that $\cA_k\vert_{\cH^n_k}$ is described by a single Jordan block and that $F_{[s]}(r,\theta)$ writes as 
\begin{align*}
F_{[s]}(r,\theta)=\sum_{j=1}^pf_j(r)g_j(\theta)
\end{align*}  
with 
\begin{align*}
\cA_k g_1&=\lambda g_1+g_2,\\
\cA_k g_2&=\lambda g_2+g_3,\\
...&...\\
\cA_kg_p&=\lambda g_p. 
\end{align*}
Here $\{g_j\}_{j=1,...,p}$ is a basis of $\cH^n_k$. Now,
\begin{align*}
0=T_sF_{[s]}(r,\theta)=\sum_{j=1}^p(\cR f_j)g_j+\sum_{j=1}^{p-1}f_j(\lambda g_j+g_{j+1})+f_p\lambda g_p.
\end{align*}
As $\{g_j\}_{j=1,...,p}$ is a basis of $\cH^n_k$ we first obtain 
\begin{align*}
\cR f_1+\lambda f_1=0
\end{align*}
and thus $f_1=0$ by the results of \cite{Wh} and \cite{AMPW}. Proceeding with $j=2$ we find 
\begin{align*}
\cR f_2+\lambda f_2=0
\end{align*}
    and thus $f_2=0$. We eventually find $f_1=f_2=...=f_p=0$ and thus $F_{[s]}=0$. This gives $\hat{\Psi}_{[s]}(r,\omega_*)=0$ for $r>r_+$. To show that $\hat{\Psi}_{[s]}(r,\omega_*)=0$ for $r\le r_+$ we apply the same argument as in the proof of \cite[Lemma 1]{Zw}. We can suppose that $\hat{\Psi}_{[s]}(r,\theta,\varphi)=e^{-ik\varphi_*}u(r,\theta)=:\psi$. Let 
    \begin{align*}
    \cE(r)=\int_{\Sph^2}(-(r-r_+))^{-N}(\Delta_b\sin\theta\vert \partial_r\psi\vert^2+\frac{1}{\sin\theta}\vert \partial_{\varphi_*}\psi\vert^2+ \sin\theta \vert\partial_{\theta}\psi\vert^2+\vert \psi\vert^2)d\theta d\varphi_*.
    \end{align*}
    We have 
    \begin{align*}
    \frac{d}{dr}\cE(r)=N(-(r-r_+))^{-1}\cE(r)+\int_{\Sph^2}(-(r-r_+))^{-N}R(\sigma,\psi,\frac{1}{\sqrt{\sin\theta}}\partial_{\varphi_*}\psi,\partial_r\psi,\sqrt{\sin\theta}\partial_{\theta}\psi),
    \end{align*}
    where $R$ is quadratic in $(\psi,\frac{1}{\sqrt{\sin\theta}}\partial_{\varphi_*}\psi,\partial_r\psi,\sqrt{\sin\theta}\partial_{\theta}\psi)$ and independent of $N$.\footnote{This wouldn't be true without fixing the angular momentum.} Thus for $N$ sufficiently large we have $ \frac{d}{dr}\cE(r)\ge 0$ for $r\le r_+$. Integrating between $r_+-\delta$ and $r_+-\epsilon\, (\delta>\epsilon)$ gives 
    \begin{eqnarray*}
    \lefteqn{\int_{\Sph^2}\delta^{-N}(\delta^2\sin\theta\vert \partial_r\psi\vert^2+\frac{1}{\sin\theta}\vert \partial_{\varphi_*}\psi\vert^2+\vert \sin\theta\psi\vert^2+\vert \psi\vert^2)d\theta d\varphi_*}\\
    &\lesssim& \epsilon^{-N}\int_{\Sph^2}\sin\theta\vert\partial_r\psi\vert^2+\frac{1}{\sin\theta}\vert \partial_{\varphi_*}\psi\vert^2+\vert \sin\theta\psi\vert^2+\vert \psi\vert^2)d\theta d\varphi_*\\
    &\lesssim& \epsilon^{-N+K}
    \end{eqnarray*}
    for all $K>0$. Indeed by the argument in the proof of Proposition \ref{prop3.1} we know that $\psi\in \Hbext^{\infty,q(s)}$. It therefore vanishes to all orders at $r=r_+$. Choosing $K$ large enough and letting $\epsilon\rightarrow 0$ gives $\psi((r=r_+-\delta),\theta)=0$. With the operator $\hat{T}_s(\sigma)$ being hyperbolic and $r=r_+-\delta$ being spacelike, this gives $\psi=0$ by classical energy estimates.   

It remains to treat the case $\sigma=0$. In this case, the operator $\cA_k$ is diagonalisable and it suffices to consider $F_{[s]}$ of the form $F_{[s]}(r,\theta)=S_{l}(\theta)R(r)$. The equation then further decouples. The function $R(r)$ fulfills :
\begin{equation}
\label{TMER}
\Delta_b^{-s}\frac{d}{dr}\Delta_b^{s+1}\frac{d}{dr}R(r)+\frac{\bha^2k^2+2i\bha(r-\bhm)ks}{\Delta_b}R-AR=0,
\end{equation}
where $A=(l-s)(l+s+1)$. If $k=0$, then we can multiply \eqref{TMER} by $\bar{R}$ and integrate by parts. This gives 
\begin{align}
\label{TMI}
\int_{r_+}^{\infty}\Delta_b\left\vert \frac{dR}{dr}\right\vert^2+A\left\vert R\right\vert^2 dr=0.
\end{align}
If $s<0$, then $A>0$ and \eqref{TMI} gives $R=0$. Let $s\ge 0$. If $A>0$, then \eqref{TMI} gives $R=0$ as before. If $A=0$, then \eqref{TMI} gives $R=const.$ which is only permitted if the constant is zero. Therefore we can suppose $k\neq 0$ in the following. 
The Teukolsky equation \eqref{TMER} has three regular singular points which are $r=r_{\pm}$ and $r=\infty$. For the general theory of this type of equations see \cite{Kr}.
\begin{enumerate}
\item Study of the singular points $r=r_{\pm}$. We rewrite the Teukolsky equation as 
\begin{equation*}
\frac{d^2}{dr^2}R(r)+2(s+1)\frac{r-\bhm}{\Delta_b}\frac{dR}{dr}+\frac{\bha^2k^2+2i\bha(r-\bhm)ks}{\Delta_b^2}R-\frac{A}{\Delta_b}R=0.
\end{equation*}
Noting that 
\[2\frac{r_{\pm}-\bhm}{r_{\pm}-r_{\mp}}=1\]
we find the indicial equation at $r=r_{\pm}$:
\begin{equation*}
\alpha_{\pm}^2+s\alpha_{\pm}+\frac{\bha^2k^2+i\bha(r_{\pm}-r_{\mp})ks}{(r_+-r_-)^2}=0
\end{equation*}
with roots 
\begin{align*}
\alpha_+=-\frac{i\bha k}{r_+-r_-}\quad\mbox{or}\quad \alpha_+=-s+\frac{i\bha k}{r_+-r_-}
\end{align*}
 resp. 
 \begin{align*}
 \alpha_-=\frac{i\bha k}{r_+-r_-}\quad\mbox{or}\quad \alpha_-=-s-\frac{i\bha k}{r_+-r_-}.
 \end{align*} 
\item Study of the singular point at $r=\infty$. We put $z=\frac{1}{r}$. We then have 
\[\frac{d}{dr}=-z^2\frac{d}{dz}.\] 
The Teukolsky equation \eqref{TMER} can be written as 
\begin{eqnarray*}
\Delta_b^{-s}z^2\frac{d}{dz}\left(\Delta_b^{s+1}z^2\frac{dR}{dz}\right)+\left(\frac{\bha^2k^2+2i\bha (r-\bhm)ks}{\Delta_b}-A\right)R&=&0\\
\Leftrightarrow \frac{d^2R}{dz^2}-\frac{2(s+1)}{z^2\Delta_b}\left(\frac{1}{z}-\bhm\right)\frac{dR}{dz}+\frac{2}{z}\frac{dR}{dz}+\left(\frac{\bha^2k^2+2i\bha\left(\frac{1}{z}-\bhm\right)ks}{\Delta_b^2z^4}-\frac{A}{z^4\Delta_b}\right)R&=&0.
\end{eqnarray*}
Taking into account that $z^2\Delta_b\rightarrow 1$ when $z\rightarrow 0$ we find the indicial equation at $z=0$:
\[\alpha^2-(2s+1)\alpha-A=0.\]
We therefore find the roots $\alpha=s-l$ and $\alpha=s+l+1$. 
\end{enumerate}
We now follow \cite{Kr} to bring this equation into its canonical form. Let
\begin{eqnarray*}
T(r)=R(r)(r-r_+)^{i\frac{\bha k}{r_+-r_-}}(r-r_-)^{-i\frac{\bha k}{r_+-r_-}},\\
u(\uprho)=T(r=\uprho(r_+-r_-)+r_-). 
\end{eqnarray*}
Then $u(\uprho)$ fulfills 
\begin{equation}
\label{hypergeom}
\uprho(\uprho-1)u''(\uprho)+((\alpha+\beta+1)\uprho-\gamma)u'(\uprho)+\alpha\beta u(\uprho)=0
\end{equation}
with 
\begin{equation*}
\alpha=s-l,\, \beta=s+l+1,\, \gamma=s+1+2i\frac{\bha k}{r_+-r_-}.  
\end{equation*}
Let $R(r)$ be an outgoing solution of the Teukolsky equation \eqref{TMER} with $\sigma=0$. From the asymptotic behavior of $R$ at $r_+,\infty$ (see Proposition \ref{prop4.3}), we can read off the asymptotic behavior of $u$
\[ u\sim (\uprho-1)^{-s+2i\frac{\bha k}{r_+-r_-}},\, \uprho\rightarrow 1;\quad\, u\sim \uprho^{-s-l-1},\, \uprho\rightarrow \infty.\] 
The equation \eqref{hypergeom} has two independent smooth solutions. Let $F(\alpha,\beta,\gamma;\uprho)$ be the solution which is analytic in a neighborhood of $0$. Starting with this solution we build up the functions 
\begin{align*}
u_1&=\uprho^{-\alpha}F(\alpha,1+\alpha-\gamma,1+\alpha+\beta-\gamma;1-\frac{1}{\uprho}),\\
u_2&=\uprho^{\beta-\gamma}(1-\uprho)^{\gamma-\alpha-\beta}F(\gamma-\beta,1-\beta,1+\gamma-\alpha-\beta;1-\frac{1}{\uprho})
\end{align*}
which are also solutions to \eqref{hypergeom} and they are analytic in ${\rm Re\, }\uprho>1/2$, see \cite[page 74]{Kr} for details. By \cite[Theorem 5.1]{Kr} the series $F(\alpha,1+\alpha-\gamma,1+\alpha+\beta-\gamma;z)$ and  $F(\gamma-\beta,1-\beta,1+\gamma-\alpha-\beta;z)$ converge at $\vert z\vert=1$. 
Analyzing the asymptotic behavior at $\uprho=1$ one easily sees that both are linearly independent (recall that we suppose  $k\neq 0$).  Therefore $u$ writes on ${\rm Re}\, \uprho>1/2$ as 
\[u=cu_1+du_2.\]  
The asymptotic behavior of $u$ at $\uprho=1$ gives $c=0$. Using the convergence of the series $F(\gamma-\beta,1-\beta,1+\gamma-\alpha-\beta;z)$ at $\vert z\vert=1$ we see that $u_2\sim \uprho^{-(s-l)}$ at $\infty$ which gives $d=0$ and thus $F_{[s]}=0$. By the same argument as in the $\sigma\neq 0$ case we obtain $\hat{\Psi}_{[s]}(0)=0$. 
\end{proof}
\begin{rmk}
The above theorem together with the Fredholm properties of the Teukolsky operator shown in \cite{M2} show that the operators 
\begin{align*}
\hat{T}_s(\sigma):\{\Psi \in \Hbext^{m,q(s)};\, \hat{T}_s(\sigma)\Psi \in \Hbext^{m-1,q(s)+2}\}\rightarrow \Hbext^{m-1,q(s)+2},\\
\hat{T}_s(0):\{\Psi \in \Hbext^{m,q(s)};\, \hat{T}_s(0)\Psi \in \Hbext^{m-1,q(s)+2}\}\rightarrow \Hbext^{m-1,q(s)+2}. 
\end{align*}
are in fact invertible. We however won't need this fact in the following. In the case $\Im\, \sigma>0$ the conditions on $m,\, q(s)$ are a little bit weaker in \cite{M2} than what we require. 
\end{rmk}

\subsection{The scalar wave operator}
For the scalar wave operator which can be considered as a special case of the Teukolsky operator, we will need a mode analysis also in spaces with weaker decay. This is completely analogous to the analysis in \cite{HHV}, we summarize :
\begin{thm}
\label{Thm0}
\begin{enumerate}
\item For $\Im\sigma\geq 0$, $\sigma\neq 0$, the operator 
\begin{equation}
  \label{Eq0}
    \wh{\Box_{g_b,0}}(\sigma)\colon\{u\in\Hbext^{m,q}(X)\colon\wh{\Box_{g_b,0}}(\sigma)u\in\Hbext^{m-1,q+2}(X)\} \to\Hbext^{m-1,q+2}(X)
  \end{equation}
  is invertible when $m>\half$, $q<-\half$, and $m+q>-\half$. 
  \item The stationary operator
  \[
    \wh{\Box_{g_b,0}}(0)\colon\{u\in\Hbext^{m,q}(X)\colon\wh{\Box_{g_b,0}}(0)u\in\Hbext^{m-1,q+2}(X)\} \to\Hbext^{m-1,q+2}(X)
  \]
    is invertible for all $m>\half$ and $q\in(-\tfrac32,-\half)$. 
    \item We have
  \begin{subequations}
  \begin{alignat}{3}
  \label{Eq00Grows0Ker}
    &\ker\wh{\Box_{g_b,0}}(0) \cap \Hbext^{\infty,-3/2-}&=&\ \la u_{b,s 0}\ra, \\
  \label{Eq00Grows0KerAdj}
    &\ker\wh{\Box_{g_b,0}}(0)^* \cap \Hbsupp^{-\infty,-3/2-}&=&\ \la u_{b,s 0}^*\ra,
  \end{alignat}
  \end{subequations}
  where
  \begin{equation}
  \label{Eq00Grows0}
    u_{b,s 0}=1,\quad u_{b,s 0}^*=H(r-r_+).
  \end{equation}
  \item 
  Furthermore, the spaces
  \begin{subequations}
  \begin{alignat}{4}
  \label{Eq00Grows1Ker}
    &\ker\wh{\Box_{g_b,0}}(0) \cap \Hbext^{\infty,-5/2-}& &=\ &\la u_{b,s 0}\ra & \oplus \{ u_{b,s 1}(\scal) \colon \scal\in\scalspace_1 \}, \\
  \label{Eq00Grows1KerAdj}
    &\ker\wh{\Box_{g_b,0}}(0)^* \cap \Hbsupp^{-\infty,-5/2-}& &=\ &\la u_{b,s 0}^*\ra & \oplus \{ u_{b,s 1}^*(\scal) \colon \scal\in\scalspace_1 \}
  \end{alignat}
  \end{subequations}
  are $4$-dimensional. Let $b=(\bhm,\bha)$ and $b_0=(\bhm,0)$. Then we have 
  \begin{align*}
  u_{b_0,s1}(\scal)=(r-\bhm)\scal,\quad u^*_{b_0,s1}= (r-\bhm)H(r-2\bhm)\scal
  \end{align*}
 
  and 
  \begin{align}
  \label{4.21}
  u_{b,s1}-u_{b_0,s1}\in \Hbext^{\infty,-1/2-},\quad u^*_{b,s1}-u^*_{b_0,s1}\in\Hbsupp^{-\infty,-1/2-}. 
  \end{align}
\end{enumerate}
\end{thm}
Before proving the theorem we make the following observation on Fredholm operators. 
\begin{lemma}
\label{Fredholm}
Let $X,Y,Z$ be Banach spaces and suppose that $Z$ is continuously embedded in $Y$. Suppose furthermore that $P:X\rightarrow Y$ is a Fredholm operator. Then $P_Z:P^{-1}(Z)\rightarrow Z$ is a Fredholm operator.   
\end{lemma}
\begin{proof}
The lemma follows from the fact that $\Ker P_Z\subset \Ker P,\, Z/(Z\cap P(X))\subset Y/P(X)$ and the fact that $Z\cap P(X)$ is closed in $Z$ because $P(X)$ is closed in $Y$.   
\end{proof}
\begin{proof}[Proof of the Theorem]
We first show that the operators in $(1),(2)$ are Fredholm operators. For $(2)$ this follows directly from the results of \cite{M2} which uses \cite{Va1}, \cite{Va2}. The proof follows the general scheme in the proof of 
\cite[Theorem 4.3]{HHV} 
and uses that the general setting of \cite{Va2} applies to the wave equation on the Kerr metric for all subextreme values of $\bha$ as well as the fact that the principal features of the Hamiltonian flow of the classical symbol are the same for all subextreme values of $\bha$ (see also the proof of our theorem \ref{ThmOp}). Concerning $(1)$ the spaces used in \cite{HHV}, \cite{M2}  are slightly different, we therefore have to argue that we can also use the spaces in the above theorem. By the results of \cite{Va1}, \cite{Va2} as well as \cite{M2} we know that  
\begin{align*}
\wh{\Box}_{g_b,0}(\sigma):\left\{u\in \bar{H}^{m,m+q,q}_{sc,\bop,res}(X);\wh{\Box}_{g_b,0}(\sigma)u\in\bar{H}_{sc,\bop,res}^{m-2,m+q+1,q+1}(X)\right\}\rightarrow \bar{H}^{m-2,m+q+1,q+1}_{sc,\bop,res}(X)
\end{align*}
are Fredholm operators. Here $\bar{H}^{m,r,q}_{sc,\bop,res}(X)$ are scattering-b Sobolev spaces as defined in \cite[Section 3]{Va2}. We have in particular
\begin{align*}
\bar{H}_{sc,\bop,res}^{m,m+q,q}(X)=\Hbext^{m,q}(X).
\end{align*} 

We have 
\begin{align*}
Z:=\Hbext^{m-1,q+2}=\bar{H}_{sc,\bop,res}^{m-1,m+q+1,q+2}\subset\bar{H}^{m-2,m+q+1,q+1}_{sc,\bop,res}=:Y
\end{align*} 
with continuous embedding. We then can apply Lemma \ref{Fredholm} with 
\begin{align*}
X:=\left\{u\in \bar{H}^{m,m+q,q}_{sc,\bop,res}(X);\wh{\Box}_{g_b,0}(\sigma)u\in\bar{H}_{sc,\bop,res}^{m-2,m+q+1,q+1}(X)\right\}.
\end{align*}
We now argue that the operators have index zero. We first define the spaces 
\begin{align*}
\Hbext^{m,q;k}=\{u\in \Hbext^{m,q};\, u=e^{ik\varphi_*}\tilde{u}(r,\theta)\}. 
\end{align*}
We start with $\sigma=0$. The restriction of $\wh{\Box_{g_b,0}}(0)$ to $\Hbext^{m,q;k}$ 
\begin{align*}
    \wh{\Box_{g_b,0}}^k(0)\colon\{u\in\Hbext^{m,q;k}(X)\colon\wh{\Box_{g_b,0}}(\sigma)u\in\Hbext^{m-1,q+2;k}(X)\} \to\Hbext^{m-1,q+2;k}(X)
\end{align*}
is also Fredholm. We have 
\begin{align*}
\ker\wh{\Box_{g_b,0}}(0)=\left\la \bigcup_{k\in \{-N,...,N\}} \ker \wh{\Box_{g_b,0}}^k(0)\right\ra
\end{align*}
as well as an equivalent equality for the adjoint. Note that the union is finite here because the kernel is finite dimensional. It is therefore sufficient to show the index zero property for $\wh{\Box_{g_b,0}}^k(0)$ for all $k\in \{-N,...,N\}$. Now 
\begin{align*}
\wh{\Box_{g_b,0}}^k(0)=\wh{\Box_{g_{b_0},0}}^k(0)+P_k,\quad P_k\in \rho^2 {\rm Diff}^1_b. 
\end{align*}
As adding an element of $\rho^2 {\rm Diff}^1_b$ does not change the domain, we can continuously deform the operator $\wh{\Box_{g_b,0}}^k(0)$ on Kerr to the corresponding operator on Schwarzschild for which we know from \cite[Theorem 6.1]{HHV} that it is invertible. Note in this context that we can work on the same manifold for all angular momenta per unit mass $0\le {\mathfrak a}'\le {\mathfrak a}$ because the condition \eqref{eqr0} entails that the same condition holds with $\bha$ replaced by $\bha'$. A similar argument shows that $\wh{\Box_{g_b}}(\sigma)$ has Fredholm index zero for $\sigma\neq 0$. Now by theorem \ref{Thm3.1} we know that the kernels of both operators are equal to $\{0\}$, they are thus invertible. 

The proof of (3)-(4) is strictly analogous to the proof of \cite[Proposition 6.2]{HHV}. Nevertheless to show in addition \eqref{4.21} we construct the solution $u_{b,s1}(\scal)$ starting with the corresponding solutions $u_{b_0,s1}(\scal)$ for Schwarzschild rather than the one for Minkowski like in \cite{HHV}. We construct the solutions $u_{b,s1}(S)$, the argument for $u^*_{b,s1}(\scal)$ is analogous. Let $v:=(r-\bhm)\scal\in \Hbext^{\infty,-5/2-}$ and fix a cutoff $\chi\in C^{\infty}(\R)$ with $\chi=0$ for $r\le 3\bhm,\, \chi=1$ for $r\ge 4\bhm$. Then 
\begin{align*}
e:=\wh{\Box_{g_b}}(0)(\chi v)&=\chi\wh{\Box_{g_{b_0}}}(0)(v)+[\wh{\Box_{g_{b_0}}}(0),\chi]v+(\wh{\Box_{g_b}}(0)-\wh{\Box_{g_{b_0}}}(0))(\chi v)\\
&\in0+\Hbext^{\infty,\infty}+\Hbext^{\infty,3/2-}=\Hbext^{\infty,3/2-}.
\end{align*} 
In the last step we have used that $\wh{\Box_{g_b}}(0)-\wh{\Box_{g_{b_0}}}(0)\in \rho^4{\rm Diff}^2_b$, see \cite[(3.44)]{HHV}. Now $\wh{\Box_{g_b}}(0)w=-e$ can be solved by $w\in \Hbext^{\infty,-1/2-}$; indeed $e$ is orthogonal to the kernel of $\wh{\Box_{g_b}}^*(0)$ in $\Hsupp^{-\infty,-3/2+}$ which is trivial by (2). We then have $u_{b,s1}=\chi v+w$ and it fulfills \eqref{4.21}. 
\end{proof}

\section{The $1-$form wave operator}
\label{Sec1form}
We now analyze mode solutions of the $1-$form wave operator. 
\subsection{Mode solutions}

\begin{thm}
\label{Thm1}
  Consider $\Box_{g_b,1}$ acting on $1-$forms. There exists $m_1>0$ with the following property.  
  \begin{enumerate}
  \item For $\Im\sigma\geq 0$, $\sigma\neq 0$, the operator
    \[
      \wh{\Box_{g_b,1}}(\sigma)\colon\left\{\omega\in \Hbext^{m,q}(X;\wt\Tsc{}^*X)\colon\wh{\Box_{g_b,1}}(\sigma)\omega\in\Hbext^{m-1,q+2}(X;\wt\Tsc{}^*X)\right\} \to\Hbext^{m-1,q+2}(X;\wt\Tsc{}^*X)
    \]
    is invertible when $m>m_1$, $q<-\half$, $m+q>-\half$.
  \item For $m>m_1$ and $q\in(-\tfrac32,-\half)$, the stationary operator
    \begin{equation}
    \label{Eq1StatOp}
    \begin{split}
      \wh{\Box_{g_b,1}}(0)\colon&\left\{\omega\in\Hbext^{m,q}(X;\wt\Tsc{}^*X)\colon\wh{\Box_{g_b,1}}(0)\omega\in\Hbext^{m-1,q+2}(X;\wt\Tsc{}^*X)\right\} \\
        &\qquad\to\Hbext^{m-1,q+2}(X;\wt\Tsc{}^*X)
    \end{split}
    \end{equation}
    has 1-dimensional kernel and cokernel. We have 
    \begin{align*}
    &{\rm ker}\, \wh{\Box_{g_b,1}}(0)\cap \Hbext^{\infty,-\half-}=\la\omega_{b,s_0}\ra,\\
    &{\rm ker}\, \wh{\Box_{g_b,1}}(0)^*\cap\Hsupp^{-\infty,-\half-}=\la \omega_{b,s_0}^*\ra
    \end{align*}
    with
    \begin{eqnarray*}
    \omega_{b,s_0}&=&\left\{\begin{array}{c} \frac{r}{\varrho_b^2}(dt_*-a\sin^2\theta d\varphi_*)+\frac{r_+-r}{\Delta_b} dr,\, r\le 3,\bhm\\
    \frac{r}{\varrho_b^2}(dt_*-a\sin^2\theta d\varphi_*)+\left(\frac{(r^2+a^2)r+\varrho_b^2r_+}{\varrho_b^2\Delta_b}\right)dr,\, r\ge 4\bhm,\end{array}\right.\\
    \quad \omega_{b,s_0}^*&=&\delta(r-r_+)dr. 
    \end{eqnarray*}
  \end{enumerate}
  \end{thm}
  \begin{proof}
As in the small $\bha$ case the operators $\hat{\Box}_{g_b,1}(\sigma)$ are Fredholm operators of index $0$. {We} revisit the argument given in \cite{HHV} for the more complicated case of the gauge fixed Einstein equation in the proof of Theorem \ref{ThmOp}. We now want to compute the different kernels. As $g$ is Ricci flat we have 
\begin{equation}
\label{1M1}
\Box_{g_b,1}=(d+\delta_{g_b})^2=d\delta_{g_b}+\delta_{g_b}d.
\end{equation}
Let $\omega=e^{-i\sigma t_*}h$ be a mode solution :
\[\Box_{g_b,1}\omega=0.\]
Then 
\begin{equation*}
0=\delta_{g_b}\Box_{g_b,1}\omega=\delta_{g_b}d\delta_{g_b}\omega=\Box_{g_b,0}\delta_{g_b}\omega=0.
\end{equation*} 
Here we have used that $\delta_{g_b}^2=0$. Now let us note that 
\[\delta_{g_b}\omega=e^{-i\sigma t_*}f, \, f\in  \Hbext^{m-1,q}(X).\]
We can therefore apply Theorem \ref{Thm3.1} for $s=0$ to obtain: 
\begin{equation}
\label{1M2}
\delta_{g_b}\omega=0.
\end{equation}
Putting this into the wave equation we find :
\begin{equation}
\label{1M3}
\delta_{g_b}d\omega =0.
\end{equation}
Let $F=d\omega$. Then $F$ is a Maxwell field 
\begin{equation}
\label{Maxwelldomega}
dF=0,\quad \delta_{g_b} F =0. 
\end{equation}
We now proceed as in Section \ref{SecTeu} and build up the scalars $\hat{\Psi}_{[s]}\in \Hbext^{m-1,q(s)}$. By Theorem \ref{Thm3.1} we obtain that $\hat{\Psi}_{[\pm 1]}$ and thus $\Phi_{\pm 1}$ are zero. We now go back to Boyer-Lindquist coordinates $(t,r,\theta,\varphi)$.  We distinguish two cases :

{\bf 1st case} : $\sigma\neq 0$. Using the first and the third equation in \cite[Chapter 8 (11)]{Ch} we obtain 
\[\frac{iK}{\Delta}\Phi_0=0,\] 
with $K=(r^2+a^2)\sigma+am$ and $m$ is the $\varphi$ mode in the separation of variables.\footnote{Note that $\Phi_0=\phi_1$ in the notations of Chandrasekhar.} Thus $\Phi_0=0$. This means that the whole Maxwell field $F$ is zero :
\begin{equation}
\label{1M4}
d\omega=0.
\end{equation}
Recall that $\cX=(r_+,\infty)\times\Sph^2$. We write 
\begin{equation*}
\omega=e^{-i\sigma t}(h_T+ h_Ndt), 
\end{equation*}
where $h_T$ is a $1-$form on $\cX$. 
Then $d\omega=0$  is equivalent to :
\begin{equation}
\label{1M5}
\left.\begin{array}{rcl}
d_{\cX}h_T&=&0,\\
-i\sigma h_T-d_{\cX} h_N&=&0. \end{array}\right\}
\end{equation}
By Poincar\'e's lemma we have $h_T=d_{\cX}\tilde{A}$. Now observe that 
\begin{align*}
d_{\cX}(i\sigma^{-1}h_N-\tilde{A})=h_T-h_T=0
\end{align*}
and thus $h_N=-i\sigma(\tilde{A}+c)$ for some constant $c$. It follows that $\omega=d (e^{-i\sigma t}(\tilde{A}+c))$. We now put $A=e^{-i\sigma t}(\tilde{A}+c)$.  Note that we could also apply Poincar\'e's lemma directly on spacetime, but we have to make sure that the potential $A$ is a mode solution. 
Putting this now into \eqref{1M2} we find that 
\begin{equation}
\label{1M6}
\Box_{g_b,0}A=0.
\end{equation}
Now $\tilde{A}+c=i\sigma^{-1}h_N\in \Hbext^{m-1,q}$. We can therefore apply Theorem \ref{Thm0} to conclude that $A=0$ and thus 
\begin{equation}
\label{1M7}
\omega=0.
\end{equation}

{\bf 2nd case }: $\sigma=0$. If $m\neq 0$ we obtain by the same argument $\Phi_0=0$. If $m=0$, the equations in \cite[Chapter 8 (11)]{Ch} give :
\begin{eqnarray}
\label{Coulomb1}
\partial_r\Phi_0&=&-\frac{2}{r-ia\cos \theta}\Phi_0,\\
\label{Coulomb2}
\partial_{\theta}\Phi_0&=&-\frac{2ia\sin \theta}{r-ia\cos\theta}\Phi_0.
\end{eqnarray}
 Integrating \eqref{Coulomb1} we find $\Phi_0=\frac{C(\theta)}{(r-ia\cos \theta)^2}$. Putting this into \eqref{Coulomb2} we find $C(\theta)=const.$ It follows that $F_{\mu\nu}$ is a Coulomb solution:
 \[F_{\mu\nu}=4({\rm Re} \Phi_0\, n_{[\mu}l_{\nu]}+i\Im\Phi_0\,m_{[\mu}\bar{m}_{\nu]}).\]

 $\omega$ is a potential for the Coulomb solution and it fulfills the Lorenz gauge \eqref{1M2}. Therefore $\omega=C(\omega_0+d\tilde{f})$ with 
 \[
    \omega_0= \frac{r}{\varrho_b^2}(d t-a\sin^2\theta\,d\varphi).
  \]
  We will suppose $C=1$ in the following. The $1-$form $\omega_0$ is singular at the horizon, we therefore have to correct this behavior by a gauge term. Concretely we have for $r\le 3\bhm$:
  \begin{align*}
  \omega_0=\frac{r}{\varrho_b^2}(dt_*-a\sin^2\theta d\varphi_*)-\frac{r}{\Delta_b}dr
  \end{align*}
  and for $r\ge 4\bhm$ 
  \begin{align}
   \omega_0=\frac{r}{\varrho_b^2}(dt_*-a\sin^2\theta d\varphi_*)+\frac{(r^2+a^2)r}{\varrho_b^2\Delta_b}dr.
   \end{align}
 Therefore we put
 \begin{align*}
 \omega_{b,s_0}=\omega_0+\frac{r_+}{\Delta_b}dr.
 \end{align*}  
 Note that $\omega_{b,s_0}\in \Hbext^{\infty,q}$. An explicit calculation gives
 \begin{align*}
 \Box_{g_b,1}\omega_{b,s_0}=0.
 \end{align*}
 By the same argument as before we find 
 \begin{align*}
 \delta_{g_b}\omega_{b,s_0}=0.
 \end{align*}
 Let now $f=\tilde{f}-\int\frac{r_+}{\Delta_b}$. We find
 \begin{align*}
 \Box_{g_b,0} f&=\delta_{g_b}d\tilde{f}-\delta_{g_b}\frac{r_+}{\Delta_b}dr=\delta_{g_b}\omega-\delta_{g_b}\omega_b=0. 
 \end{align*}
 By Theorem \ref{Thm0} $f=const$. 
 It remains to show that $\omega_{b}^*$ is in the cokernel.  We first observe that $\omega^*_{b,s_0}=d (H(r-r_+))$ and then compute 
 \begin{equation*}
 \hat{\Box}^*_{g_b,1}(0)d (H(r-r_+))=d\hat{\Box}^*_{g_b,0}(0) H(r-r_+)=0.
 \end{equation*}   
 \end{proof}
 \subsection{Growing modes}
 \begin{prop}
\label{Prop10Grow}
  We have
  \begin{subequations}
  \begin{alignat}{3}
  \label{Eq10Grows0Ker}
    &\ker\wh{\Box_{g_b,1}}(0) \cap \Hbext^{\infty,-3/2-}& =&\ \la\omega_{b,s 0}\ra \oplus \la\omega_{b,s 0}^{(0)}\ra \oplus \{\omega_{b,s 1}(\scal) \colon \scal\in\scalspace_1\}, \\
  \label{Eq10Grows0KerAdj}
    &\ker\wh{\Box_{g_b,1}}(0)^* \cap \Hbsupp^{-\infty,-3/2-}& =&\ \la\omega_{b,s 0}^*\ra \oplus \{\omega_{b,s 1}^*(\scal) \colon \scal\in\scalspace_1\},
  \end{alignat}
  \end{subequations}
  where, with $\flat$ denoting the musical isomorphism $V^\flat:=g_b(V,-)$, and using~\eqref{Eq00Grows1Ker}--\eqref{Eq00Grows1KerAdj},
  \begin{alignat}{3}
  \label{Eq10Grows0}
    &\omega_{b,s 0}^{(0)}=\pa_t^\flat,&& \\
  \label{Eq10Grows1}
    &\omega_{b,s 1}(\scal)=d u_{b,s 1}(\scal),&\quad
    &\omega_{b,s 1}^*(\scal)=d u_{b,s 1}^*(\scal).
  \end{alignat}
\end{prop}
\begin{proof}
We follow closely the proof of \cite[Proposition 7.8]{HHV}. Using $d\wh{\Box_{g_b,0}}(0)=\wh{\Box_{g_b,1}}(0)d$ we see that the R.H.S. of \eqref{Eq10Grows0Ker} lies in the left hand side of the same equation, the same argument shows the inclusion of the R.H.S. of  \eqref{Eq10Grows0KerAdj} in the L.H.S. 

For the inclusions $\subseteq$ the one forms 
 \begin{equation}
  \label{Eq10GrowsNormal}
    d t,\ d x^1,\ d x^2,\ d x^3;
  \end{equation}
play a central role as they are annihilated by the normal operator $\wh{\Box_{\ubar g,1}}(0)$.

  In order to prove `$\subseteq$' in~\eqref{Eq10Grows0Ker}, note that any $\omega\in\ker\wh{\Box_{g_b}}(0)\cap\Hbext^{\infty,-3/2-}$ is of the form $\omega=\chi\,v+\tilde\omega$ where $v$ is a linear combination (with constant coefficients) of the 1-forms~\eqref{Eq10GrowsNormal}, and $\tilde\omega\in\Hbext^{\infty,-1/2-}$. Here $\chi$ is a radial cutoff which equals $1$ at infinity and $0$ for $r\le 3\bhm$. This follows from a normal operator argument. Upon subtracting a linear combination of $\omega_{b,s 0}^{(0)}$ and $\omega_{b,s 1}(\scal)$ from $\omega$, we can thus assume $\omega=\tilde\omega$, which by Theorem~\ref{Thm1} is a scalar multiple of $\omega_{b,s 0}$.

  The argument for proving `$\subseteq$' in~\eqref{Eq10Grows0KerAdj} is slightly more subtle. There is an obstruction to the existence of a mode with $\pa_t^\flat$ asymptotics given by the nonvanishing pairing 
   \begin{equation}
  \label{Eq10NoDualdt}
    \big\la\wh{\Box_{g_{b,1}}}(0)^*(\chi dt),\omega_{b,s 0}\big\ra=4\pi\neq 0. 
  \end{equation}
  To show \eqref{Eq10NoDualdt}, first note that \eqref{Eq10NoDualdt} does not depend on the choice of the cutoff. Indeed if $\chi,\, \tilde{\chi}$ are two such cutoffs, then we have 
  \begin{align*}
  (\chi-\tilde{\chi}) dt\in \Hbext^{\infty,\infty}. 
  \end{align*} 
  Now fix such a cutoff $\chi$ and consider $\chi_{\epsilon}(\rho)=\chi\left(\frac{\rho}{\epsilon}\right)$. The result will then be independent of $\epsilon$. To compute the exact value, first note that for $Q\in \rho^3{\rm Diff}_b^2$, we have 
  \begin{align*}
  \la Q(\chi_{\epsilon} dt),\omega_{b,s0}\ra=\cO(\epsilon^{1-}). 
  \end{align*}
  To show this we have to consider terms of the form $\rho^3\eta,\, \rho^4\chi'\left(\frac{\rho}{\epsilon}\right)\frac{1}{\epsilon}\eta$ and $\rho^5\chi''\left(\frac{\rho}{\epsilon}\right)\frac{1}{\epsilon^2}\eta$, where $\eta$ is one of the forms $dt,\, dx^1,\, dx^2,\, dx^3$.  The statement then follows from ($\delta>0$):
  \begin{align*}
  \int_0^1\left\vert\rho^{-1/2-\delta}\rho^3\chi_{\epsilon}(\rho)\right\vert\frac{d\rho}{\rho^4}&\lesssim \int_0^{\epsilon}\rho^{1-2\delta}d\rho\lesssim \epsilon^{2(1-\delta)},\\ 
  \int_0^1\left\vert\rho^{-\half-\delta}\rho^4\chi'\left(\frac{\rho}{\epsilon}\right)\right\vert^2\frac{1}{\epsilon^2}\frac{d\rho}{\rho^4}&=\int_0^1\rho^{3-2\delta}\left\vert\chi'\left(\frac{\rho}{\epsilon}\right)\right\vert^2\frac{1}{\epsilon^2}d\rho\\
  &=\epsilon^{2(1-\delta)}\int_0^{1/\epsilon}\rho^{3-2\delta}\vert\chi'(\rho)\vert^2d\rho,\\
  \int_0^1\left\vert \rho^{-\half-\delta}\rho^5\chi''\left(\frac{\rho}{\epsilon}\right)\right\vert^2\frac{1}{\epsilon^4}\frac{d\rho}{\rho^4}&=\epsilon^{2(1-\delta)}\int_0^{1/\epsilon}\rho^{5-2\delta}\vert\chi''(\rho)\vert d\rho.   
  \end{align*} 
  We therefore only have to compute the pairing for
  \begin{align*}
  \tilde{\omega}=((\rho^2D_{\rho})^2+2\rho^3\partial_{\rho})\chi_{\epsilon})dt=-\rho^4\chi''\left(\frac{\rho}{\epsilon}\right)\frac{1}{\epsilon^2}dt. 
  \end{align*}
  Using that 
  \begin{align*}
  \omega_{b,s0}=\frac{r}{\varrho_b^2}dt-\frac{a\sin^2r}{\varrho_b^2}d\varphi+\frac{r_+}{\Delta_b}dr
  \end{align*}
  we find
  \begin{align*}
  G(\tilde{\omega},\omega_{b,s0})=-\frac{\rho^4}{\epsilon^2}\chi''\left(\frac{\rho}{\epsilon}\right)\frac{r}{\varrho_b^2\Delta_b}(r^2+a^2). 
  \end{align*}
  We then compute 
  \begin{align*}
  \la \tilde{\omega},\omega_{b,s0}\ra&=\frac{1}{\epsilon^2}\int\int\int\chi''\left(\frac{\rho}{\epsilon}\right)\frac{r(r^2+a^2)\sin\theta}{\varrho_b^2\Delta_b}d\rho d\theta d\varphi\\
  &=4\pi\int\rho\chi''(\rho)d\rho+\cO(\epsilon^{1-})=4\pi+\cO(\epsilon^{1-}). 
  \end{align*}
  As the result has to be independent of $\epsilon$ this gives \eqref{Eq10NoDualdt}. Let us also note that 
  \begin{align*}
  \pa_t^\flat=\frac{\Delta_b-\bha^2\sin^2\theta}{\varrho_b^2}dt+\frac{2\bha\bhm r\sin^2\theta}{\varrho_b^2}d\varphi
  \end{align*}
  and thus 
  \begin{align*}
  \pa_t^\flat-dt\in \Hbext^{\infty,-\half-}, \quad \wh \Box^*_{g_b,1}(0) \chi( \pa_t^\flat-dt)\in \Hbext^{\infty,3/2-}. 
  \end{align*}
  Therefore replacing $\chi dt$ by $\chi  \pa_t^\flat$ in \eqref{Eq10NoDualdt} gives the same result. 
  Now, $\omega^*\in\ker\wh{\Box_{g_b}}(0)^*\cap\Hbsupp^{-\infty,-3/2-}$ can be written as $\omega^*=\chi v+\tilde\omega^*,\, \tilde{\omega}^*\in \Hbsupp^{-\infty,-1/2-},$ where $v=v_0\,d t+v'$ with $v_0\in\C$ and $v'$ a linear combination of $d x^1,d x^2,d x^3$. Upon subtracting $\omega_{b,s 1}^*(\scal)$ for a suitable $\scal\in\scalspace_1$, we can assume $v'=0$. Therefore
  \[
    v_0\wh{\Box_{g_b,1}}(0)^*(\chi\,d t)=-\wh{\Box_{g_b,1}}(0)^*\tilde\omega^*  \]
  is necessarily orthogonal to $\ker\wh{\Box_{g_b,1}}(0)\cap\Hbext^{\infty,-1/2-}=\la\omega_{b,s 0}\ra$, which in view of~\eqref{Eq10NoDualdt} implies $v_0=0$, thus $\omega^*=\tilde\omega^*$ is a scalar multiple of $\omega_{b,s 0}^*$ by Theorem~\ref{Thm1}.
\end{proof}
\begin{prop}
\label{Prop10Genv1}
  There exist families
  \[
    \omega_{b,v 1}(\vect) \in \ker\wh{\Box_{g_b,1}}(0)\cap\Hbext^{\infty,-5/2-}, \quad
    \omega_{b,v 1}^*(\vect) \in \ker\wh{\Box_{g_b,1}}(0)^*\cap\Hbsupp^{-\infty,-5/2-},
  \]
  linear in $\vect\in\vectspace_1$, which satisfy 
  
  \begin{align}
  \label{Eq10Genv1}
    \omega_{b_0,v 1}(\vect)=r^2\vect,\quad
    \omega_{b_0,v 1}^*(\vect)=r^2\vect H(r-2\bhm),\\
     \omega_{b,v 1}(\vect)- \omega_{b_0,v 1}(\vect)\in \Hbext^{\infty,-1/2-},\quad  \omega_{b,v 1}^*(\vect)- \omega_{b_0,v 1}^*(\vect)\in \Hsupp^{-\infty,-3/2-}.
  \end{align}
  and which are such that $\delta_{g_b}^*\omega_{b,v 1}(\vect)\in\Hbext^{\infty,1/2-}$ and $\delta_{g_b}^*\omega_{b,v 1}^*(\vect)\in\Hbsupp^{-\infty,-1/2-}$.
\end{prop}
The proof is strictly analogous to the proof of \cite[Proposition 7.10]{HHV}, we omit the details. In particular the decay properties are already obtained in this proof. Note however that we give up one decay order for $\omega_{b,v 1}^*(\vect)$ with respect to $\omega_{b,v 1}(\vect)$. 

 \section{The linearized Einstein equation} 
 In this section we prove the main theorem already stated in section \ref{SecMT}.
 \label{SecEinst}
 \subsection{Main theorem}
 \begin{thm}
\label{ThmMS}
  Let $0<\bha<\bhm$, $\sigma\in\C$, $\Im\sigma\geq 0$, and suppose $\dot g=e^{-i\sigma t_*}\tilde{h},\, \tilde{h}\in \Hbext^{\infty,q}(X;S^{2}\wt{\Tsc^*X}),$ $q\in (-3/2;-1/2)$ is an outgoing mode solution of the linearized Einstein equation
  \begin{equation}
  \label{EqMSLinEq}
    D_{g_b}\Ric(\dot g)=0.
  \end{equation}
  Then there exist parameters $\dot\bhm\in\R$, $\dot\bha\in\R^3$, and an outgoing 1-form $\omega\in \Hbext^{\infty,q-1}(X;\wt{\Tsc}^*X)$, such that
  \begin{equation}
  \label{EqMSLinDec}
    \dot g-\dot g_{(\bhm,\bha)}(\dot\bhm,\dot\bha)=\delta_{g_b}^*\omega,
  \end{equation}
  where $\dot g_{(\bhm,\bha)}(\dot\bhm,\dot\bha)$ is defined in equation \eqref{gdoteqn}.
\end{thm}
\subsection{Link to the gauge fixed linearized Einstein operator} 
\label{Sec6.2}
Let 
\begin{align*}
L_{g_b}:=2(D_{g_b}{\rm Ric}+\delta^*_{g_b}\delta_{g_b}G_{g_b})
\end{align*}
be the linearized Einstein operator around the Kerr metric in the wave map/De Turck gauge. Here $G_{g_b}={\bf 1}-\frac{1}{2}g_b {\rm tr}_{g_b}$ denotes the trace reversal operator in $4$ spacetime dimensions. We start with the following
\begin{prop}
\label{prop5.1}
Suppose $\dot{g}\in \Hbext^{\infty,q}(X),\, q\in (-3/2;-1/2)$. If $\wh{L_{g_b}}(0)\dot{g}=0$, then there exists
$g_0\in C^{\infty}(\partial X;S^{2}\wt{\Tsc_{\partial X}^*X})$ such that  $\dot{g}-\rho g_0\in \cA^{2-}$. 
\end{prop} 
\begin{proof}
The proof is in principle the same as the proof of \cite[Proposition 4.4]{HHV}, we refer to this proof for details.  We note that in the full subextreme range of $\bha$ the normal operator of $\wh{L_{g_{b}}}(0)$ is the negative euclidian Laplacian tensored with the $10\times 10$ identity matrix when working in the standard coordinate trivialization.  The boundary spectrum of the scalar euclidian Laplacian is, by definition, the divisor of $\wh{\rho^{-2}\Delta}(\lambda)^{-1},$ where the hat stands for Mellin transform in $\rho$ and $-\Delta=\rho^2\rho D_{\rho}\rho D_{\rho}+i\rho^3D_{\rho}-\rho^2\slDelta$ is the positive euclidean Laplacian. Decomposing functions on $\partial X$ into spherical harmonics, and denoting by $\Sph_l$ 
a degree $l\in \N_0$ spherical harmonic, we have
\[\rho^{-i\lambda}(-\rho^{-2}\Delta)(\rho^{i\lambda}\Sph_l)=(\lambda(\lambda+i)+l(l+1))\Sph_l,\]
which vanishes for $\lambda=il,\, \lambda=-i(l+1)$. With respect to the choice of our Sobolev spaces only the $\lambda=-i(l+1)$ are relevant for the expansion, the $l=0$ spherical harmonics gives the $\rho$ term in the expansion. 
\end{proof}
\begin{lemma}
\label{lem5.1}
Suppose that $\dot{g}\in \Hbext^{\infty,q}(X;S^{2}\wt{\Tsc^*X}),\, q\in (-3/2,-1/2)$ is a stationary solution of the linearized Einstein equation 
\[D_{g_b}{\rm Ric}(\dot{g})=0.\]
\begin{enumerate}
\item  For all $\dot{\bha}\in \R^3$, there exists $\lambda(\dot{\bha})\in \R$, a $1-$form $\omega\in \Hbext^{\infty,q-1}(X;\wt{\Tsc}^*X)$ and $g_0\in C^{\infty}(\partial X;S^{2}\wt{\Tsc_{\partial X}^*X})$  such that 
\[\dot{g}-\delta_{g_b}^*\omega-\dot{g}_b(\lambda(\dot{\bha}),\dot{\bha})-\rho g_0\in \cA^{2-}.\] 

\item The Teukolsky scalars $\hat{\Psi}_{[\pm 2]}$ are zero.
\end{enumerate}
\end{lemma}
\begin{proof}
\begin{enumerate}
\item By adding a linearized Kerr metric and a pure gauge solution we want to correct $\dot{g}$ to a solution of the gauge fixed operator. More precisely we are looking for a $1-$form $\omega$ and parameters $\dot{\bha}\in \R^3$ and $\lambda(\dot{\bha})\in \R$, such that 
\[L_{g_b}(\dot{g}-\delta_{g_b}^*\omega-\dot{g}_b(\lambda(\dot{\bha}),\dot{\bha}))=0.\]
This equation is satisfied provided 
\begin{equation}
\label{5.1}
\hat{\Box}_{g_b,1}(0)\omega=-2\delta_{g_b}G_{g_b}\dot{g}+2\delta_{g_b}G_{g_b}\dot{g}_b(\lambda(\dot{\bha}),\dot{\bha}).
\end{equation}
To arrange that the R.H.S. lies in the image of $\hat{\Box}_{g_b}(0)$ we need
\begin{align*}
\la-\delta_{g_b}G_{g_b}\dot{g}+\delta_{g_b}G_{g_b}\dot{g}_b(\lambda(\dot{\bha}),\dot{\bha}),\omega_{b,s_0}^*\ra=0.
\end{align*}
If $\la\delta_{g_b}G_{g_b}\dot{g}_b(1,0),\omega_{b,s_0}^*\ra\neq 0$ we can arrange this by choosing 
\begin{align*}
\lambda(\dot{\bha})=\frac{\la\delta_{g_b}G_{g_b}\dot{g}-\delta_{g_b}G_{g_b}\dot{g}_b(0,\dot{\bha}),\omega_{b,s_0}^*\ra}{\la\delta_{g_b}G_{g_b}\dot{g}_b(1,0),\omega_{b,s_0}^*\ra}.
\end{align*}
We now compute $\la\delta_{g_b}G_{g_b}\dot{g}_b(1,0),\omega_{b,s_0}^*\ra$. We have for $r\le 3\bhm$
\begin{align*}
\dot{g}_b(1,0)&=-\frac{2r}{\varrho_b^2}(dt_*-a\sin^2\theta d\varphi_*)^2.
\end{align*}
First note that $tr_{g_b}(\dot{g}_b(1,0))=0$ so that $G_{g_b}\dot{g_b}(1,0)=\dot{g}_b(1,0).$
Let 
\begin{align*}
\hat{n}^{\mu}&=-\frac{1}{\sqrt{2}}\partial_r,\\
\hat{n}_{\mu}&=\frac{1}{\sqrt{2}}(dt_*-a\sin^2\theta d\varphi_*). 
\end{align*}
We then have 
\begin{align*}
(\dot{g}_b(1,0))_{\mu\nu}=-\frac{4r}{\varrho_b^2}\hat{n}_{\mu}\hat{n}_{\nu}.
\end{align*}
Now,
\begin{align*}
\nabla^{\mu}(\dot{g}_b(1,0))_{\mu\nu}&=-4\nabla^{\mu}\left(\frac{r}{\varrho_b^2}\hat{n}_{\mu}\hat{n}_{\nu}\right)
=-4\left(\hat{n}^{\mu}\nabla_{\mu}\left(\frac{r}{\varrho_b^2}\right)\hat{n}_{\nu}+\frac{r}{\varrho_b^2}\nabla^{\mu}\hat{n}_{\mu}\hat{n}_{\nu}\right).
\end{align*}
Note that 
\begin{align*} 
\hat{n}^{\mu}\nabla_{\mu}r=-\frac{1}{\sqrt{2}}
\end{align*} 
Recall that the volume element of the Kerr metric is $\sqrt{|\det g_b|} = \varrho_b^2 \sin\theta$. This gives  
\begin{align*}
\nabla^{\mu}\hat{n}_{\mu}={} \frac{1}{\varrho_b^2\sin\theta} \partial_\mu \left (\varrho_b^2 \sin\theta \hat n^\mu\right )  ={} -\sqrt{2} \frac{r}{\varrho_b^2}. 
\end{align*}
This gives 
\begin{align*}
\nabla^{\mu}(\dot{g}_b(1,0))_{\mu\nu}=\frac{2}{\varrho_b^2}(\partial_r)_{\nu}
\end{align*}
and thus 
\begin{align}
\label{EC1}
\la\delta_{g_b}G_{g_b}\dot{g}_b(1,0),\omega_{b}^*\ra&=\int_{r_0}^{\infty}\int_{\Sph^2}g^{\gamma\nu}\nabla^{\mu}(\dot{g}_b(1,0))_{\mu\nu}(\omega^*_{b,s_0})_{\gamma}(r^2_++a^2\cos^2\theta) \sin\theta dr d\theta d\varphi_*\nonumber\\
&= \int_0^\pi  \int_0^{2\pi} 2 \sin\theta d\theta d\varphi_* =8\pi\neq 0.
\end{align}
We can therefore choose $\lambda(\dot{\bha})$ in the above manner and obtain a solution $\omega\in \Hbext^{\infty,q-1}(X)$ of \eqref{5.1}. We now apply Proposition \ref{prop5.1} to see that 
\[\dot{g}-\delta_{g_b}^*\omega-\dot{g}_b(\lambda(\dot{\bha)},\dot{\bha})-\rho g_0\in \cA^{2-}\]
for some suitable $g_0\in C^{\infty}(\partial X;S^{2}\wt{\Tsc_{\partial X}^*X})$.
\item The Teukolsky scalars $\hat{\Psi}_{[\pm 2]}$ have the required regularity to apply Theorem \ref{Thm3.1}. 
\end{enumerate}
\end{proof}
\begin{rmk}[Boyer-Lindquist coordinates]
If $\tilde{h}\in \cA^{2-}(X;S^{2}\wt{\Tsc^*X})$, then the coefficients of $\tilde{h}$ in the Boyer-Lindquist representation fulfill 
\[\tilde{h}_{\alpha\beta}=\cO\left(\frac{1}{r^{2-\epsilon}}\right),\, \forall \epsilon>0.\]
Indeed near infinity we have $dt_*=dt-\frac{r^2+a^2}{\Delta_b}dr=dt-dr+\cO\left(\frac{1}{r}\right)$. 
\end{rmk}
\subsection{Gauge invariants}\label{gauge_inv}
The gauge invariants of linearized gravity on the Kerr spacetime have been completely classified in  \cite{2018PhRvL.121e1104A, 2019arXiv191008756A} (see Appendix \ref{prop6.5Ap} for further details).  We will mainly be interested in linearized vacuum perturbations with $\Phi_{2}=\Phi_{-2}=0$. By \cite[Corollary 3]{2018PhRvL.121e1104A} the only non-vanishing gauge invariants for such perturbations are those given by $\II_\xi$, $\II_\zeta$, which in our notation take the form 
\begin{align}
\label{Invxi}
\InvSymb_\xi={}&- \bhp (\varrho'\tho
 + \varrho\tho'
 -  \tau'\edt
 -  \tau\edt' )(\bhp^4\vartheta \Psi_{0})
 -  \tfrac{1}{2} \Psi_{0} \bhp^5 \vartheta \Psi_{0}\nonumber\\
& -  \tfrac{1}{2} \bar\Psi_{0} \overline{\bhp}^5 \overline{\vartheta \Psi}_{0}
 + \tfrac{3}{2} \Psi_{0} \bhp^5 (h_{nn} \varrho^2
 + 2 h_{\ell n} \varrho \varrho'
 + h_{\ell \ell} \varrho'^2\nonumber\\
&{}\hspace{2ex}{} - 2 h_{n\bar{m}} \varrho \tau
 - 2 h_{\ell\bar{m}} \varrho' \tau
 + h_{\bar{m} \bar{m}} \tau^2
 - 2 h_{nm} \varrho \tau'\nonumber\\
&{}\hspace{2ex}{} - 2 h_{\ell m} \varrho' \tau'
 + 2 h_{m\bar{m}} \tau \tau'
 + h_{mm} \tau'^2),
\end{align}
and, with $\bhp_+ = \bhp + \overline{\bhp}, \bhp_- = \bhp - \overline{\bhp}$,
\begin{align}
\label{Invzeta}
\InvSymb_\zeta={}&\tfrac{1}{4} \bhp \big( \bhp_-^2 (\varrho'\tho
 + \varrho\tho' ) -  \bhp_+^2 (\tau'\edt
  + \tau\edt' ) \big)(\bhp^4\vartheta \Psi_{0})\nonumber\\
& + \tfrac{1}{4} \Re\Bigl(\bhp^5 \vartheta \Psi_{0} \bigl(
 \Psi_{0} (\bhp^2 + \overline{\bhp}^2) - 2 \bar\Psi_{0} \overline{\bhp}^2
 - 4 \bhp (\bhp_- \varrho \varrho' \nonumber\\
 &{}\hspace{2ex}{} -\bhp_+ \tau \tau')\bigr)\Bigr)
 + 2i \Im\bigl(\bhp^6 \overline{\bhp} (\vartheta \Psi_{-1} \varrho \tau
 + \vartheta \Psi_{1} \varrho' \tau')\bigr)\nonumber\\
 & -  \tfrac{3}{8} \Psi_{0} \bhp^5 \bigl( 
 \bhp_-^2 (h_{nn} \varrho^2  +  2 h_{ln} \varrho\varrho'  + h_{ll} \varrho'^2)\nonumber\\
 &{}\hspace{2ex}{} -2 (\bhp^2 + \overline{\bhp}^2) (h_{n\bar{m}} \varrho \tau
   + h_{l\bar{m}} \varrho' \tau
  + h_{nm} \varrho \tau' + h_{lm} \varrho' \tau')\nonumber\\
 &{}\hspace{2ex}{} + \bhp_+^2 (h_{\bar{m} \bar{m}} \tau^2
  +  2 h_{m\bar{m}} \tau\tau' + h_{mm} \tau'^2) \bigr),
\end{align}
Recall that the Plebanski-Demianski family of line elements \cite{PD} are vacuum metrics of Petrov type D, parametrized by $\mfrak, \afrak, \nfrak, \cfrak$, that reduces to the Kerr family of line elements in case $\nfrak = \cfrak = 0$. The form of $\II_\xi$, $\II_\zeta$ for explicit $h$ defined by  perturbations with respect to $\dot \mfrak, \dot \afrak, \dot \nfrak, \dot \cfrak$ in the Plebanski-Demianski family of line elements  are given by (see \cite[Eq. (24)]{2018PhRvL.121e1104A}):
\begin{enumerate}
\item For pure mass $\dot{\bhm}$ and angular momentum $\dot{\bha}$ perturbations,
the invariants take the form
\begin{equation}
\label{Inv1}
\InvSymb_\xi=\dot{\bhm},\, \InvSymb_\zeta=2\bha^2\dot{\bhm}-3\bhm\bha\dot{\bha}.
\end{equation}
\item For perturbations in the direction of the NUT parameter $\dot{\bhn}$ we obtain 
\begin{equation}
\label{Inv2}
\InvSymb_\xi=-i\dot{\bhn}+\frac{2i\bhm}{\bar{p}}\dot{\bhn},\, \InvSymb_\zeta=-ia^2\dot{\bhn}+a\cos\theta\left(r-2\bhm-\frac{\bhm p}{\bar{p}}\right)\dot{\bhn}.
\end{equation}
\item For perturbations in the $c$ metric direction the invariants take the form :
\begin{equation}
\label{Inv3}
\InvSymb_\xi=\frac{6\bhm^2r\cos\theta}{\bar{p}}\dot{\bhc}+3\bhm(ia+(\bhm-r)\cos\theta)\dot{\bhc},\, \InvSymb_\zeta=\frac{6\bhm^2a^2r\cos^3\theta}{\bar{p}}\dot{\bhc}-3i\bhm a(p^2-r^2\cos^2\theta)\dot{\bhc}. 
\end{equation}

\end{enumerate}
It is a remarkable fact that for linearized vacuum perturbations with $\Phi_{2}=\Phi_{-2}=0$, the general form of $\InvSymb_\xi,\, \InvSymb_\zeta$ is, in fact, the one given by \eqref{Inv1}-\eqref{Inv3}. 
We  have
\begin{prop}[Aksteiner and B\"ackdahl \cite{Aksteiner:GI}]
\label{prop6.5}
Let $\dot g_{ab}$ be a vacuum type D perturbation on the Kerr background.  Then there exist parameters $\dot{\bhm},\, \dot{\bha},\, \dot{\bhc},\, \dot{\bhn}$ such that 
\begin{align*}
\InvSymb_\xi&=\dot{\bhm}-i\dot{\bhn}+\frac{2i\bhm}{\bar{p}}\dot{\bhn}+\frac{6\bhm^2r\cos\theta}{\bar{p}}\dot{\bhc}+3\bhm(i\bha+(\bhm-r)\cos\theta)\dot{\bhc},\\
\InvSymb_\zeta&=2\bha^2\dot{\bhm}-3\bhm\bha\dot{\bha}-i\bha^2\dot{\bhn}+\bha\cos\theta\left(r-2\bhm-\frac{\bhm p}{\bar{p}}\right)\dot{\bhn}\\
&+\frac{6\bhm^2\bha^2r\cos^3\theta}{\bar{p}}\dot{\bhc}-3i\bhm \bha(p^2-r^2\cos^2\theta)\dot{\bhc}.
\end{align*}
\end{prop}
For completeness, we sketch the proof of Proposition \ref{prop6.5} in Appendix \ref{prop6.5Ap}. 

\begin{lemma} \label{lem:commut}
Let $v_1$ be a linearized vacuum metric perturbation on the Kerr exterior $\mathcal M$ with vanishing gauge invariants  and
$\dot \mfrak = \dot \afrak = \dot \nfrak = \dot \cfrak = 0$. Then there is a gauge vector field $v_0$ on $\mathcal M$ such that 
\begin{align} 
\delta^*_{g_b} v_0 = v_1
\end{align} 
\end{lemma} 
\begin{rmk}
By the results of \cite{2019arXiv191008756A}, the two extreme Teukolsky scalars, linearized Ricci curvature and $\InvSymb_\xi$ and $\InvSymb_\zeta$ constitute a complete set of gauge invariants. It follows from Proposition \ref{prop6.5} that a {linearized vacuum} perturbation with $\boldsymbol{\Phi}_2=\boldsymbol{\Phi}_{-2}=0$ is locally a Plebanski-Demianski line element plus a pure gauge term. An application of Lemma \ref{lem:commut} makes this result global. We therefore obtain a new proof of the result of Wald that linearized vacuum perturbations of the Kerr metric with vanishing extreme Teukolsky scalars are Plebanski-Demianski line elements modulo gauge, see \cite{Wa}.
\end{rmk}

\begin{proof} 
As in \cite{2019arXiv191008756A}, let  $\tilde \Kop_1$ be the operator that sends a linearized metric on the Kerr background to the collection of its gauge invariants, as defined in \cite{2018PhRvL.121e1104A}. From the assumptions, 
\begin{align} 
\tilde \Kop_1 v_1 = 0.
\end{align} 
As shown in \cite[\S 5]{2019arXiv191008756A}, $\tilde \Kop_1$ is equivalent to $\Kop_1$ as defined in \cite[Eq. (4.26b)]{2019arXiv191008756A}, and hence 
\begin{align}\label{eq:v1K1} 
\Kop_1 v_1 = 0 .
\end{align} 
Let
\begin{align}\label{eq:K0def} 
\Kop_0 = \delta_{g_b}^*
\end{align}  denote the Killing operator on the Kerr background. Then, in order to prove the lemma, we must construct a solution to the equation
\begin{align} 
\Kop_0 v_0 ={}& v_1 \in \ker \Kop_1 .
\end{align} 
Let now $\Cop_l, \Dop_l,\Hop_l$, $l = 0,1$ be as in \cite[Def. 13]{2019arXiv191008756A}. In particular, these are local differential operators acting on sections of the bundles $V_l$, $V_l'$, $l=0,1,2$. We have that $V_0 = T^*\mathcal M$, $V_1$ is the space of symmetric 2-tensors, and $V_0'$ is the subbundle of $T^*\mathcal M$ with sections of the form 
\begin{align} 
\alpha \xi + \beta \zeta 
\end{align} 
where $\xi, \zeta$ are the Killing fields on $\mathcal M$, and $\alpha, \beta$ are scalar functions. Then the commuting diagram  \cite[Eq. (4.1)]{2019arXiv191008756A} is valid. The part of this diagram that is relevant for our purpose is 
\begin{equation} \label{diag:DiffComplex4D}
\begin{tikzcd}
 V_0 \ar[swap,shift right]{d}{\Cop_0} \arrow[r, "\Kop_0"]  \& 
 V_1 \ar[swap,shift right]{d}{\Cop_1} \arrow[r, "\Kop_1"] \arrow[l, dashed, bend left=20, "\Hop_0"]  \& V_2
\ar[swap,shift right]{d}{\Cop_2}
 \\
V'_0 \ar[swap,shift right]{u}{\Dop_0} \arrow[r, "\Kop'_0"'] \& 
V'_1 \ar[swap,shift right]{u}{\Dop_1} \arrow[r, "\Kop'_1"'] \arrow[l, dashed, bend right=20, "\Hop'_0"'] \& V_2' \ar[swap,shift right]{u}{\Dop_2}
\end{tikzcd}.
\end{equation} 
Here $\Kop_l'$ are defined in terms of a flat connection on $V_0'$, cf. 
\cite[\S 4]{2019arXiv191008756A}. The operators $\Cop_l, \Cop_l', \Dop_l, \Dop_l'$, $l=0,1,2$, and $\Hop_0, \Hop_0'$ are local differential operators acting on sections of the bundles $V_l, V_l'$, $l=0,1,2$, and the operators $\Kop_l'$ define a twisted deRham complex, 
\begin{align} 
\Kop_l' = d^{\mathbb{D}}_l
\end{align} 
defined in terms of the unique flat connection $\mathbb{D}$ on $V_0'$ such that the Killing fields are parallel with respect to $\mathbb{D}$.  In particular $\Kop_0'$ is defined by the restriction of $\Kop_0$ to $V_0'$, i.e. 
\begin{align} 
\Kop_0' (\alpha \xi + \beta \zeta) ={}& d^{\mathbb{D}} (\alpha \xi + \beta \zeta) \nonumber \\
={}& (d\alpha)  \xi + (d \beta)  \zeta \label{eq:Kopd}
\end{align} 
cf. \cite[Eq. (4.2)]{2019arXiv191008756A}. Thus, $\Kop_0'$ acts on $V_0'$ as two copies of the exterior derivative on scalars. 
 
From \eqref{diag:DiffComplex4D}, we have the identities 
\begin{align} 
\Kop_0 \circ \Dop_0 ={}& \Dop_1 \circ \Kop_0' \label{eq:comm0}, \\ 
\Kop_1' \circ \Cop_1 ={}& \Cop_2 \circ \Kop_1 \label{eq:comm1}. 
\end{align} 
Further, we have the homotopy identity  \cite[Eq. (2.4a)]{2019arXiv191008756A},
\begin{align}\label{eq:homotop} 
\Dop_1 \circ \Cop_1 = \id - \Kop_0 \circ \Hop_0 - \Hop_1 \circ \Kop_1
\end{align} 
Let 
\begin{align} \label{eq:v1'def}
v_1' = \Cop_1 v_1.
\end{align} 
By \eqref{eq:v1K1} and \eqref{eq:comm1}, 
\begin{align} 
\Kop_1' v_1' = 0, 
\end{align} 
and hence since the twisted deRham complex with operators $\Kop_j'$ is exact, the equation 
\begin{align} \label{eq:v0v1}
\Kop_0' v_0' = v_1' 
\end{align} 
has local solutions. In view of \eqref{eq:Kopd}, and the fact that the Kerr exterior $\mathcal{M}$ is simply connected, we may apply the Poincar\'e Lemma to conclude that the equation \eqref{eq:v0v1} has a global solution $v_0'$. 

By \eqref{eq:comm0}, we have 
\begin{align} 
\Kop_0 \Dop_0 v_0' ={}& \Dop_1 \Kop_0' v_0' \\
\intertext{use \eqref{eq:v1'def}}
={}& \Dop_1 \Cop_1 v_1 \\ 
\intertext{use \eqref{eq:homotop} and $\Kop_1 v_1 = 0$}
={}& v_1 - \Kop_0 \Hop_0 v_1. 
\end{align} 
This means that setting 
\begin{align} 
v_0 =  \Dop_0 v_0' +  \Hop_0 v_1 
\end{align} 
gives a solution to 
\begin{align} 
\Kop_0 v_0 = v_1
\end{align} 
on $\mathcal M$. By construction, $v_0$ is globally defined. 
\end{proof} 

We now want to show that the parameters $\dot{\mathfrak n}$ and $\dot{\mathfrak c}$ are zero. This will follow from the

\begin{prop}
\label{propdecayInv}
\begin{enumerate}
\item Let $g_0\in C^{\infty}(\partial X;S^{2}\wt{\Tsc_{\partial X}^*X})$, $h=\frac{g_0}{r}$. Then we have 
\begin{equation*}
\InvSymb_\zeta=\cO(r),\, {\rm Re }\, \InvSymb_\zeta=\cO(1).
\end{equation*} 
\item If $h=\cO\left(\frac{1}{r^{1+\epsilon}}\right)$, then $\InvSymb_\zeta=\cO(r^{1-\epsilon})$. 
\end{enumerate}
\end{prop}
The proof of the above proposition can be found in Appendix \ref{prop6.8Ap}.

\subsection{Proof of Theorem \ref{ThmMS}}
We start with the $\sigma\neq 0$ case. Let $\xi=\partial_{t_*}$ be the stationary Killing field in Kerr.  If $\sigma \ne 0$, then since $\dot g$ is a mode solution by assumption, we have that 
\begin{align} 
\label{Lie1}
\Lie_\xi \dot g = -i \sigma \dot g
\end{align} 
By \cite[Eq. (1)]{2016arXiv160106084A}  we know that there exists a 1-form $\omega$ and a symmetric $2-$tensor $k$ which vanishes when $\hat{\Psi}_{[\pm 2]}$ vanish such that 
\begin{align}
\label{Lie2}
k=\delta_{g_b}^*\omega+\Lie_\xi \dot g. 
\end{align}  
Applying Theorem \ref{Thm3.1} we see that $\hat{\Psi}_{[-2]}=\hat{\Psi}_{[2]}=0$ and therefore \eqref{Lie1} and \eqref{Lie2} give 
\begin{align*}
\dot g=\frac{i}{\sigma}\delta_{g_b}^*\omega.
\end{align*}
$\dot g$ is pure gauge in this case. In particular \eqref{EqMSLinDec} holds with $\mathfrak{\dot m} = \mathfrak{\dot a} = 0$. 

Let us now consider the case $\sigma=0$. Again by Theorem \ref{Thm3.1} we know that $\hat{\Psi}_{[\pm 2]}=0$, i.e. $\dot{g}$ is a type D perturbation. By Proposition \ref{prop6.5}, we know that the gauge invariants $\InvSymb_\xi,\,  \InvSymb_\zeta$ are those of the Plebanski-Demianski line element. We now apply  Lemma \ref{lem5.1} and Proposition \ref{propdecayInv} to see that the parameters $\dot{\bhn}$ and $\dot{\bhc}$ have to be zero. By completeness of the gauge invariants (extreme Teukolsky scalars, linearized Ricci, $\InvSymb_\xi,\,  \InvSymb_\zeta$, see \cite{2019arXiv191008756A}), we know that the linearized metric writes locally as
\begin{align}
\label{5.11}
\dot g=\dot g_b(\dot \bhm,\dot\bha)+\delta^*_{g_b}\omega
\end{align}
for some suitable $\omega$. Applying Lemma \ref{lem:commut} to $\dot g-\dot g_b(\dot \bhm,\dot\bha)$ gives \eqref{5.11} globally on the manifold. This completes the proof of the theorem. 
\qed

\section{The gauge fixed linearized Einstein operator}
\label{Sec7}
In this section we present our results on the mode analysis of the linearized Einstein operator around the Kerr metric in the wave map/De Turck gauge. The results are analogous to those obtained in the small $\bha$ case in \cite{HHV}. Recall that the linearized Einstein operator around the Kerr metric in the wave map/De Turck gauge is given by
\begin{align*}
L_{g_b}:=2(D_{g_b}{\rm Ric}+\delta^*_{g_b}\delta_{g_b}G_{g_b}).
\end{align*}
Here $G_{g_b}={\bf 1}-\frac{1}{2}g_b {\rm tr}_{g_b}$ denotes the trace reversal operator in $4$ spacetime dimensions. In this gauge fixed setting a zero mode solution of $\wh{L_{g_b}}(0)h=0$ writes again as $h=\dot{g}_{(\bhm,\bha)}(\dot{\bhm},\dot{\bha})+\delta_{g_b}^*\omega$, but the pure gauge term now has to lie in a fixed $7$ dimensional space. These gauge solutions correspond to
\begin{enumerate}
\item  the Coulomb solutions of the $1-$form wave operator, representative of a residual gauge freedom,
\item asymptotic translations in space and asymptotic rotations, representatives of symmetries in flat space. 
\end{enumerate}
To parametrize the asymptotic rotations correctly, we will allow perturbations in $\dot{\bhvecta}\in \R^3$, thus including changes in the axis of rotation. Note however that solutions of the form $\dot{g}_b(\dot{\bhm},\dot{\bhvecta}^{\perp})$, where $\dot{\bhvecta}^{\perp}$ is orthogonal to the axis of rotation are pure gauge solutions : they merely describe the same Kerr black hole with rotation axis rotated infinitesimally. On the other hand, $\dot{g}_b(\dot{\bhm},\dot{\bhvecta}^{\parallel})$ where $\dot{\bhvecta}^{\parallel}$ is parallel to the axis of rotation have to be considered as gauge independent solutions (a gauge term nevertheless has to be added to make them solutions of the gauge fixed equation). As we will also see in the following the mass perturbation $\dot{\bhm}$ has in fact to be equal to zero.   
We start our analysis with the Fredholm setting. 
\begin{thm}
\label{ThmOp}
  There exists $m_2>0$ with the following property. Suppose that $m>m_2$ and $q<-\half$ with $m+q>-\half$. Then for any fixed $C>1$, and $m_0<m$, $q_0<q$, there exists a constant $C'>0$ such that
    \begin{equation}
    \label{EqOpFinite}
      \Vert u\Vert_{\Hbext^{m,q}} \leq C'\Bigl(\left\Vert \wh{L_{g_b}}(\sigma)u\right\Vert_{\Hbext^{m-1,q+2}} + \Vert u\Vert_{\Hbext^{m_0,q_0}} \Bigr)
    \end{equation}
    for all $\sigma\in\C$, $\Im\sigma\in[0,C]$, satisfying $C^{-1}\leq |\sigma|\leq C$. If $q\in(-\tfrac32,-\half)$, then this estimate holds uniformly down to $\sigma=0$, i.e.\ for $|\sigma|\leq C$.
  Moreover, the operators
  \begin{subequations}
  \begin{alignat}{3}
  \label{EqOpFredS}
    \wh{L_{g_b}}(\sigma)&\colon\{u\in\Hbext^{m,q}(X)\colon\wh{L_{g_b}}(\sigma)u\in\Hbext^{m-1,q+2}(X)\}&&\to\Hbext^{m-1,q+2}(X), \quad \Im\sigma\geq 0,\ \sigma\neq 0, \\
  \label{EqOpFred0}
    \wh{L_{g_b}}(0) &\colon \{u\in\Hbext^{m,q}(X)\colon\wh{L_{g_b}}(0)u\in\Hbext^{m-1,q+2}(X)\}&&\to \Hbext^{m-1,q+2}(X)
  \end{alignat}
  \end{subequations}
  are Fredholm operators of index $0$.
\end{thm}
\begin{rmk}
Under the same hypotheses as for \eqref{EqOpFinite} we also have the estimate 
\begin{equation}
\label{Fredholm2}
 \Vert u\Vert_{\Hbext^{m,q}} \leq C''\Bigl(\left\Vert \wh{L_{g_b}}(\sigma)u\right\Vert_{\Hbext^{m,q+1}} + \Vert u\Vert_{\Hbext^{m_0,q_0}} \Bigr).
\end{equation}
In fact both estimates \eqref{EqOpFinite} and \eqref{Fredholm2} follow from a more precise estimate using resolved scattering-b Sobolev spaces, see \cite{Va1} and \cite{Va2}. A similar remark holds for the Teukolsky operator and the scalar and $1-$form wave operators. 
\end{rmk}
\begin{proof}
The proof is analogous to the proof of \cite[Theorem 4.3]{HHV}, we recall here the principal ingredients.  It relies in a crucial manner on the properties of the Hamiltonian flow. The principal features of this flow remain unchanged also in the large $\bha$ case. In particular the discussion in \cite[Section 3.4]{HHV} on the flow of the classical symbol remains unchanged in the large $\bha$ case. We refer to \cite{Va1}, \cite{Ga} and \cite{M2} for details of the calculation of the flow. We use the notations introduced in \cite[Section 3.4]{HHV}. In particular the set of radial points which lie in the conormal bundle of the event horizon are called $\cR^{\pm}$ and those at infinity $\cR_{\sigma,in/out}$. 

We first have to consider radial point estimates. At the horizon radial point estimates require the calculation of threshold regularity. The existence of the threshold regularity at the horizon follows from the fact that $\tilde{\beta}$ as defined on page 404 of \cite{Va1} has a maximum and a minimum on the compact set $L_{\pm}$.\footnote{Note that the requirement $\tilde{\beta}>0$ on $L_{\pm}$ in \cite{Va1} is only formulated for notational convenience.} 
This gives the threshold regularity $m_2$ in our theorem.

  For $0\neq\sigma\in\R$, radial point estimates at infinity for $\wh{L_{g_b}}(\sigma)$ similarly require the computation of a threshold decay rate relative to $L^2(X)$. Concretely, the threshold $-\half$ from \cite[Propositions~9 and 10]{MelroseEuclideanSpectralTheory}, \cite{VasyZworskiScl}, \cite[Theorems~1.1 and 1.3]{Va1} is modified by the subprincipal symbol $\sigmasc_1(\frac{1}{2 i\rho}(\wh{L_{g_b}}(\sigma)-\wh{L_{g_b}}(\sigma)^*))|_{\cR_{\sigma,\rm in/out}}$; we now argue that this symbol vanishes. Indeed, formally taking $b=(0,0)$, so $g_b=\ubar g$ is the Minkowski metric, and working in the trivialization of $S^2\,\wt{\Tsc^*}X$ given in terms of the differentials of standard coordinates $t,x^1,x^2,x^3$, the operator $L_{g_b}$ is the wave operator on Minkowski space acting on symmetric 2-tensors, hence a $10\times 10$ diagonal matrix of scalar wave operators, and therefore the subprincipal symbol vanishes when using the fiber inner product on $S^2\,\wt{\Tsc^*}X$ which makes $d t^2$, $2\,d t\,d x^i$, $d x^i\,d x^j$ orthonormal. Changing from the Minkowski metric to a Kerr metric does not affect the subprincipal symbol at $\cR_{\sigma,\rm in/out}$, as already argued in the proof of \cite[Theorem 4.3]{HHV}.   
  Combining the radial point estimates at infinity from \cite{Va1} with those at the event horizon from \cite{VasyMicroKerrdS} (see also \cite[Proposition~2.1]{HintzVasySemilinear}), gives the stated uniform estimates for $\Im\sigma\in[0,C]$, $C^{-1}\leq|\sigma|\leq C$ for any fixed $C>1$. The uniformity of the stated estimate down to $\sigma=0$ is proved in \cite[Proposition~5.3]{Va2}; this uses the invertibility of a model operator, see \cite[\S5]{Va2}, which in the current setting and in the standard coordinate trivialization of $S^2\,\wt{\Tsc^*}X$ is the $10\times 10$ identity matrix tensored with the scalar model operator discussed  in \cite[Proposition~5.4]{Va2}.

It remains to prove that $\wh{L_{g_b}}(\sigma)$ has index $0$ as stated in~\eqref{EqOpFredS}--\eqref{EqOpFred0}. We use a deformation argument, which reduces the index $0$ property of $\wh{L_{g_b}}(\sigma)$ to that of the Fourier-transformed \emph{scalar} wave operator. 
  
  We first treat the case $\sigma=0$: choose a global trivialization of $S^2\,\wt{\Tsc^*}X$, then $\wh{L_{g_b}}(0)$ is a $10\times 10$ matrix of scalar operators in $\rho^2\Diffb^2(X)$, with the off-diagonal operators lying in $\rho^2\Diffb^1(X)$. Since adding an element of $\rho^2\Diffb^1$ to $\wh{L_{g_b}}(0)$ does not change the domain in~\eqref{EqOpFred0}, we can continuously deform $\wh{L_{g_b}}(0)$ within the class of Fredholm operators on the spaces in~\eqref{EqOpFred0} to a diagonal $10\times 10$ matrix with all diagonal entries equal to the scalar wave operator at zero energy, $\wh{\Box_{g_b}}(0)$; the latter operator known to be invertible by Theorem~\ref{Thm0}; in particular, it has index $0$. Thus, $\wh{L_{g_b}}(0)$ has index $0$ as well. The index zero property for $\sigma\neq 0$ follows from the same kind of deformation argument using the invertibility of $\hat{\Box}_g(\sigma)$.  \end{proof}

\begin{rmk}
\begin{enumerate}
\item As has been proven by Dyatlov \cite{Dy} the trapping remains r-normally hyperbolic in the whole subextreme range of $\bha$ so that \cite{DyatlovSpectralGaps}, \cite{WZ} apply. We therefore expect that the high energy estimates of  \cite[Theorem 4.3]{HHV} also remain valid in the whole range of $\bha$. We however postpone this aspect (which is not needed for the mode analysis) to future work. 
\item The threshold regularity at $\cR^{\pm}$ has been calculated in detail for Schwarz\-schild--de~Sitter metrics in~\cite{HV}. As already mentioned in \cite{HHV} the same calculation can be carried out for the Kerr spacetime in the whole subextreme range of $\bha$ and also gives the threshold regularity  $5/2$. For the purpose of this paper we however don't need the exact value of the threshold regularity and therefore avoid this rather lengthly calculation. 

\end{enumerate}
\end{rmk}

We will need the following definition :
\begin{definition}
\label{def7.1}
Given two lorentzian metrics $g,g^0$, we define the gauge $1-$form $\Upsilon$ by   
\begin{align*}
\Upsilon(g;g^0):=g(g^0)^{-1}\delta_gG_gg^0. 
\end{align*}
\end{definition}

\begin{thm}
\label{PropL0}
Let $\bha\bhm\neq 0$ and $m_2$ as in Theorem \ref{ThmOp}.
    \begin{enumerate}
  \item For $\Im\sigma\geq 0$, $\sigma\neq 0$, the operator
    \begin{align*}
      \wh{L_{g_{b}}}(\sigma)\colon&\{\omega\in \Hbext^{m,q}(X;S^2\,\wt\Tsc{}^*X)\colon\wh{L_{g_{b}}}(\sigma)\omega\in\Hbext^{m-1,q+2}(X;S^2\,\wt\Tsc{}^*X)\} \\
        &\qquad \to\Hbext^{m-1,q+2}(X;S^2\,\wt\Tsc{}^*X)
    \end{align*}
    is invertible when $m>m_2$, $q<-\half$, $m+q>-\half$.
  \item For $m>m_2$ and $q\in(-\tfrac32,-\half)$, the zero energy operator
    \begin{equation}
    \label{EqL0Op}
    \begin{split}
      \wh{L_{g_{b}}}(0)\colon&\{\omega\in\Hbext^{m,q}(X;S^2\,\wt\Tsc{}^*X)\colon\wh{L_{g_{b}}}(0)\omega\in\Hbext^{m-1,q+2}(X;S^2\,\wt\Tsc{}^*X)\} \\
        &\qquad \to\Hbext^{m-1,q+2}(X;S^2\,\wt\Tsc{}^*X)
    \end{split}
    \end{equation}
    has 7-dimensional kernel and cokernel.
  \end{enumerate}
  
  Concretely,   
  \begin{subequations}
  \begin{alignat}{3}
  \label{EqL0Ker}
    &\ker\wh{L_{g_b}}(0) \cap \Hbext^{\infty,-1/2-}& &=&\ \la h_{b,s 0}\ra \oplus \{ h_{b,v 1}(\vect)\colon\vect\in\vectspace_1 \} \oplus \{ h_{b,s 1}(\scal)\colon \scal\in\scalspace_1 \}, \\
  \label{EqL0KerAdj}
    &\ker\wh{L_{g_b}}(0)^* \cap \Hbsupp^{-\infty,-1/2-}& &=& \la h_{b,s 0}^*\ra \oplus \{ h_{b,v 1}^*(\vect) \colon \vect\in\vectspace_1 \} \oplus \{ h_{b,s 1}^*(\scal)\colon \scal\in\scalspace_1 \},
  \end{alignat}
  \end{subequations}
with 
  \begin{subequations} \label{eq:hdel*}
  \begin{alignat}{3}
  \label{EqL0s0}
    h_{b,s 0} &= \delta_{g_b}^*\omega_{b,s 0}, & h_{b,s 0}^* &= \sfG_{g_b}\delta_{g_b}^*\omega_{b,s 0}^*, \\
  \label{EqL0s1}
    h_{b,s 1}(\scal) &= \delta_{g_b}^*\omega_{b,s 1}(\scal),\quad & h_{b,s 1}^*(\scal) &= \sfG_{g_b}\delta_{g_b}^*\omega_{b,s 1}^*(\scal), \\
  \label{EqL0v1}
    h_{b,v 1}(\vect)&=\dot g_b(0,\dot \bhvecta)+\delta_{g_b}^*\omega,\quad & h_{b,v 1}^*(\vect) &= \sfG_{g_b}\delta_{g_b}^*\omega_{b,v 1}^*(\vect),
  \end{alignat}
  \end{subequations}
where $\dot \bhvecta$, $\omega\in\Hbext^{\infty,-1/2-}$ depend on $b,\vect$; here $\scal\in\scalspace_1$, $\vect\in\vectspace_1$. We also have 
  \begin{align*}
   h_{b,s 0}\in \Hbext^{\infty,1/2-},\,  h_{b,s 1}(\scal)\in \Hbext^{\infty,-1/2-},\, h_{b,v 1}(\vect)\in \Hbext^{\infty,1/2-}.
  \end{align*}
  At $b=b_0$, we have $h_{b_0,v 1}(\vect)=2\omega_{b_0,s 0}\otimes_s\vect$. The dual states are supported in $r\geq r_+$, $\CI$ in $r>r_+$, conormal at $\pa_+X$ with the stated weight, and lie in $H^{-3/2-}$ near the event horizon.  
  Furthermore, all zero modes are solutions of the linearized Einstein equations and satisfy the linearized gauge condition; that is, $\ker\wh{L_{g_b}}(0)\cap\Hbext^{\infty,-1/2-}\subset\ker D_{g_b}\Ric\cap\ker D_{g_b}\Ups(-;g_b)$ in the notation of Definition \ref{def7.1}.
\end{thm}
\begin{rmk}
Asymptotic boosts are not captured by this theorem. They are generalized mode solutions meaning that polynomial growth in $t_*$ has to be permitted, in particular the boosts having linear growth in $t_*$. In the small $\bha$ case also quadratically growing generalized modes exist which is essentially due to the choice of the gauge. In the small $\bha$ case they could be eliminated by constraint dumping, see sections 9 and 10 of \cite{HHV}.  We postpone the analysis of these generalized modes in the large $\bha$ case to future work.  
\end{rmk}
\begin{rmk}
Recall from the beginning of this section that the solution in  \eqref{EqL0v1} can also be written as $h_{b,v 1}(\vect)=\dot g_b(0,\dot \bhvecta^{\parallel})+\delta_{g_b}^*\omega$ for some appropriate gauge term $\delta_{g_b}^*\omega$. 
\end{rmk}

We also repeat remark 9.2 of \cite{HHV}: 
 \begin{rmk} 
\label{RmkL0RotVF} 
  The `asymptotic rotations' $\omega_{b,v 1}(\vect)$ of Proposition~\ref{Prop10Genv1} were not used here, even though they give rise to zero energy states $h_b(\vect):=\delta_{g_b}^*\omega_{b,v 1}(\vect)\in\ker\wh{L_{g_b}}(0)\cap\Hbext^{\infty,-1/2-}$. To explain why they are, in fact, already captured by Theorem~\ref{PropL0}, note first that when $b=(\bhm,0)$ describes a \emph{Schwarzschild} black hole, then $\omega_{b,v 1}(\vect)=r^2\vect$ is dual to a rotation, thus Killing, vector field, hence $h_b(\vect)\equiv 0$. On the other hand, when $b=(\bhm,\bha)$ with $\bha\neq 0$, consider the orthogonal splitting $\vectspace_1=\la\pa_\varphi^\flat\ra \oplus \vectspace^\perp$, where $\pa_\varphi$ is unit speed rotation around the axis of rotation; the latter is a Killing vector field for the metric $g_b$, and thus $h_b(\pa_\varphi^\flat)=0$. On the other hand, $\vectspace^\perp\ni\vect\mapsto h_b(\vect)$ is now injective; that this does \emph{not} give rise to new (i.e.\ not captured by Theorem~\ref{PropL0}) zero energy states is due to the fact that for such  $b$, the parametrization of the linearized Kerr family $\R^4\ni(\dot{\bhm},\dot{\bhvecta})\mapsto\dot g_b(\dot{\bhm},\dot{\bhvecta})$ is no longer injective when quotienting out by pure gauge solutions, but rather has a 2-dimensional kernel. As already explained at the beginning of this section, if $\dot\bhvecta$ is orthogonal to the axis of rotation, then $\dot g_{(\bhm,\bha)}(0,\dot\bhvecta)$ is pure gauge: it merely describes the same Kerr black hole with rotation axis rotated infinitesimally, i.e.\ is precisely of the form $h_b(\vect)$ for $\vect\in\vectspace^\perp$ (plus an extra pure gauge term depending on the presentation of the Kerr family). In summary then,
  \[
    h_{b,v 1}(\vectspace_1) + h_b(\vectspace_1) = h_{b,v 1}(\vectspace_1),\quad b=(\bhm,\bha),
  \]
  is 3-dimensional for $\bha=0$ as well as for $\bha\neq 0$.
\end{rmk}

\begin{proof}
  
  Consider a non-zero frequency mode solution $\wh{L_{g_b}}(\sigma)h=0,\, \Im\sigma\geq 0$. We will put $\dot{g}=e^{-i\sigma t_*}h$.

  The linearized second Bianchi identity implies
  \[
    \delta_{g_b}\sfG_{g_b}\delta_{g_b}^*(\delta_{g_b}\sfG_{g_b} \dot{g})=0.
  \]
  If $\sigma\neq 0$, then $\delta_{g_b}\sfG_{g_b} \dot{g}$ is an outgoing mode; if $\sigma=0$, then $\hat{\delta}_{g_b}(0)\sfG_{g_b} h\in\Hbext^{\infty,1/2-}$. Indeed we have 
  \begin{align*}
   \wh{\delta_{g_b}}(0) \in \rho\Diffb^1(X;S^2\,\wt\Tsc{}^*X,\wt\Tsc{}^*X), \\
     \wh{\delta_{g_b}^*}(0) \in \rho\Diffb^1(X;\wt\Tsc{}^*X,S^2\,\wt\Tsc{}^*X).
\end{align*}
This can be shown as in the small $\bha$ case, see \cite[equation (3.42)]{HHV}. 
  In both cases, Theorem~\ref{Thm1} and the fact that the generator $\omega_{b,s0}$ of the kernel does \emph{not} lie in $\Hbext^{\infty,1/2-}$ imply\
    \begin{equation}
  \label{EqL0LinGauge}
    \delta_{g_b}\sfG_{g_b} \dot{g} = 0
  \end{equation}
  and thus also
  \begin{equation}
  \label{EqL0LinEin}
    D_{g_b}\Ric(\dot{g}) = 0.
  \end{equation}

  Next, we apply the mode stability result, Theorem~\ref{ThmMS}. Consider first the case $\sigma\neq 0$; then $\dot{g}=\delta_{g_b}^*\omega$ with $\omega$ an outgoing mode; plugging this into \eqref{EqL0LinGauge}, we obtain $\Box_{g_b,1}\omega=0$ and hence $\omega=0$ by Theorem \ref{Thm1}, thus $h=0$. This proves the injectivity of $\wh{L_{g_b}}(\sigma)$ for non-zero $\sigma$ with $\Im\sigma\geq 0$, hence its invertibility by Theorem \ref{ThmOp}.

  Suppose now $\sigma=0$, that is we consider 
\begin{align}
   \label{zmLEG}
   \wh{L_{g_b}}(0)h=0.
\end{align}
 By theorem \ref{ThmMS} we know that 
  \begin{align*}
  h=\dot{g}_b(\dot{\bhm},\dot{\bhvecta})+\delta_{g_b}^*\omega,\, \omega\in \Hbext^{m,q-1}. 
  \end{align*}
  Plugging this into \eqref{EqL0LinGauge} gives 
  \begin{align}
\label{Eqomega}
  \wh{\Box}_{g_b,1}(0)\omega=-2\delta_{g_b}G_{g_b}\dot{g}_b(\dot{\bhm},\dot{\bhvecta}).
  \end{align}
Pairing with $\omega^*_{b,s0}$ gives 
\begin{align*}
0&=-2\la\delta_{g_b}G_{g_b} \dot{g}_b(\dot{\bhm},\dot{\bhvecta}),\omega_{b,s0}^*\ra.
 \end{align*} 
 This entails 
 \begin{align*}
 \dot{\bhm}=-\frac{\la \delta_{g_b}G_{g_b}  \dot{g}_b(0,\dot{\bhvecta}),\omega_{b,s0}^*\ra}{\la\delta_{g_b}G_{g_b} \dot{g}_b(1,0),\omega_{b,s0}^*\ra}=-\frac{1}{8\pi}\la \delta_{g_b}G_{g_b}  \dot{g}_b(0,\dot{\bhvecta}),\omega_{b,s0}^*\ra,
 \end{align*}
 where we have used \eqref{EC1}. An explicit calculation shows\footnote{Calculation realized with maple.}
 \begin{align}
 \label{7.11}
 \frac{1}{8\pi}\la \delta_{g_b}G_{g_b}  \dot{g}_b(0,\dot{\bhvecta}),\omega_{b,s0}^*\ra=0
 \end{align}
 and thus $ \dot{\bhm}=0$. Note that
  the general solution of \eqref{Eqomega} writes as an element of the kernel given by proposition \ref{Prop10Grow} plus a ``special solution''. These special solutions are parametrized by $\dot{\bhvecta}\in \R^3$.  This shows that every solution of $\wh{L_{g_b}}(0)h=0$ is of the form \eqref{EqL0s0}-\eqref{EqL0v1}. We now have to show that  \eqref{EqL0s0}-\eqref{EqL0v1} define indeed solutions of \eqref{zmLEG}. $h_{b,s_0}$ and $h_{b,s_1}$ are solutions of both  \eqref{EqL0LinGauge} (by construction of $\omega_{b,s_0}$ and $\omega_{b,s_1})$ and  \eqref{EqL0LinEin} (as pure gauge solutions). 
   
 It remains to construct a continuous family (in $b$) of elements of $\ker\wh{L_{g_b}}(0)\cap\Hbext^{\infty,-1/2-}$ extending $h_{b_0,v 1}(\vect)$. For $\vect\in\vectspace_1$ which is (dual to) the rotation around the axis $\dot\bhvecta\in\R^3$ (with $\vect$ having angular speed $|\dot\bha|$), we make the ansatz
  \begin{equation}
  \label{EqL0v1Ansatz}
    h_{b,v 1}(\vect) = \dot g_b(0,\dot\bhvecta) + \delta_{g_b}^*\omega,
  \end{equation}
  with $\omega\in\Hbext^{\infty,-3/2-}$ to be found. The equation $\wh{L_{g_b}}(0)h_{b,v 1}(\vect)=0$ is then satisfied provided\footnote{The RHS has been computed explicitly with maple.}
 
  \begin{equation}
  \label{EqL0v1Kerr}
    \wh{\Box_{g_b}}(0)\omega = -2\delta_{g_b}\sfG_{g_b}\dot g_b(0,\dot\bhvecta) \in \Hbext^{\infty,3/2-}. 
      \end{equation}
  In view of Theorem~\ref{Thm1}, the obstruction for solvability of this is the cokernel $\ker\wh{\Box_{g_b}}(0)^*\cap\Hbsupp^{-\infty,-1/2+}=\la\omega_{b,s 0}^*\ra$. In view of \eqref{7.11} the RHS is in the image of  $\wh{\Box_{g_b}}(0)$
  and the equation can be solved with some $\omega\in \Hbext^{\infty,-1/2-}$. 
   By theorem \ref{ThmOp}, $\wh{L_b}(0)$ is Fredholm of index zero, therefore the cokernel has dimension 7. It can then be checked that the elements of the RHS of \eqref{EqL0s0}-\eqref{EqL0v1} are elements of the cokernel and have the required properties, we refer to the proof of \cite[Proposition 9.1]{HHV} for details. 
   \end{proof}
   \begin{rmk}
 Note that when changing coordinates by $\ft=t_*+F$, $F\in\CI(X^\circ_b)$, $\dot{g}_b(0,\dot{\bhvecta})$ will be changed by a gauge term $\delta_{g_b}^*\omega$. Changing $\dot{g}_b(0,\dot{\bhvecta})$ by a gauge term $\delta_{g_b}^*\omega\in\Hbext^{\infty,-1/2-}$ will not change the pairing in \eqref{7.11}. Indeed 
 \begin{align*}
 -2\la \delta_{g_b}G_{g_b}\delta_{g_b}^*\omega,\omega_{b,s0}\ra=\la\wh{\Box_{g_b}}(0)\omega,\omega_{b,s0}^*\ra=\la\omega,\wh{\Box_{g_b}}^*(0)\omega^*_{b,s0}\ra=0.
 \end{align*} 

 This also means that \eqref{7.11} only has to be computed for perturbations $\dot{\bhvecta}$, which are parallel to the axis of rotation, see Remark \ref{RmkL0RotVF}.
 \end{rmk}

\appendix 

\section{Proof of proposition \ref{prop6.5}}\label{prop6.5Ap}

In \cite{2018PhRvL.121e1104A}, two complex scalar gauge invariants, $\InvSymb_\xi$ and $\InvSymb_\zeta$, were presented for perturbations of the Kerr spacetime.  
The authors also identified specific curvature invariants that reduce to these gauge invariants in the linearized theory.  As already indicated in section \ref{gauge_inv}, these invariants are sensitive to variations of the Kerr parameters.  Together with the Teukolsky scalars, $\boldsymbol{\Phi}_{\pm 2}\equiv\vartheta\Psi_{\pm 2}$ (in the notation of \eqref{varPsis}), and the linearized Ricci tensor, (denoted by the linearized Ricci spinor $\vartheta\Phi_{ABA'B'}$ in the NP formalism), they form a minimal set that generates all local gauge invariants.  For Proposition \ref{prop6.5}, we are interested in vacuum, type D, perturbations, for which both $\vartheta\Phi_{ABA'B'}=0$ (vacuum) and $\vartheta\Psi_{\pm 2}=0$ (Type D).  In this Appendix, we discuss vacuum, type D, perturbations of Kerr in Boyer-Lindquist coordinates $(t,r,x=\cos\theta,\phi)$. By comparison with equation (24a-e) of \cite{2018PhRvL.121e1104A}, the equations \eqref{eq:IIIAlgSpecSol} below, for $\InvSymb_\xi$ and $\InvSymb_\zeta$, show that there are then only perturbations within the Plebanski-Demianski family. This confirms the classical result of Wald \cite{Wa} obtained by using a different technique. We are grateful to Steffen Aksteiner \cite{Aksteiner:private} for providing this argument, cf. \cite{Aksteiner:GI}.

The extra notation introduced here has been defined in \cite{2016arXiv160106084A}.  
Denote the linearized metric by $h_{ab}$ and its trace-free and trace parts by $\htf_{ab}, \slashed{h}$, respectively. 

Let $\mathcal{K}^i$ be the projection operators defined in 
\cite[\S II.D]{2019JMP....60h2501A}. Let $\kappa_{AB}$ be the Killing spinor in the Kerr spacetime, and let $\kappa_0$ be the corresponding spin-weight zero scalar  
so that $\kappa_{AB} = -2 \kappa_0 o_{(A} \iota_{B)}$ for a principal dyad $o_A, \iota_B$, cf. \cite[Eq. (21)]{2019JMP....60h2501A}. The operators $\mathcal{K}^i$ acting on a spinor $\varphi_{A \dots D A' \cdots D'}$ is defined, up to a normalization, by tensoring with $\kappa_0^{-1} \kappa_{AB}$ and contracting $i$ indices. For example, for $\varphi_{ABA'B'}$, we have 

\begin{align} 
(\mathcal{K}^0 \varphi)_{ABCDA'B'} ={}& 2 \kappa_0^{-1} \kappa_{(AB} \varphi_{CD)A'B'} \\  
(\mathcal{K}^1 \varphi)_{ABA'B'} ={}& \kappa_0^{-1} \kappa_{(A}{}^F \varphi_{B)FA'B'} \\ 
(\mathcal{K}^2 \varphi)_{A'B'} ={}& -\tfrac{1}{2} \kappa_0^{-1} \kappa^{AB} \varphi_{ABA'B'} 
\end{align} 
We shall also need the fundamental spinor operators $\sTwist, \sCurl, \sCurlDagger$, cf. \cite[\S II.C]{2019JMP....60h2501A}. For example, for a spinor $\varphi_{BC}{}^{B'C'}$ we have  
\begin{align} 
(\sTwist \varphi)_{ABC}{}^{A'B'C'} ={}& \nabla_{(A}{}^{(A'} \varphi_{BC)}{}^{B'C')} \\ 
(\sCurl \varphi)_{ABC}{}^{C'} ={}& \nabla_{(AB'} \varphi_{BC)}{}^{B'C'} \\ 
(\sCurlDagger \varphi)_{C}{}^{A'B'C'} ={}& \nabla^{B(A'} \varphi_{BC}{}^{B'C')}
\end{align} 
Finally, the spin projection operator $\mathcal{P}^i$, cf. \cite[\S II.D]{2019JMP....60h2501A} yields a spinor depending only on the components of spin-weights $\pm i$. Here we shall need only $\mathcal{P}^2$ which when acting on $\varphi_{ABCD}$ takes the form 
\begin{align} 
(\mathcal{P}^2 \varphi)_{ABCD} = (\mathcal{K}^1\mathcal{K}^1\mathcal{K}^1\mathcal{K}^1 \varphi)_{ABCD} - \tfrac{1}{16} (\mathcal{K}^0\mathcal{K}^1\mathcal{K}^1\mathcal{K}^2 \varphi)_{ABCD}. 
\end{align} 
In particular, $(\mathcal{P}^2 \vartheta\Psi)_{ABCD}$ depends only on the scalars $\vartheta\Psi_{\pm 2}$. 

Recall the definition of $\InvSymb_V, V\in\{\xi,\zeta\}$ from Section~5 of 
\cite{2019arXiv191008756A} together with the definition of $\AA^a$ from equation (57a) of 
\cite{2016arXiv160106084A},  assuming vanishing linearized Ricci spinor, $\vartheta\Phi_{ABA'B'} = 0$:
\begin{align} \label{AAdef}
\AA_a={}&- \tfrac{1}{108} M \slashed{h}_{} \xi_a
 -  \tfrac{1}{54} M \xi^{BB'} (\mathcal{K}^0\mathcal{K}^2\htf)_{ABA'B'}\\
 {}&+ \tfrac{2}{3} \kappa_{1}{}^3 \xi^{B}{}_{A'} (\mathcal{K}^1\mathcal{K}^2\vartheta \Psi)_{AB} + (\mathcal{K}^1\sTwist (\kappa_{1}{}^4 \vartheta \Psi_0))_{AA'}.\nonumber
\end{align}
\begin{align} \label{IVDef}
\InvSymb_V={}&-81 \AA^{a} V_{a}
 -  \tfrac{3}{2} \bhm h_{ab} V^{a} \xi^{b}\\
 {}&+54 \Re\left(\kappa_{1}{}^3 V^{AA'} \xi^{B}{}_{A'} (\mathcal{K}^1 \mathcal{K}^2 \vartheta \Psi)_{AB}
 -  \tfrac{3}{2} \kappa_{1}{}^4 (\mathcal{K}^2\sCurl V) \vartheta \Psi_0 \right).\nonumber
\end{align}

\begin{prop}
\label{propA.1}
Assume perturbations with vanishing linearized Ricci spinor, $\vartheta\Phi_{ABA'B'} = 0$, and  
the linearized Weyl scalars 
$\vartheta\Psi_{\pm 2}=0$. Then compatibility conditions between gauge invariants yield the gradient
\begin{align} \label{GradIV}
\nabla_a \InvSymb_V={}&-81i V_{A}{}^{B'} (\sCurlDagger \ImA)_{B'A'}
 + 81i (\sCurlDagger V)^{B'}{}_{A'} \ImA_{AB'}
\end{align}
for $ V\in\{\xi,\zeta\}$.
\end{prop}
\begin{proof}
Introduce the notation $\partial^{\leq n}h$ for a collection of terms containing up to $n$ derivatives of the linearized metric $h_{ab}$. It follows from the classification of gauge invariants, cf. \cite{2019arXiv191008756A,2018PhRvL.121e1104A} that since we consider only vacuum perturbations with $\vartheta\Psi_{\pm 2}=0$, all gauge invariants of at most second differential order in $h_{ab}$ are zero by assumption. This implies that we can prove \eqref{GradIV} up to $\partial^{\leq 2}h$ terms which will, by construction and gauge invariance, be automatically zero in the final step . We do this computation in several steps and the terms in $\partial^{\leq n}h$ may differ from line to line.

The gradient of \eqref{IVDef} is of the form 
\begin{align} \label{GradIVStep1}
\nabla_c\InvSymb_V={}&-81 \nabla_c( \AA^{a} V_{a} ) \\
 {}& +54 \Re\left(\kappa_{1}{}^3 V^{AA'} \xi^{B}{}_{A'} \nabla_c(\mathcal{K}^1 \mathcal{K}^2 \vartheta \Psi)_{AB}
 -  \tfrac{3}{2} \kappa_{1}{}^4 (\mathcal{K}^2\sCurl V) \nabla_c\vartheta \Psi_0 \right) + \partial^{\leq 2}h.\nonumber
\end{align}

We compute the three non-trivial terms, $\nabla_c\vartheta \Psi_0, \nabla_c(\mathcal{K}^1 \mathcal{K}^2 \vartheta \Psi)_{AB}, \nabla_c( \AA^{a} V_{a} )$, separately to finally recombine them in \eqref{GradIVStep1} to proof the result.

The first term is straightforward, as from \eqref{AAdef} we have
\begin{align} \label{gradPsi2dot} 
 \nabla_a \vartheta \Psi_0 = (\sTwist \vartheta \Psi_0)_{AA'} ={}& \kappa_{1}{}^{-4} (\mathcal{K}^1\AA)_{AA'} + \partial^{\leqslant 2}{}h.
\end{align}
For the second term, we have to use linearized Bianchi identities. By Lemma~3.1 of 
\cite{2016arXiv160106084A} and because we assume $\vartheta\Phi_{ABA'B'} = 0$, the linearized Bianchi identities are of the form
\begin{align} \label{linBianchi}
(\sCurlDagger \vartheta \Psi) = \partial^{\leq 1}h,
\end{align}
and we collect some consequences:
\begin{itemize}
\item Applying $\mathcal{K}^1\mathcal{K}^2$ to \eqref{linBianchi} we find
\begin{align}
\sCurlDagger \mathcal{K}^1\mathcal{K}^2\vartheta \Psi={}& \sTwist \vartheta \Psi_0 + \partial^{\leqslant 2}{}h.
\end{align}
\item Applying $\mathcal{K}^1\mathcal{K}^1$ to \eqref{linBianchi} we find
\begin{align}
\sTwist \mathcal{K}^1\mathcal{K}^2\vartheta \Psi={}& -  \sCurlDagger \mathcal{K}^0\mathcal{K}^2\vartheta \Psi + \partial^{\leqslant 1}{}h.
\end{align}
\item Using a spin decomposition of $\vartheta\Psi$, see Example~II.8 in 
\cite{2016arXiv160106084A}, and some commutators of $\mathcal{K}$ operators, we find
\begin{align}
\sCurlDagger \vartheta \Psi = \sCurlDagger  \mathcal{P}^{2} \vartheta \Psi + \sCurlDagger  \mathcal{K}^0 \mathcal{K}^2 \vartheta \Psi   + \tfrac{1}{4}\mathcal{K}^0 \mathcal{K}^1 \sTwist \vartheta \Psi_0 + \partial^{\leq 2}h,
\end{align}
which leads to
\begin{align}
\sCurlDagger  \mathcal{K}^0 \mathcal{K}^2 \vartheta \Psi = - \tfrac{1}{4}\mathcal{K}^0 \mathcal{K}^1 \sTwist \vartheta \Psi_0 + \partial^{\leq 2}h,
\end{align}
by \eqref{linBianchi} and the assumption $\vartheta\Psi_{\pm 2} = 0$ which is equivalent to $\mathcal{P}^{2} \vartheta \Psi = 0$.
\end{itemize}
Using these points, we can compute the second term,
\begin{align} \label{nablaK1K2cqPsi}
\nabla_c(\mathcal{K}^1 \mathcal{K}^2 \vartheta \Psi)_{AB}={}&\tfrac{2}{3} \epsilon_{(A|C|}(\sCurlDagger \mathcal{K}^1 \mathcal{K}^2 \vartheta \Psi)_{B)C'}
 + (\sTwist \mathcal{K}^1 \mathcal{K}^2 \vartheta \Psi)_{ABCC'} \nonumber \\
 ={}&\tfrac{2}{3} \epsilon_{(A|C|}(\sTwist \vartheta \Psi_0)_{B)C'}
  + \tfrac{1}{4}(\mathcal{K}^0 \mathcal{K}^1 \sTwist \vartheta \Psi_0)_{ABCC'} + \partial^{\leqslant 2}{}h \nonumber \\
 ={}&\tfrac{2}{3} \kappa_{1}{}^{-4}\epsilon_{(A|C|}(\mathcal{K}^1\AA)_{B)C'}
  + \tfrac{1}{4} \kappa_{1}{}^{-4}(\mathcal{K}^0 \AA)_{ABCC'} + \partial^{\leqslant 2}{}h,   
\end{align}
where \eqref{gradPsi2dot} was used in the last step.

The third term involves derivatives of $\AA^a$ which can be found in \cite{2016arXiv160106084A}, in particular we have
\begin{align} \label{nablaAAsym}
\nabla_{(a}\AA_{b)} = \partial^{\leq 1}h,
\end{align}
with real right hand side, and applying $\mathcal{K}^1$ to (58c) of \cite{2016arXiv160106084A} and commuting\footnote{Commutators with $\mathcal{K}$-operators are given in Appendix~B of \cite{2019JMP....60h2501A}.}
 operators, we find
\begin{align} \label{CurlA}
\sCurl \AA ={}&
 -  \frac{2}{3 \kappa_{1}{}}\mathcal{K}^1(\xi \overset{0,1}{\odot}\AA)
 -  \tfrac{2}{3} \kappa_{1}{}^3\xi \overset{1,1}{\odot}\sTwist \mathcal{K}^1\mathcal{K}^2\vartheta \Psi + \tfrac{4}{9} \kappa_{1}{}^3\xi \overset{0,1}{\odot}\sCurlDagger \mathcal{K}^1\mathcal{K}^2\vartheta \Psi + \partial^{\leq 2}h.
\end{align}
Using \eqref{nablaAAsym} and that $V^a$ is Killing, the third term becomes
\begin{align} \label{nablaVA}
\nabla_c(V^a \AA_a)={}&\tfrac{1}{2} V^{A}{}_{C'} (\sCurl\AA)_{CA}
 + \tfrac{1}{2} \AA^{A}{}_{C'} (\sCurl V)_{CA}
 + \tfrac{1}{2} V_{C}{}^{A'} (\sCurlDagger \AA)_{C'A'} + \tfrac{1}{2} \AA_{C}{}^{A'} (\sCurlDagger V)_{C'A'} +  \partial^{\leq 1}h.
\end{align} 
Inserting \eqref{gradPsi2dot}, \eqref{nablaK1K2cqPsi}, \eqref{nablaVA} and \eqref{CurlA} in  \eqref{GradIVStep1} leads to
\begin{align}
\nabla_c\InvSymb_V={}&- \tfrac{81}{2} \AA^{A}{}_{C'} (\sCurl V)_{CA}
 -  \tfrac{81}{2} V_{C}{}^{A'} (\sCurlDagger \AA)_{C'A'}
 -  \tfrac{81}{2} \AA_{C}{}^{A'} (\sCurlDagger V)_{C'A'} +  \partial^{\leq 2}h \nonumber\\
& + \frac{27 V^{AA'} \xi^{B}{}_{A'} (\mathcal{K}^0 \AA)_{CABC'}}{4 \kappa_{1}{}}
 + \frac{27 V^{AA'} \xi_{A}{}^{B'} (\overline{\mathcal{K}}^0 \overline{\AA})_{CC'A'B'}}{4 \bar{\kappa}_{1'}{}}\nonumber\\
& + \frac{9 V_{C}{}^{A'} \xi^{A}{}_{A'} (\mathcal{K}^1 \AA)_{AC'}}{\kappa_{1}{}}
 + \frac{9 V^{AA'} \xi_{CA'} (\mathcal{K}^1 \AA)_{AC'}}{\kappa_{1}{}}
 + \frac{9 V^{A}{}_{C'} \xi_{A}{}^{A'} (\overline{\mathcal{K}}^1 \overline{\AA})_{CA'}}{\bar{\kappa}_{1'}{}}\nonumber\\
& + \frac{9 V^{AA'} \xi_{AC'} (\overline{\mathcal{K}}^1 \overline{\AA})_{CA'}}{\bar{\kappa}_{1'}{}}
 -  \tfrac{81}{2} (\mathcal{K}^1 \AA)_{CC'} (\mathcal{K}^2 \sCurl V)_{} -  \tfrac{81}{2} (\overline{\mathcal{K}}^1 \overline{\AA})_{CC'} (\overline{\mathcal{K}}^2 \sCurlDagger V)_{}.
\end{align}
The final step consists of splitting $\AA_c = \ReA_c + i \ImA_c$, eliminating $\sCurlDagger \ReA$ using the complex conjugate of \eqref{CurlA}, and computing that all terms involving $\ReA_c$ cancel for $V^a \in \{\xi^a, \zeta^a\}$. Similarly, most terms involving $\ImA_c$ cancel so that we end up with \eqref{GradIV}.

\end{proof}

In Boyer-Lindquist coordinates $\xi=\partial_t$ and $\zeta = \bha^2 \partial_t + \bha \partial_\phi$. $\ImA_a$ is a real, gauge invariant vector field.
Since a (nice) identity, see 
\cite{2016arXiv160106084A}, dictates $\ImA_a$ to be a Killing vector, we can make an ansatz
\begin{align} \label{ImAAnsatz}
\ImA_a = A \xi_a + B \zeta_a,
\end{align}
for real constants $A,B$. 

As $\zeta_a$ can be written in terms of $\xi_a$ and the Killing spinor,
\begin{align}
\zeta_{AA'}={}&- \tfrac{9}{4} (\kappa_{1}{}^2 + \bar{\kappa}_{1'}{}^2) \xi_{AA'}
 + \tfrac{9}{2} \kappa_{AB} \bar{\kappa}_{A'B'} \xi^{BB'},
\end{align}
the ansatz \eqref{ImAAnsatz} inserted into \eqref{GradIV} can be simplified and an expansion in Boyer-Lindquist coordinates leads to $\partial_t \InvSymb_V = \partial_\phi \InvSymb_V  = 0$ and (remember that $x=\cos\theta$)
\begin{subequations}
\begin{align}
\partial_r \InvSymb_\xi={}&\frac{81 \Bigl(-2i A M + B \bha x \bigl(r^2 + 2i \bha r x -  \bha x (2i M + \bha x)\bigr)\Bigr)}{(r + i \bha x)^2}, \\
\partial_x \InvSymb_\xi={}&\frac{81 \bha \Bigl(2 A M + B \bigl(r (r + i \bha x)^2 + M (-3 r^2 - 2i \bha r x + \bha^2 x^2)\bigr)\Bigr)}{(r + i \bha x)^2}, \\
\partial_r \InvSymb_\zeta={}&\frac{81 A \bha x \bigl(r^2 + 2i \bha r x -  \bha x (2i M + \bha x)\bigr)}{(r + i \bha x)^2}\nonumber\\
& -  \frac{162 B \bha^2 \bigl(\bha^2 x^3 (\bha + i M x) + i r^3 (-1 + x^2) - i \bha^2 r x^2 (1 + x^2) + \bha r^2 (x - 2 x^3)\bigr)}{(r + i \bha x)^2}, \\
\partial_x \InvSymb_\zeta={}&\frac{81 A \bha \bigl(r (r + i \bha x)^2 + M (-3 r^2 - 2i \bha r x + \bha^2 x^2)\bigr)}{(r + i \bha x)^2}\nonumber\\
& + \frac{162 B \bha^2 \bigl(-i r^4 x + \bha^3 r x^2 + i \bha^4 x^3 + i \bha^2 r x (r - 2 M x^2 + r x^2) + \bha r^2 (r - 3 M x^2 + 2 r x^2)\bigr)}{(r + i \bha x)^2}.
\end{align}
\end{subequations}

With complex constants $C,D$, the general solution is given by
\begin{subequations}\label{eq:IIIAlgSpecSol}
\begin{align}
\InvSymb_\xi={}&C
 + \frac{81i \bigl(2 A M + B \bha x (3i M r - i r^2 -  M \bha x + \bha r x)\bigr)}{r + i \bha x}, \\
\InvSymb_\zeta={}&D
 + \frac{81 A \bha x (3i M r - i r^2 -  M \bha x + \bha r x)}{-i r + \bha x}\nonumber\\
& + \frac{81 B \bha^2 \bigl(-i \bha^3 x^3 + \bha r x^2 (\bha + 2i M x) -  r^3 (-1 + x^2) - i \bha r^2 x (1 + x^2)\bigr)}{-i r + \bha x},
\end{align}
\end{subequations}
which, by comparison with \cite{2018PhRvL.121e1104A}, can be directly identified with perturbations within the Plebanski-Demianski family of solutions.

\section{Proof of Proposition \ref{propdecayInv}}\label{prop6.8Ap}

We work on the Kerr background with Boyer-Lindquist coordinates $(x^a) = (t,r,\theta,\phi)$.  
Use index symbols $A,B,\dots$ for $t,r$, $\alpha, \beta, \dots$ for angular, and $a,b,\dots$ for general coordinates. We start by collecting relevant decay properties of the background quantities. For the metric components we have
\begin{subequations}\label{eq:gcomp} 
\begin{align} 
g_{AB} ={}& O(1), \quad g_{A\alpha} = O(r^{-1}), \quad g_{\alpha\beta} = O(r^2), \\ 
g^{AB} ={}& O(1), \quad g^{A\alpha} = O(r^{-3}), \quad g^{\alpha\beta} = O(r^{-2}).
\end{align} 
\end{subequations}
For the Christoffel symbols we have (see \cite{GM} for their explicit form):
\begin{subequations}\label{eq:Gamcomp}
\begin{align} 
\Gamma_{t\alpha}^A = O(r^{-2}), \quad \Gamma_{t\alpha}^\gamma = O(r^{-3}), \\ 
\Gamma_{r\alpha}^A = O(r^{-2}), \quad \Gamma_{r\alpha}^\beta = O(r^{-1}),\\ 
\Gamma_{t\alpha}^a = O(r^{-2}), \quad \Gamma_{r\alpha}^a = O(r^{-1})\label{A.2.c}.
\end{align} 
\end{subequations}
The remaining quantities which are needed are
\begin{subequations}\label{eq:backgr}
\begin{align} 
\varrho, \varrho' ={}& O(r^{-1}), \\ 
\tau, \tau' ={}& O(r^{-2}), \\ 
\Psi_0 ={}& O(r^{-3}), \\
\Im \Psi_0 ={}& O(r^{-4}), \\  
p ={}& O(r), \\ 
\Im p ={}& O(1).
\end{align} 
\end{subequations}
We start by proving (1). The coordinate components of the Riemann tensor are 
\begin{align} \label{eq:Riem} 
R_{abcd} = \half (g_{ad,bc} + g_{bc,ad} - g_{ac,bd} - g_{bd,ac} ) + g_{ef} (\Gamma^e_{ad} \Gamma^f_{bc} - \Gamma^e_{ac} \Gamma^f_{bd} ).
\end{align} 
Since we are considering the vacuum case, this agrees with the Weyl tensor. We will start by showing 
\begin{align} 
\Im \vartheta \Psi_0 = O(r^{-4}).
\end{align} 
As 
\begin{align} 
\Im \vartheta \Psi_0 \approx{}& r^{-2} \dot R_{tr\theta\phi}  
\end{align} 
where $\dot R_{abcd}$ is the linearized Riemann tensor, this corresponds to
\begin{align} \label{eq:dotR-2} 
\dot R_{tr\theta\phi} = O(r^{-2}).
\end{align} 
From \eqref{eq:Riem}, the $\dot R_{tr\alpha\beta}$ has two terms $I, II$. 
We have using the special form of $h$
\begin{align*} 
I =& \half ( h_{t\varphi;r\theta} +h_{r\theta;t\varphi}- h_{t\theta;r\varphi}-h_{r\varphi;t\theta})=\frac{1}{2}(h_{t\varphi;r\theta}-h_{t\theta;r\varphi})=\cO(r^{-2}). 
\end{align*} 
For the second term, we consider the linearization of 
\begin{align} 
 g_{ef} (\Gamma^e_{t\phi} \Gamma^f_{r\theta} - \Gamma^e_{t\theta} \Gamma^f_{r\phi} ).
\end{align} 
The linearization has two types of terms, first using \eqref{A.2.c} 
\begin{align} 
II_1 ={}& h_{ef}  (\Gamma^e_{t\alpha} \Gamma^f_{r\beta} - \Gamma^e_{t\beta} \Gamma^f_{r\alpha} ) \\ 
={}& O(r^{-1}) O(r^{-2}) O(r^{-1}) = O(r^{-4})
\end{align} 
and 
\begin{align*} 
II_{2A}={}& g_{ef} \dot \Gamma^e_{t\alpha} \Gamma^f_{r\beta}, \\
II_{2B}={}& g_{ef}  \Gamma^e_{t\alpha} \dot \Gamma^f_{r\beta}. 
\end{align*} 
We have 
\begin{align*} 
g_{ef} \dot \Gamma^f_{ab} = \half (h_{ae;b} + h_{be;a} - h_{ab;e} ).
\end{align*}
A calculation shows
\begin{align*} 
2 g_{ef} \dot \Gamma^f_{t\alpha} = h_{te,\alpha} - h_{t\alpha,e}=\cO(r^{-1})
\end{align*} 
and hence 
\begin{align*}
II_{2A}=\cO(r^{-2}).
\end{align*}
Similarly, 
\begin{align*} 
2 g_{ef} \dot \Gamma^f_{r\alpha} ={}& h_{re,\alpha} +h_{\alpha e,r}- h_{r\alpha,e}=\cO(r^{-1}).
\end{align*} 
and hence 
\begin{align*} 
II_{2B} = \cO(r^{-2}).
\end{align*} 
This shows that \eqref{eq:dotR-2} holds. 

Now,
\begin{align*} 
D^k h={}& O(r^{-1-k}), \quad \text{where $D = \tho, \tho', \edt, \edt'$} \\ 
\vartheta\Psi_i ={}& O(r^{-3}), \quad i=-2,-1,0,1,2.
\end{align*} 
Further,  by the above 
\begin{align} 
\vartheta \Psi_0 ={}& O(r^{-3}), \nonumber\\
\Im \vartheta \Psi_0 ={}& O(r^{-4}). \label{eq:ImP24}
\end{align} 
In addition to \eqref{eq:backgr}, in Kerr we have $\nfrak = \cfrak = 0$ and in this case 
\begin{align*} 
\Im \Psi_0 ={}& O(r^{-4}).
\end{align*} 
We now consider the expression of $\InvSymb_\zeta$ in \eqref{Invzeta}. The only term which may have \`a priori a too strong growth is 
\begin{align*} 
\bigl(
 \Psi_{0} (\bhp^2 + \overline{\bhp}^2) - 2 \bar\Psi_{0} \overline{\bhp}^2
 - 4 \bhp (\bhp_- \varrho \varrho' 
 -\bhp_+ \tau \tau')\bigr) ={}& 4r^2i\Im\Psi_0 + 4i \afrak \cos\theta r^{-1} + O(r^{-2}) \\ 
={}& 4i\afrak \cos\theta r^{-1} + O(r^{-2}).
\end{align*} 
Using \eqref{eq:ImP24} we find that the second term in  $\InvSymb_\zeta$ is $O(1)$. The first term is $O(1)$, while the third term is imaginary, and the fourth term is $O(r^{-1})$. In particular, 
\begin{align*} 
\InvSymb_\zeta = \cO(r),  
\Re \InvSymb_\zeta = \cO(1).
\end{align*}
This finishes the proof of (1). Let us now show (2). The difference is now that 
\begin{subequations}\label{eq:lins}
\begin{align*} 
D^k h =O(r^{-1-k-\eps}), \quad \text{where $D = \tho, \tho', \edt, \edt'$}, \\ 
\vartheta\Psi_i = O(r^{-3-\eps}), \quad i=-2,-1,0,1,2. 
\end{align*} 
\end{subequations}
With this information, a straightforward computation using \eqref{Invzeta} shows that for $h_{ab} = O(r^{-1-\eps})$, we have 
\begin{align*} 
\II_\zeta = O(r^{1-\eps}).
\end{align*}


\begin{thebibliography}{99}

\bibitem[Aks]{Aksteiner:private}
S. Aksteiner, private communication, 2020. 
\bibitem[AAB]{2016arXiv160106084A}
S. Aksteiner, L. Andersson, T. B\"ackdahl, \emph{New identities for linearized gravity on the Kerr spacetime}, Phys. Rev. D {\bf 99} (2019), art. no. 044043.  
\bibitem[AABKW]{2019arXiv191008756A}
S. Aksteiner, L. Andersson, T. B\"ackdahl, I. Khavkine, B. Whiting, \emph{Compatibility complex for black hole spacetimes}, Comm. Math. Phys. {\bf 384} (2021), 1585-1614. 
\bibitem[AB1]{2018PhRvL.121e1104A}
S. Aksteiner, T. B\"ackdahl, \emph{All Local Gauge Invariants for Perturbations of the Kerr Spacetime}, Phys. Rev. Lett. {\bf 121}, 051104 (2018). 
\bibitem[AB2]{2019JMP....60h2501A}
S.~Aksteiner and T.~B\"ackdahl,
\emph{Symmetries of linearized gravity from adjoint operators},
J. Math. Phys. \textbf{60} (2019) no.8, 082501
\bibitem[AB3]{Aksteiner:GI}
S. Aksteiner, T. B\"ackdahl, in preparation. 

\bibitem[An]{Ananna:2020ben}
T.~T.~Ananna, C.~M.~Urry, E.~Treister, R.~C.~Hickox, F.~Shankar, C.~Ricci, N.~Cappelluti, S.~Marchesi and T.~J.~Turner,
\emph{Accretion History of AGNs. III. Radiative Efficiency and AGN Contribution to Reionization}, 
Astrophys. J. \textbf{903} (2020) no.2, 85. 
\bibitem[ABBM]{ABBM}
L. Andersson, T. B\"ackdahl, P. Blue, S. Ma, \emph{Stability for linearized gravity on the Kerr spacetime}, arXiv:1903.03859.
\bibitem[AMPW]{AMPW}
L. Andersson, S. Ma, C. Paganini, B. Whiting, \emph{Mode stability on the real axis}, J. Mat. Phys. {\bf 58} (2017), 072501, 19 pp. 
\bibitem[BVK]{BVK}
T.~B\"ackdahl and J.~A.~Valiente~Kroon,
\emph{A formalism for the calculus of variations with spinors,}'
J. Math. Phys. \textbf{57} (2016) no.2, 022502
\bibitem[BH]{BH}  
N. Besset, D. H\"afner, \emph{Existence of exponentially growing finite energy solutions for the charged Klein-Gordon equation on the De Sitter-Kerr-Newman metric},  J. Hyp. Diff. Equ. {\bf 18} (2021),  293-310.
\bibitem[BoHa]{BoHa}
J.-F. Bony, D. H\"afner, \emph{Decay and non-decay of the local energy for the wave equation in the De Sitter - Schwarzschild metric},
Comm. Math. Phys. {\bf 282} (2008), 697-719. 
\bibitem[CTC]{CTC} 
 M. Casals, R. Teixeira da Costa, \emph{Hidden spectral symmetries and mode stability of subextremal Kerr(-dS) black holes}, arXiv:2105.13329.
\bibitem[Ch]{Ch}
S. Chandrasekhar, \emph{The Mathematical Theory of Black Holes}, International Series of Monographs on Physics {\bf 69}, Oxford University Press 1983.  
\bibitem[DHR1]{DHR1}
M. Dafermos, G. Holzegel, I. Rodnianski, \emph{Boundedness and decay for the Teukolsky equation on Kerr spacetimes I: the case $\vert a\vert<<M$}, Annals of PDE (2019) {\bf 5}, No. 1, paper No. 2, 118p.  
\bibitem[DHR2]{DHR2}
M. Dafermos, G. Holzegel, I. Rodnianski, \emph{
The linear stability of the Schwarzschild solution to gravitational perturbations}, 
Acta Math. {\bf 222} (2019), 1-214. 
\bibitem[DHRT]{DHRT}
 M. Dafermos, G. Holzegel, I. Rodnianski, M. Taylor, \emph{The non-linear stability of the Schwarzschild family of black holes}, arXiv:2104.08222. 
\bibitem[Dy1]{Dy}
S. Dyatlov, \emph{Asymptotics of linear waves and resonances with applications to black holes}, Comm. Math. Phys. {\bf 335} (2015), 1445-1485.  
\bibitem[Dy2]{DyatlovSpectralGaps}
S. Dyatlov.
\newblock Spectral gaps for normally hyperbolic trapping.
\newblock {\em Ann. Inst. Fourier (Grenoble)}, 66(1):55--82, 2016.
\bibitem[FS1]{FS1}
F. Finster, J. Smoller, \emph{A spectral representation for spin-weighted spheroidal wave operators with complex aspherical parameter}, 
Methods Appl. Anal. {\bf 23} (2016), 35-118. 
\bibitem[FS2]{FS}
F. Finster, J. Smoller, \emph{Linear stability of the non-extreme Kerr black hole}, 
Adv. Theor. Math. Phys. {\bf 21} (2017), 1991-2085. 
\bibitem[Ga]{Ga}
O. Gannot, \emph{The null-geodesic flow near horizons},
Trans. Am. Math. Soc. {\bf 371} (2019), 4769-4791. 
\bibitem[GHP]{GHP}
R. Geroch, A. Held, R. Penrose, \emph{A space-time calculus based on pairs of null directions}, J. Math. Phys., {\bf 14} (1973), 874-881.  
\bibitem[Gi]{Gi}
E. Giorgi, \emph{The linear stability of {R}eissner--{N}ordstr{\"o}m spacetime for
  small charge}, Commun. Math. Phys. {\bf 380} (2020), 1313?1360.
  \bibitem[GKS]{GKS}
  E. Giorgi, S. Klainerman, J. Szeftel, \emph{Wave equations estimates and the nonlinear stability of slowly rotating Kerr black holes}, arXiv:2205.14808.
  \bibitem[GMNRS]{GMNRS}
  J.N. Goldberg, A. J. Macfariane, E.T. Newman, F. Rohlich and  E.C.G. Sudarshan, \emph{Spin-s Spherical Harmonics and $\text{\dh}$}, 
  J. Math. Phys. {\bf 8} (1967), 2155.
 \bibitem[GM]{GM}
  F. Grave, T. M\"uller, \emph{Catalogue of spacetimes}, arXiv : 0904.4184.
\bibitem[HHV]{HHV}
D. H\"afner, P. Hintz, A. Vasy, \emph{Linear stability of slowly rotating Kerr black holes}, Invent. Math. {\bf 223} (2021), 1227-1406.
\bibitem[Har]{Har}
G. Harnett, \emph{The ghp connection : a metric connection with torsion determined by a pair of null directions}, Class. Quant. Grav. {\bf 7} (1990), p. 1681.
\bibitem[Hi1]{HintzPsdoInner}
P. Hintz,
\newblock Resonance expansions for tensor-valued waves on asymptotically
  {K}err--de {S}itter spaces.
\newblock {\em J. Spectr. Theory}, 7:519--557, 2017.
\bibitem[Hi2]{Hi2}
P. Hintz, \emph{Mode stability and shallow quasinormal modes of Kerr-de Sitter black holes away from extremality}, arXiv:2112.14431. 

\bibitem[HV1]{HintzVasySemilinear}
P. Hintz and A. Vasy.
\newblock Semilinear wave equations on asymptotically de {S}itter, {K}err--de
  {S}itter and {M}inkowski spacetimes.
\newblock {\em Anal. PDE}, 8(8):1807--1890, 2015.
\bibitem[HV2]{HintzVasyQuasilinearKdS}
P. Hintz and A. Vasy.
\newblock {G}lobal {A}nalysis of {Q}uasilinear {W}ave {E}quations on
  {A}symptotically {K}err--de {S}itter {S}paces.
\newblock {\em International Mathematics Research Notices},
  2016(17):5355--5426, 2016.
\bibitem[HV3]{HV}
P. Hintz and A. Vasy, \emph{
The global non-linear stability of the Kerr-de Sitter family of black holes}, 
Acta Math. {\bf 220} (2018), 1-206. 
\bibitem[Ho]{Ho} L. H\"ormander, \emph{The analysis of linear partial differential operators. III.} Classics in Mathematics.
Springer, Berlin, 2007.
\bibitem[HKW]{HKW}
P.-K. Hung, J. Keller, and M.-T. Wang,\emph{Linear stability of {S}chwarzschild spacetime: decay of metric
  coefficients},  J. Differ. Geom. {\bf 116} (2020), 481-541.
\bibitem[IoKl]{IoKl}
A.~D. Ionescu, S. Klainerman, \emph{On the global stability of the wave-map equation in {K}err spaces with small angular momentum}, Ann. PDE {\bf 1} (2015) 1-78.
\bibitem[Kh]{Kh}
I. Khavkine, \emph{Compatibility complexes of overdetermined PDEs of finite type, with applications to the Killing equation},
      Classical and Quantum Gravity {\bf 36} (2018), 185012.
\bibitem[Ki]{Ki}
W. Kinnersley, \emph{Type D vacuum metrics}, J. Math. Phys. {\bf 10} (1969) 1195-1203.
\bibitem[KS1]{KS1}
  S. Klainerman, J. Szeftel, \emph{Global nonlinear stability of Schwarzschild spacetime under polarized perturbations},
Annals of Mathematics Studies {\bf 210} (2020). Princeton, NJ: Princeton University Press  xv, 840 p.   
 \bibitem[KS2]{KS}
 S. Klainerman, J. Szeftel, \emph{Kerr stability for small angular momentum}, arXiv:2104.11857.
\bibitem[Kr]{Kr}
G. Kristensson, \emph{Second order differential equations}, Springer 2010. 
\bibitem[LPV]{LPV}
O. Lindblad Petersen, A. Vasy, \emph{Wave equations in the Kerr-de Sitter spacetime: the full subextremal range}, arXiv:2112.01355.
\bibitem[Ma]{Ma}
S. Ma, \emph{Analysis of Teukolsky equations
on slowly rotating Kerr spacetimes}, PhD thesis Universit\" at Potsdam 2018. 
\bibitem[MZ]{MZ}
  S. Ma, L. Zhang, \emph{Sharp decay for Teukolsky equation in Kerr spacetimes}, arXiv:2111.04489. 
\bibitem[Me1]{MelroseEuclideanSpectralTheory}
R. Melrose.
\newblock Spectral and scattering theory for the {L}aplacian on asymptotically
  {E}uclidian spaces.
\newblock In {\em Spectral and scattering theory ({S}anda, 1992)}, volume 161
  of {\em Lecture Notes in Pure and Appl. Math.}, pages 85--130. Dekker, New
  York, 1994.
  \bibitem[Me2]{MelroseAPS}
R. Melrose.
\newblock {\em The {A}tiyah--{P}atodi--{S}inger index theorem}, volume~4 of
  {\em Research Notes in Mathematics}.
\newblock A K Peters, Ltd., Wellesley, MA, 1993.
\bibitem[M1]{M}
P. Millet, \emph{Geometric background for Teukolsky equation revisited}, arXiv:2111.03347. 
\bibitem[M2]{M2}
P. Millet, \emph{The decay of the solutions of the Teukolsky equation on Kerr spacetime}, in preparation. 
\bibitem[NP1]{newman1962approach}
Ezra Newman and Roger Penrose.
\newblock An approach to gravitational radiation by a method of spin
  coefficients.
\newblock {\em Journal of Mathematical Physics}, 3(3):566--578, 1962.  See also\cite{newman1963errata}.
\bibitem[NPE]{newman1963errata}
Ezra Newman and Roger Penrose.
\newblock Errata: an approach to gravitational radiation by a method of spin
  coefficients.
\newblock {\em Journal of Mathematical Physics}, 4(7):998--998, 1963.
\bibitem[PD]{PD}
J.F. Plebanski, M. Demianski, \emph{Rotating, charged, and uniformly accelerating mass in general relativity}, Annals of Physics {\bf 98} (1976), 98-127.
\bibitem[PR1]{PeRi1}
R. Penrose, W. Rindler, \emph{Spinors and Space-Time vol. 1: Spinor and Twistor Methods in Space-Time Geometry}, Cambridge University Press 1984.
\bibitem[PR2]{PeRi}
R. Penrose, W. Rindler, \emph{Spinors and Space-Time vol. 2: Spinor and Twistor Methods in Space-Time Geometry}, Cambridge University Press 1986.
\bibitem[SR]{SR}
Y. Shlapentokh-Rothman, \emph{
Exponentially growing finite energy solutions for the Klein-Gordon equation on sub-extremal Kerr spacetimes}, 
Comm. Math. Phys. {\bf 329} (2014), 859-891.
\bibitem[SRTC]{SRTC}
 Y. Shlapentokh-Rothman, R. Teixeira da Costa, \emph{Boundedness and decay for the Teukolsky equation on Kerr in the full subextremal range $\vert a\vert<M$: frequency space analysis}, arXiv : arXiv:2007.07211.  
\bibitem[Te]{Te}
S. A. Teukolsky, \emph{Rotating Black Holes: Separable Wave Equations for Gravitational
and Electromagnetic Perturbations}, Phys. Rev. Lett. {\bf 29} (1972), 1114-1118.
\bibitem[Th]{Thorne:1974ve}
K.~S.~Thorne,
\emph{Disk accretion onto a black hole. 2. Evolution of the hole}, 
Astrophys. J. \textbf{191} (1974), 507-520.
\bibitem[Vai]{Vai}
B. Vainberg, \emph{Asymptotic methods in equations of mathematical physics}, Gordon and Breach Science Publishers, 1989.
\bibitem[Va1]{VasyMicroKerrdS}
A. Vasy.
\newblock Microlocal analysis of asymptotically hyperbolic and {K}err--de
  {S}itter spaces (with an appendix by {S}emyon {D}yatlov).
\newblock {\em Invent. Math.}, 194(2):381--513, 2013.
\bibitem[Va2]{VasyLowEnergy}
A. Vasy.
\newblock {R}esolvent near zero energy on {R}iemannian scattering
  (asymptotically conic) spaces, Pure and applied mathematics {\bf 3} (2021), 1-74. 
\bibitem[Va3]{Va1}
A. Vasy, \emph{Limiting absorption principle on Riemannian scattering (asymptotically conic) spaces, a Lagrangian approach}, CPDE {\bf 46} (2021), 780-822.
\bibitem[Va4]{Va2}
A. Vasy, \emph{Resolvent near zero energy on Riemannian scattering (asymptotically conic) spaces, a Lagrangian approach}, CPDE {\bf 46} (2021),  823-863. 
\bibitem[VZ]{VasyZworskiScl}
A. Vasy and M. Zworski.
\newblock {S}emiclassical {E}stimates in {A}symptotically {E}uclidean
  {S}cattering.
\newblock {\em Communications in Mathematical Physics}, 212(1):205--217, 2000.

\bibitem[Wa]{Wa} 
R. Wald, \emph{On perturbations of a Kerr black hole}, J. Math. Phys. {\bf 14} (1973), 1453-1461. 
\bibitem[Wh]{Wh}
B. Whiting, \emph{Mode stability of the Kerr black hole}, Journal of Mathematical Physics {\bf 30} (1989), 1301-1305. 
\bibitem[WZ]{WZ}
J. Wunsch and M. Zworski, \emph{
Resolvent estimates for normally hyperbolic trapped sets},
Ann. Henri Poincar\'e {\bf 12} (2011), 1349-1385. 
\bibitem[Zw]{Zw}
M. Zworski, \emph{Resonances for asymptotically hyperbolic manifolds: {V}asy's method
  revisited}, J. Spectr. Theory {\bf 6} (2016),1087--1114.


\end{thebibliography}
\end{document}